\documentclass[12pt]{amsart}
\usepackage{amsmath,amssymb,amsthm,url}
\usepackage{graphicx}
\usepackage{fullpage}
\usepackage[all]{xy}
\usepackage{comment}

\numberwithin{equation}{section}

\newtheorem{thm}[equation]{Theorem}
\newtheorem*{thm*}{Theorem}
\newtheorem{prop}[equation]{Proposition}

\newtheorem{conj}[equation]{Conjecture}
\newtheorem*{conj*}{Conjecture}
\newtheorem{lem}[equation]{Lemma}
\newtheorem{defn}[equation]{Definition}

\theoremstyle{remark}
\newtheorem{rmk}[equation]{Remark}
\newtheorem{exm}[equation]{Example}

\newenvironment{enumroman}
{\begin{enumerate}}
{\end{enumerate}}

\newcommand{\slsh}[1]{\,|_{#1}\,}

\newcommand{\field}[1]{\mathbb{#1}} 
\newcommand{\Q}{\field{Q}}
\newcommand{\C}{\field{C}}
\newcommand{\R}{\field{R}}

\newcommand{\HH}{\field{H}}
\newcommand{\N}{\field{N}}

\newcommand{\Z}{\field{Z}}

\newcommand{\F}{\field{F}}

\newcommand{\T}{\field{T}}

\newcommand{\PP}{\field{P}}

\newcommand{\quat}[2]{\displaystyle{\biggl(\frac{#1}{#2}\biggr)}}

\newcommand{\calD}{\mathcal{D}}
\newcommand{\calH}{\mathcal{H}}
\newcommand{\calM}{\mathcal{M}}
\newcommand{\calO}{\mathcal{O}}

\newcommand{\fraka}{\mathfrak{a}}

\newcommand{\frakb}{\mathfrak{b}}

\newcommand{\frakd}{\mathfrak{d}}
\newcommand{\frakD}{\mathfrak{D}}
\newcommand{\frakl}{\mathfrak{l}}
\newcommand{\frakp}{\mathfrak{p}}

\newcommand{\frakn}{\mathfrak{n}}
\newcommand{\frakM}{\mathfrak{M}}
\newcommand{\frakN}{\mathfrak{N}}

\newcommand{\frakP}{\mathfrak{P}}

\DeclareMathOperator{\img}{img}
\DeclareMathOperator{\Aut}{Aut}
\DeclareMathOperator{\Tr}{Tr}
\DeclareMathOperator{\Cl}{Cl}
\DeclareMathOperator{\Coind}{Coind}
\DeclareMathOperator{\disc}{disc}
\DeclareMathOperator{\End}{End}
\DeclareMathOperator{\Frob}{Frob}
\DeclareMathOperator{\Gal}{Gal}
\DeclareMathOperator{\Hom}{Hom}
\DeclareMathOperator{\Jac}{Jac}
\DeclareMathOperator{\GL}{GL}
\DeclareMathOperator{\SL}{SL}
\DeclareMathOperator{\Map}{Map}
\DeclareMathOperator{\M}{M}
\DeclareMathOperator{\PSL}{PSL}
\DeclareMathOperator{\PGL}{PGL}

\DeclareMathOperator{\repart}{Re}

\DeclareMathOperator{\nrd}{nrd} 
\DeclareMathOperator{\sgn}{sgn} 
\DeclareMathOperator{\red}{red}
\DeclareMathOperator{\SO}{SO}

\DeclareMathOperator{\Lat}{Lat}

\newcommand{\spnew}{{}^{\text{new}}}

\newcommand{\Qbar}{\overline{\Q}}

\newcommand{\ahat}{\widehat{a}}
\newcommand{\alphahat}{\widehat{\alpha}}
\newcommand{\bhat}{\widehat{b}}
\newcommand{\betahat}{\widehat{\beta}}
\newcommand{\Bhat}{\widehat{B}}
\newcommand{\Fhat}{\widehat{F}}
\newcommand{\PGamma}{\mathrm{P}\Gamma}
\newcommand{\calOhat}{\widehat{\mathcal O}}
\newcommand{\gammahat}{\widehat{\gamma}}

\newcommand{\uhat}{\widehat{u}}

\newcommand{\pihat}{\widehat{\pi}}
\newcommand{\phat}{\widehat{p}}

\newcommand{\Qhat}{\widehat{\mathbb Q}}
\newcommand{\Zhat}{\widehat{\mathbb Z}}
\newcommand{\ZFhat}{\widehat{\mathbb Z}_F}

\newcommand{\gp}{\mathfrak{p}}
\newcommand{\gP}{\mathfrak{P}}
\newcommand{\ga}{\mathfrak{a}}
\newcommand{\gb}{\mathfrak{b}}

\newcommand{\CO}{\mathcal{O}}

\newcommand{\prodprime}{\sideset{}{^{'}}{\prod}}

\begin{document}

\title{Explicit methods for Hilbert modular forms}
\author{Lassina Demb\'el\'e}
\address{Warwick Mathematics Institute, University of Warwick, Coventry CV4 7AL, UK}
\email{lassina.dembele@gmail.com}
\author{John Voight}
\address{Department of Mathematics and Statistics, University of Vermont, 16 Col\-chester Ave, Burlington, VT 05401, USA}
\email{jvoight@gmail.com}
\date{\today}

\begin{abstract}
We exhibit algorithms to compute systems of Hecke eigenvalues for spaces of Hilbert modular forms over a totally real field.  We provide many explicit examples as well as applications to modularity and Galois representations.
\end{abstract}

\maketitle
\tableofcontents

The study of modular forms remains a dominant theme in modern number theory, a consequence of their intrinsic appeal as well as their applications to a wide variety of mathematical problems.  This subject has seen dramatic progress during the past half-century in an environment where both abstract theory and explicit computation have developed in parallel.  Experiments will remain an essential tool in the years ahead, especially as we turn from classical contexts to less familiar terrain.

In this article, we discuss methods for explicitly computing spaces of Hilbert modular forms, refashioning algorithms over $\Q$ to the setting of totally real fields.  Saving definitions for the sections that follow, we state our main result.

\begin{thm*}
There exists an algorithm that, given a totally real field $F$, a nonzero ideal $\frakN$ of the ring of integers of $F$, and a weight $k \in (\Z_{\geq 2})^{[F:\Q]}$, computes the space $S_k(\frakN)$ of Hilbert cusp forms of weight $k$ and level $\frakN$ over $F$ as a Hecke module.
\end{thm*}

This theorem is the work of the first author \cite{Dembeleclassno1} together with Donnelly \cite{DembeleDonnelly} combined with work of the second author \cite{Voightclno} together with Greenberg \cite{GreenbergVoight}.

The outline of this article is as follows.  After briefly recalling methods for classical (elliptic) modular forms in \S 1, we introduce our results for computing Hilbert modular forms in the simplest case (of parallel weight $2$ over a totally real field of strict class number $1$) in \S 2.  In brief, our methods employ the Jacquet-Langlands correspondence to relate spaces of Hilbert modular forms to spaces of quaternionic modular forms that are more amenable to computation; we discuss this matter in \S 3, and consider two approaches (definite and indefinite) in \S\S 4--5.  In \S 6 we consider several examples in detail.  Having thereby established the main ideas, we turn to an adelic description of Hilbert modular forms and their quaternionic incarnations in \S 7, then give a complete and general description of our algorithms in adelic language in \S\S 8--9.  

Although it is our intention to keep these notes as self-contained as possible, we will assume that the reader has a basic familiarity with classical modular forms and the methods employed to compute them.  The algorithms exhibited below have been implemented in the computer algebra system \textsf{Magma} \cite{magma1} and our examples are computed using this implementation.  Donnelly and the second author \cite{DonnellyVoight} are using this implementation to compute Hecke data for thousands of forms over totally real fields up to degree $6$.

These notes arose from lectures at the Centre de Recerca Matem\`atica (CRM) in Barcelona; it is our great pleasure to thank the CRM for the invitation to speak and the hospitality of the organizers, Luis Dieulefait and Victor Rotger.  The authors would also like to thank Matthew Greenberg, Ariel Pacetti, Aurel Page, Jeroen Sijsling, and the referee for many helpful comments as well as Benedict Gross for his remarks which we include at the end of Section 6.  The first author is supported by a Marie-Curie Fellowship, and the second author  by an NSF Grant No.\ DMS-0901971.

\section{Classical (elliptic) modular forms}

To motivate our study of Hilbert modular forms, we begin by briefly considering algorithms for classical (elliptic) modular forms.  For a more detailed introduction to modular forms, see the books by Darmon \cite{Darmon} and Diamond and Shurman \cite{DiamondShurman}, and for more information on computational aspects see Cremona \cite{cremonatables}, Kilford \cite{Kilford}, Stein \cite{Stein}, and the many references contained therein.

Let $\calH=\{x+yi \in \C : y>0\}$ denote the upper-half plane and let $\calH^*=\calH \cup \PP^1(\Q)$ denote the completed upper half-plane with the \emph{cusps} $\PP^1(\Q)$.  The group 
\[ \GL_2^+(\Q)=\{\gamma \in \GL_2(\Q) : \det \gamma > 0\} \] 
acts on $\calH^*$ by linear fractional transformations.  For $N \in \Z_{>0}$, we consider the subgroup of those integral matrices of determinant $1$ that are upper-triangular modulo $N$,
\[ \Gamma_0(N)= \left\{\gamma = \textstyle{\begin{pmatrix} a & b \\ c & d \end{pmatrix}} \in \SL_2(\Z) : N \mid c\right\} \subseteq  \GL_2^+(\Z) = \SL_2(\Z) \subseteq  \GL_2^+(\Q).   \]
The group $\PGamma_0(N)=\Gamma_0(N)/\{\pm 1\}$ is a discrete subgroup of $\PSL_2(\R)$.  A \emph{modular form} of weight $k \in \Z_{>0}$ and level $N$ is a holomorphic function $f:\calH \to \C$ such that 
\begin{equation} \label{fgammaz}
f(\gamma z)   =f\left(\frac{az+b}{cz+d}\right)   = (cz+d)^k f(z)
\end{equation}
for all $\gamma \in \Gamma_0(N)$ and such that $f(z)$ tends to a finite limit as $z$ tends to any cusp (i.e., $f$ is holomorphic at the cusps).  

One can equivalently write this as follows.  For $\gamma \in \GL_2(\R)$ and $z \in \calH$ we define $j(\gamma,z)=cz+d$.  We then define a \emph{weight $k$ action} of $\GL_2^+(\Q)$ on the space of complex-valued functions on $\calH$ by
\begin{equation} \label{fslsh}
(f\slsh{k} \gamma)(z)=\frac{(\det\gamma)^{k-1}}{j(\gamma,z)^k} f(\gamma z).
\end{equation}
Then (\ref{fgammaz}) is equivalent to $f\slsh{k} \gamma = f$ for all $\gamma\in\Gamma_0(N)$. 

Note that the determinant factor $(\det\gamma)^{k-1}$ in our definition is different from the usual $(\det\gamma)^{k/2}$, which is an analytic normalization.  Consequently, the central subgroup $\Q^\times \subseteq \GL_2^+(\Q)$ acts by $f \slsh{k} \gamma = \gamma^{k-2} f$.

\begin{rmk}
For simplicity we treat only the case of $\Gamma_0(N)$-level structure in this article.  If desired, one could without difficulty extend our methods to more general level structures with characters, and so on.
\end{rmk}

The $\C$-vector space of modular forms of weight $k$ and level $N$ is finite-dimensional and is denoted $M_k(N)$.  

If $f \in M_k(N)$, then $f(z+1)=f(z)$ so $f$ has a Fourier expansion 
\begin{equation} \label{qexpclassical}
f(z)=\sum_{n=0}^{\infty} a_n q^n = a_0 + a_1 q + a_2 q^2 + a_3 q^3 + \dots
\end{equation}
where $a_n \in \C$ and $q=\exp(2\pi iz)$.  

We say that $f$ is a \emph{cusp form} if $f(z) \to 0$ as $z$ tends to any cusp (i.e., $f$ vanishes at the cusps).  The $\C$-vector space of cusp forms of weight $k$ and level $N$ is denoted $S_k(N)$.  We have $M_k(N)=S_k(N) \oplus E_k(N)$ where $E_k(N)$ is spanned by the Eisenstein series of level $N$.  Note that when $k\geq 2$, then (\ref{fgammaz}) is equivalent to
\[ f(\gamma z)\, (d(\gamma z))^{k-1} = f(z)\,(dz)^{k-1} \]
so one may equivalently think of a cusp form $f \in S_k(N)$ as a holomorphic differential $(k-1)$-form on the \emph{modular curve} $X_0(N)=\Gamma_0(N) \backslash \calH^*$.  (Because of our normalization, such differential forms will be global sections of the line bundle corresponding to the algebraic local system $\mathrm{Sym}^{k-2}(\C^2)$, corresponding to bivariate homogeneous polynomials of degree $k-2$.  Some authors use an analytic normalization instead.)

The spaces $M_k(N)$ and $S_k(N)$ are equipped with an action of pairwise commuting diagonalizable \emph{Hecke operators} $T_n$ for each integer $n \in \Z_{>0}$.  The Hecke operators can be thought of in several different ways: they arise from correspondences on the modular curve $X_0(N)$, as ``averaging'' operators over lattices of index $n$, or more formally from double coset decompositions for the group $\Gamma_0(N)$ inside $\SL_2(\Z)$.  The action of the Hecke operator $T_n$ is determined by the action of $T_p$ for $p \mid n$, and the latter for $p \nmid N$ in weight $k$ are given simply by the formula
\[ (T_p f)(z) = p^{k-1} f(pz) + \frac{1}{p} \sum_{a=0}^{p-1} f\left(\frac{z+a}{p}\right). \]
(For primes $p \mid N$ one omits the first term, and there are also operators called \emph{Atkin-Lehner involutions}.)  We say therefore that $S_k(N)$ is a \emph{Hecke module}, namely, an abelian group equipped with an action of the \emph{Hecke algebra} $\widetilde{\T}=\Z[T_p]_p=\Z[T_2,T_3,\dots]$, a polynomial ring in countably many variables over $\Z$ indexed by the primes.  Our Hecke modules will always be finite-dimensional $\C$-vector spaces.  

A form $f \in S_k(N)$ is an \emph{oldform (at $d$)} if $f(z)=g(dz)$ for some $g \in S_k(M)$ with $M \mid N$ a proper divisor and $d \mid N/M$; we say $f$ is a \emph{newform} if $f$ is a normalized eigenform which is orthogonal to the space of oldforms (with respect to the Petersson inner product).  

The space $S_k(N)$ consequently has a basis of \emph{eigenforms}, i.e., functions that are eigenfunctions for each Hecke operator $T_n$.  If $f$ is an eigenform, \emph{normalized} so that $a_1=1$ in its $q$-expansion (\ref{qexpclassical}), then $T_n f = a_n f$.  Moreover, the field $\Q(\{a_n\})=E \subseteq  \C$ is a number field and each \emph{Hecke eigenvalue} $a_n$ is an algebraic integer in $E$.  

In this way, the system of Hecke eigenvalues $(a_p)_p$ for a normalized eigenform $f \in S_k(N)$ determine the form $f:\calH \to \C$.    These eigenvalues also determine the $L$-series 
\[ L(f,s) = \sum_{n=1}^{\infty} \frac{a_n}{n^s} = \prod_{p \nmid N} \left(1-\frac{a_p}{p^s}+\frac{1}{p^{2s+1-k}}\right)^{-1} \prod_{p \mid N} \left(1-\frac{a_p}{p^s}\right)^{-1} \]
associated to $f$ (defined for $\repart s>1$), as well as the $\ell$-adic Galois representations 
\[ \rho_{f,\ell}:\Gal(\Qbar/\Q) \to \GL_2(\overline{\Z}_\ell) \] 
associated to $f$ with the property that for any prime $p \nmid \ell N$, we have
\[ \Tr(\rho_{f,\ell}(\Frob_p))=a_p(f) \quad \text{ and } \quad \det(\rho_{f,\ell}(\Frob_p))=p^{k-1}. \]

Several methods have been proposed for making the Hecke module $S_k(N)$ into an object amenable to explicit computation.  With a view to their generalizations to Hilbert modular forms, we mention two approaches which have seen wide application.  (We neglect the method of graphs \cite{mestre} and a method which uses the Eichler-Selberg trace formula \cite{Hijikata}.)  For simplicity, we restrict our discussion to the case of weight $k=2$.

The first method goes by the name \emph{modular symbols} and has been developed by Birch, Swinnerton-Dyer, Manin, Mazur, Merel, Cremona \cite{cremonatables}, Stein \cite{Stein}, and many others.  The Hecke operators $T_p$ act naturally on the integral homology $H_1(X_0(N),\Z;\text{cusps})$---linear combinations of paths in the completed upper half plane $\calH^*$ whose endpoints are cusps and whose images in $X_0(N)$ are linear combinations of loops---and integration defines a nondegenerate Hecke-equivariant pairing which gives rise to an isomorphism (the Eichler-Shimura theorem)
\[ H_1(X_0(N),\C;\text{cusps}) \cong S_2(N) \oplus \overline{S_2(N)} \]
where $\overline{\phantom{x}}$ denotes complex conjugation.  The formalism of modular symbols then presents the space $H_1(X_0(N),\Z;\text{cusps})$ explicitly in terms of paths in $\calH^*$ whose endpoints are cusps (elements of $\PP^1(\Q)$) and whose images in $X_0(N)$ are a linear combination of loops.  We have an explicit description of the action of the Hecke operators on the space of modular symbols, and the \emph{Manin trick} (the Euclidean algorithm) yields an algorithm for writing an arbitrary modular symbol as a $\Z$-linear combination of a finite set of generating symbols, thereby recovering $S_2(N)$ as a Hecke module. 

The second method goes by the name \emph{Brandt matrices} and goes back to Brandt, Eichler \cite{Eichlerbasis,Eichlerbasis1}, Pizer \cite{Pizer}, Kohel \cite{kohel}, and others.  In this approach, a basis for $S_2(N)$ is obtained by linear combinations of theta series associated to (right) ideals in a quaternion order of discriminant $N$.  These theta series are generating series which encode the number of elements in the ideal with a given reduced norm, and the Brandt matrices which represent the action of the Hecke operators are obtained via this combinatorial (counting) data.  

\section{Classical Hilbert modular forms}

We now consider the situation where the classical modular forms from the previous section are replaced by forms over a totally real field.  References for Hilbert modular forms include Freitag \cite{Freitag}, van der Geer \cite{geer} and Goren \cite{Goren}.

Let $F$ be a totally real field with $[F:\Q]=n$ and let $\Z_F$ be its ring of integers.  The case $n=1$ gives $F=\Q$ and this was treated in the previous section, so we assume throughout this section that $n>1$.  Let $v_1,\dots,v_n:F \to \R$ be the real places of $F$,   and write $v_i(x)=x_i$.  For $\gamma \in \M_2(F)$ we write $\gamma_i=v_i(\gamma) \in \M_2(\R)$.

For simplicity, in these first few sections (\S\S 2--6) we assume that $F$ has strict class number 1; the general case, which is more technical, is treated in \S 7 and thereafter.

The group 
\[ \GL_2^+(F)=\{\gamma \in \GL_2(F) : \det \gamma_i > 0 \text{ for $i=1,\dots,n$}\} \] 
acts naturally on $\calH^n$ by coordinatewise linear fractional transformations
\[ z \mapsto \gamma z = (\gamma_i z_i)_i = \left(\frac{a_i z_i+b_i}{c_i z_i + d_i}\right)_{i=1,\dots,n}. \]  
For a nonzero ideal $\frakN \subseteq  \Z_F$,   let
\[ \Gamma_0(\frakN)= \left\{ \gamma = \begin{pmatrix} a & b \\ c & d \end{pmatrix} \in \GL_2^+(\Z_F) : c \in \frakN\right\} \subseteq  \GL_2^+(\Z_F) \subseteq  \GL_2(F). \]

Let $\PGamma_0(\frakN) = \Gamma_0(\frakN)/\Z_F^\times \subseteq  \PGL_2^+(\Z_F)$.  Then the image of $\PGamma_0(\frakN)$ under the embeddings $\gamma \mapsto (\gamma_i)_i$ is a discrete subgroup of $\PGL_2^+(\R)^n$.  

Under the assumption that $F$ has strict class number $1$, we have 
\[ \Z_{F,+}^\times = \{x \in \Z_F^\times : x_i>0\text{ for all $i$}\} = \Z_F^{\times 2} \]
and hence $\GL_2^+(\Z_F) = \Z_F^\times \SL_2(\Z_F)$, and so alternatively we may identify 
\[ \PGamma_0(\frakN) \cong \left\{ \gamma = \begin{pmatrix} a & b \\ c & d \end{pmatrix} \in \SL_2(\Z_F) : c \in \frakN\right\}/\{\pm 1\} \]
in analogy with the case $F=\Q$.

\begin{defn} \label{Hilbclass}
A \emph{Hilbert modular form} of parallel weight $2$ and level $\frakN$ is a holomorphic function $f:\calH^{n}\to\C$   such that 
\begin{equation} \label{slashfz}
f(\gamma z) = f\left(\frac{a_1 z_1 + b_1}{c_1 z_1 + d_1}, \dots, \frac{a_n z_n + b_n}{c_n z_n + d_n}\right) = \left(\prod_{i=1}^n \frac{(c_i z_i+d_i)^2}{\det \gamma_i} \right) f(z)
\end{equation}
for all $\gamma\in\Gamma_0(\frakN)$.
\end{defn}

We denote by $M_2(\frakN)$ the space of Hilbert modular forms of parallel weight $2$ and level $\frakN$; it is a finite-dimensional $\C$-vector space.  The reader is warned not to confuse $M_2(\frakN)$ with the ring $\M_2(R)$ of $2 \times 2$-matrices over a ring $R$.

\begin{rmk}
There is no holomorphy condition at the cusps in Definition \ref{Hilbclass} as there was for classical modular forms.  Indeed, under our assumption that $[F:\Q]=n>1$, this follows automatically from Koecher's principle~\cite[\S 1]{geer}. 

Note also that if $u \in \Z_F^\times$, then $\gamma(u)=\begin{pmatrix} u & 0 \\ 0 & u \end{pmatrix} \in \GL_2^+(\Z_F)$ acts trivially on $\calH$ and at the same time gives a vacuous condition in (\ref{slashfz}), explaining the appearance of the determinant term which was missing in the classical case.
\end{rmk}

Analogous to (\ref{fslsh}), we define
\begin{equation} \label{fslshhilb}
(f \slsh{} \gamma)(z) = \left(\prod_{i=1}^n \frac{\det \gamma_i}{j(\gamma_i, z)^2}\right)f(\gamma z) 
\end{equation}
for $f:\calH^n \to \C$ and $\gamma \in \GL_2^+(F)$; then (\ref{slashfzquat}) is equivalent to $(f \slsh{} \gamma)(z) = f(z)$ for all $\gamma \in \Gamma_0(\frakN)$.  

The group $\GL_2^+(F)$ also acts naturally on the \emph{cusps} $\PP^1(F) \hookrightarrow \PP^1(\R)^n$.  We say that $f \in M_2(\frakN)$ is a \emph{cusp form} if $f(z) \to 0$ whenever $z$ tends to a cusp, and we denote the space of cusp forms (of parallel weight $2$ and level $\frakN$) by $S_2(\frakN)$.  We have an orthogonal decomposition $M_2(\frakN)=S_2(\frakN) \oplus E_2(\frakN)$ where $E_2(\frakN)$ is spanned by Eisenstein series of level $\frakN$; for level $\frakN=(1)$, we have $\dim E_2(1) = \#\Cl^+ \Z_F$, where $\Cl^+ \Z_F$ denotes the strict class group of $\Z_F$.

Hilbert modular forms admit Fourier expansions as follows.  For a fractional ideal $\frakb$ of $F$, let 
\[ \frakb_{+} = \{x \in \frakb : x_i > 0 \text{ for $i=1,\dots,n$}\}. \] 
Let $\frakd$ be the different of $F$, and let $\frakd^{-1}$ denote the inverse different.
A Hilbert modular form $f\in M_{2}(\frakN)$ admits a Fourier expansion
\begin{equation} \label{qExpHilbert}
f(z)=a_0+\sum_{\mu\in (\frakd^{-1})_{+}}a_\mu e^{2\pi i\Tr(\mu z)}.
\end{equation}
with $a_0=0$ if $f$ is a cusp form.

Let $f \in M_2(\frakN)$ and let $\frakn \subseteq  \Z_F$ be a nonzero ideal.  Then under our hypothesis that $F$ has strict class number $1$, we may write $\frakn=\nu \frakd^{-1}$ for some $\nu \in \frakd_{+}$;  we then define $a_\frakn = a_\nu$.  The transformation rule (\ref{slashfz}) implies that $a_\frakn$ does not depend on the choice of $\nu$, and we call $a_\frakn$ the \emph{Fourier coefficient} of $f$ at $\frakn$.  

The spaces $M_2(\frakN)$ and $S_2(\frakN)$ are also equipped with an action of pairwise commuting diagonalizable \emph{Hecke operators} $T_\frakn$ indexed by the nonzero ideals $\frakn$ of $\Z_F$.  For example, given a prime $\frakp \nmid \frakN$ and a totally positive generator $p$ of $\frakp$ we have
\begin{equation} \label{Heckehilbert}
(T_\frakp f)(z) = N(\frakp) f(pz) + \frac{1}{N(\frakp)} \sum_{a \in \F_\frakp} f\left(\frac{z+a}{p}\right),
\end{equation}
where $\F_\frakp = \Z_F/\frakp$ is the residue field of $\frakp$; this definition is indeed independent of the choice of generator $p$.  Using the notation (\ref{fslshhilb}), we can equivalently write
\begin{equation} \label{Heckehilbertslsh}
(T_\frakp f)(z) = \sum_{a \in \PP^1(\F_\frakp)} (f \slsh{} \pi_a)(z) 
\end{equation}
where $\pi_\infty=\begin{pmatrix} p & 0 \\ 0 & 1 \end{pmatrix}$ and $\pi_a=\begin{pmatrix} 1 & a \\ 0 & p \end{pmatrix}$ for $a \in \F_\frakp$.

If $f \in S_2(\frakN)$ is an eigenform, \emph{normalized} so that $a_{(1)}=1$, then $T_\frakn f = a_\frakn f$, and each eigenvalue $a_\frakn$ is an algebraic integer which lies in the number field $E=\Q(\{a_\frakn\}) \subseteq  \C$ (see Shimura~\cite[Section 2]{shimura2}) generated by the Fourier coefficients of $f$.  We again have notions of \emph{oldforms} and \emph{newforms}, analogously defined (so that a newform is in particular a normalized eigenform).

Associated to an eigenform $f \in S_2(\frakN)$ we have an $L$-function 
\[ L(f,s) = \sum_{\frakn} \frac{a_{\frakn}}{N\frakn^s} \]
and $\frakl$-adic Galois representations 
\[ \rho_{f,\frakl}:\Gal(\overline{F}/F) \to \GL_2(\overline{\Z}_{F,\frakl}) \] 
for primes $\frakl$ of $\Z_F$ such that, for any prime $\frakp \nmid \frakl\frakN$, we have
\[ \Tr(\rho_{f,\frakl}(\Frob_{\frakp}))=a_\frakp(f) \quad \text{ and } \quad \det(\rho_{f,\frakl}(\Frob_\frakp))=N\frakp. \]
Each of these is determined by the Hecke eigenvalues $a_\frakn$ of $f$, so we are again content to compute $S_2(\frakN)$ as a Hecke module.  

We are now ready to state the first version of our main result.

\begin{thm}[Demb\'el\'e \cite{Dembeleclassno1}, Greenberg-Voight \cite{GreenbergVoight}]  \label{mainthmclno1} 
There exists an algorithm which,   given a totally real field $F$ of strict class number $1$ and a nonzero ideal $\frakN \subseteq  \Z_F$, computes the space $S_2(\frakN)$ of Hilbert cusp forms of parallel weight $2$ and level $\frakN$ over $F$ as a Hecke module.
\end{thm}

In other words,   there exists an explicit finite procedure which takes as input   the field $F$ and the ideal $\frakN \subseteq  \Z_F$ encoded in bits (in the usual way, see e.g.\ Cohen \cite{Cohen1}), and outputs a finite set of sequences $(a_\frakp(f))_{\frakp}$   encoding the Hecke eigenvalues for each cusp form constituent $f$ in $S_2(\frakN)$,   where $a_\frakp(f) \in E_f \subseteq  \Qbar$.  (This algorithm will produce any finite subsequence in a finite amount of time, but in theory will produce the entire sequence if it is left to run forever.)  Alternatively, this algorithm can simply output matrices for the Hecke operators $T_\frakp$; one recovers the constituent forms using linear algebra.

\begin{exm}
Let $F=\Q(\sqrt{5})$.  Then $\Z_F=\Z[w]$ where $w=(1+\sqrt{5})/2$ satisfying $w^2-w-1=0$.  Let $\frakN=(3w-14) \subseteq  \Z_F$; we have $N(\frakN)=229$ is prime.  We compute that $\dim S_2(\frakN)=4$.    There are $2$ Hecke irreducible subspaces of dimensions $1$ and $3$, corresponding to newforms $f$ and $g$ (and its Galois conjugates).  We have the following table of eigenvalues; we write $\frakp=(p)$ for $p \in \Z_F$.
\[
\begin{array}{c||ccccc}
\frakp  & (2) & (w+2) & (3) & (w+3) & (w-4) \\
N\frakp & 4 & 5   & 9 & 11  & 11   \\
\hline
\rule{0pt}{2.5ex} 
a_\frakp(f) & -3 & -4 & -1 & 0 & -2 \\
a_\frakp(g) & t & t^2-4t+1 & -t^2 + 2t + 2 & t^2 -2t -3 & -3t^2+8t+1
\end{array}  
\] 
Here, the element $t \in \Qbar$ satisfies $t^3-3t^2-t+1=0$ and $E=\Q(t)$ is an $S_3$-field of discriminant 148.  
\end{exm}

Recall that in the method of modular symbols, a cusp form $f \in S_2(N)$ corresponds to a holomorphic differential (1-)form $(2\pi i) f(z)\,dz$ on $X_0(N)$ and so, by the theorem of Eichler-Shimura, arises naturally in the space $H^1(X_0(N),\C)$.  In a similar way, a Hilbert cusp form $f \in S_2(\frakN)$ gives rise to a holomorphic differential $n$-form $(2\pi i)^n f(z_1,\ldots,z_n)\, dz_1 \cdots dz_n$   on the \emph{Hilbert modular variety}   $X_{0}(\frakN)$, the desingularization of the compact space $\Gamma_0(\frakN) \backslash (\calH^n)^*$ where $(\calH^n)^* = \calH^n \cup \PP^1(F)$.  But now $X_0(\frakN)$ is an algebraic variety of complex dimension $n$ and $f$ arises in the cohomology group $H^n(X_0(\frakN),\C)$.  Computing with higher dimensional varieties (and higher degree cohomology groups) is not an easy task!  So we seek an alternative approach.

Langlands functoriality predicts that $S_2(\frakN)$ as a Hecke module occurs in the cohomology of other ``modular'' varieties as well.  This functoriality was already evident by the fact that both modular symbols and their quaternionic variant, Brandt matrices, can be used to compute the classical space $S_2(N)$.  In our situation, this functoriality is known as the Jacquet-Langlands correspondence, which ultimately will allow us to work with varieties of complex dimension $1$ or $0$   by considering twisted forms of $\GL_2$ over $F$ arising from quaternion algebras.  In dimension $1$, we will arrive at an algorithm which works in the cohomology of a Shimura curve, analogous to a modular curve, and thereby give a kind of analogue of modular symbols; in dimension $0$, we generalize Brandt matrices by working with theta series on (totally definite) quaternion orders.  

\section{Quaternionic modular forms} \label{sec:auto-forms-1} 

In this section, we define modular forms on quaternion algebras; our main reference is Hida \cite{Hidabook}.  We retain the notation of the previous section; in particular, $F$ is a totally real field of degree $[F:\Q]=n$ with ring of integers $\Z_F$.

A \emph{quaternion algebra} $B$ over $F$ is a central simple algebra of dimension $4$.  Equivalently, a quaternion algebra $B$ over $F$ is an $F$-algebra generated by elements $i,j$ satisfying
\begin{equation} \label{quateq}
i^2=a,\quad j^2=b,\quad \text{and}\quad ji=-ij
\end{equation}
for some $a,b \in F^\times$; we denote such an algebra $B=\quat{a,b}{F}$.  For more information about quaternion algebras, see Vign\'eras \cite{Vigneras}.  

Let $B$ be a quaternion algebra over $F$.  Then $B$ has a unique involution $\overline{\phantom{x}}:B \to B$ called \emph{conjugation} such that $x\overline{x} \in F$ for all $x \in B$; we define the \emph{reduced norm} of $x$ to be $\nrd(x)=x\overline{x}$.  For $B=\quat{a,b}{F}$ as in (\ref{quateq}) and $x=u+v i+z j+w ij\in B$, we have 
\[ \overline{x}=u-(v i+z j+w ij) \quad \text{and} \quad \nrd(x)=u^2-av^2-bz^2+abw^2. \]

A $\Z_F$-\emph{lattice} of $B$ is a finitely generated $\Z_F$-submodule $I$ of $B$ such that $FI=B$.  An \emph{order} $\calO$ of $B$ is a $\Z_F$-lattice which is also a subring of $B$.  A \emph{maximal order} of $B$ is an order which is not properly contained in any other order.  Let $\calO_0(1) \subseteq  B$ be a maximal order in $B$.  

A \emph{right fractional $\calO$-ideal} is a $\Z_F$-lattice $I$ such that its \emph{right order} $\calO_R(I)=\{x \in B : xI \subseteq  I\}$ is equal to $\calO$; left ideals are defined analogously.  

Let $K \supset F$ be a field containing $F$.  Then $B_K=B \otimes_F K$ is a quaternion algebra over $K$, and we say $K$ \emph{splits} $B$ if $B_K \cong \M_2(K)$.  

Let $v$ be a noncomplex place of $F$, and let $F_v$ denote the completion of $F$ at $v$.  Then there is a unique quaternion algebra over $F_v$ which is a division ring, up to isomorphism.  We say $B$ is \emph{unramified} (or \emph{split}) at $v$ if $F_v$ splits $B$, otherwise we say $B$ is \emph{ramified} at $v$.  The set $S$ of ramified places of $B$ is a finite set of even cardinality which characterizes $B$ up to isomorphism, and conversely given any such set $S$ there is a quaternion algebra over $B$ ramified exactly at the places in $S$.  We define the \emph{discriminant} $\frakD$ of $B$ to be the ideal of $\Z_F$ given by the product of all finite ramified places of $B$.

Let $\frakN \subseteq  \Z_F$ be an ideal which is coprime to the discriminant $\frakD$.  Then there is an isomorphism
\[ \iota_\frakN : \calO_0(1) \hookrightarrow \calO_0(1) \otimes_{\Z_F} \Z_{F,\frakN} \cong \M_2(\Z_{F,\frakN}) \]
where $\Z_{F,\frakN}$ denotes the completion of $\Z_F$ at $\frakN$.  Let 
\[ \calO_0(\frakN) = \{x \in \calO_0(1) : \iota_\frakN(x) \text{ is upper triangular modulo $\frakN$}\}; \]
the order $\calO_0(\frakN)$ is called an \emph{Eichler order} of level $\frakN$.  We will abbreviate $\calO=\calO_0(\frakN)$.

We number the real places $v_1,\dots,v_n$ of $F$ so that $B$ is split at $v_1,\dots,v_r$ and ramified at $v_{r+1},\dots,v_n$, so that
\[ B \otimes_\Q \R \cong \M_2(\R)^{r} \times \HH^{n-r} \]
where $\HH=\quat{-1,-1}{\R}$ is the division ring of Hamiltonians.  If $B$ is ramified at all real places (i.e.\ $r=0$) then we say that $B$ is \emph{(totally) definite}, and otherwise we say $B$ is \emph{indefinite}.  The arithmetic properties of the algebra $B$ and its forms are quite different according as $B$ is definite or indefinite, and so we consider these two cases separately.  Using an adelic language, one can treat them more uniformly (though to some extent this merely repackages the difference)---we refer to \S 8 for this approach.

\medskip

First, suppose that $B$ is indefinite, so that $r>0$.  The case $B\cong \M_2(\Q)$ corresponds to the classical case of elliptic modular forms; this was treated in \S 1, so we assume $B \not\cong \M_2(\Q)$.  Let 
\[ \iota_\infty: B \hookrightarrow \M_2(\R)^r \]
denote the map corresponding to the split embeddings $v_1,\dots,v_r$.  Then the group
\[ B_+^\times = \{ \gamma \in B^\times : \det \gamma_i = (\nrd \gamma)_i > 0 \text{ for $i=1,\dots,r$}\} \]
acts on $\calH^r$ by coordinatewise linear fractional transformations.  Let
\[ \calO_+^\times=\calO^\times \cap B_+^\times. \]
Under the assumption that $F$ has strict class number $1$, which we maintain, we have 
\[ \calO_+^\times = \Z_F^\times \calO^\times_1 \]
where $\calO^\times_1 = \{ \gamma \in \calO : \nrd(\gamma)=1\}$.  Let
\[ \Gamma=\Gamma_0^B(\frakN) = \iota_{\infty}(\calO_+^\times) \subseteq  \GL_2^+(\R)^r. \]
and let $\PGamma = \Gamma/\Z_F^\times$.  Then $\PGamma$ is a discrete subgroup of $\PGL_2^+(\R)^r$ which can be identified with
\[ \PGamma \cong \iota_{\infty}(\calO_1^\times)/\{\pm 1\} \subseteq  \PSL_2(\R). \]

\begin{defn} \label{quatcuspformindef}
Let $B$ be indefinite.  A \emph{quaternionic modular form} for $B$ of parallel weight $2$ and level $\frakN$ is a holomorphic function $f:\calH^r \to \C$ such that
\begin{equation} \label{slashfzquat}
f(\gamma z) = f\left(\frac{a_1 z_1 + b_1}{c_1 z_1 + d_1}, \dots, \frac{a_r z_r + b_r}{c_r z_r + d_r}\right) = \left(\prod_{i=1}^r \frac{(c_i z_i+d_i)^2}{\det \gamma_i} \right) f(z)
\end{equation}
for all $\gamma\in\Gamma_0^B(\frakN)$.
\end{defn}

Analogous to (\ref{fslshhilb}), we define
\begin{equation} \label{fslshquat}
(f \slsh{} \gamma)(z) = f(\gamma z) \prod_{i=1}^r \frac{\det \gamma_i}{j(\gamma_i, z)^2}
\end{equation}
for $f:\calH^r \to \C$ and $\gamma \in B_+^\times$; then (\ref{slashfzquat}) is equivalent to $(f \slsh{} \gamma)(z) = f(z)$ for all $\gamma \in \Gamma_0^B(\frakN)$.  

We denote by $M_2^B(\frakN)$ the space of quaternionic modular forms for $B$ of parallel weight $2$ and level $\frakN$, a finite-dimensional $\C$-vector space.  

A quaternionic modular form for $B=\M_2(F)$ is exactly a Hilbert modular form over $F$; our presentation in these three sections has been consciously redundant so as to emphasize this similarity.  (We could recover the definition of cusp forms given in Section 1 if we also impose the condition that the form vanish at the cusps.)  As we will see later, this similarity is less apparent when the general and more technical theory is exposited.  

The Hecke operators are defined on $M_2^B(\frakN)$ following their definition in (\ref{Heckehilbertslsh}).  Let $\frakp$ be a prime of $\Z_F$ with $\frakp \nmid \frakN$, and let $p$ be a totally positive generator of $\frakp$.  Define
\[ \Theta(\frakp)=\calO_+^\times \backslash \left\{\pi \in\calO_{+}:\nrd(\pi)\Z_F=\frakp\right\}
=\calO_{+}^\times \backslash \left\{\pi\in\calO_{+}:\nrd(\pi)=p\right\}. \]
The set $\Theta(\frakp)$ has cardinality $N\frakp+1$.  The Hecke operator $T_\frakp$ is then given by
\begin{equation} \label{Heckeindefinite}
(T_\frakp f)(z) = \sum_{ \pi \in \Theta(\frakp)} (f \slsh{} \pi)(z).
\end{equation} 

The set $\Theta(\frakp)$ admits an explicit description as follows.  As above, let $\F_\frakp=\Z_F/\frakp$ be the residue field of $\frakp$, and let $\iota_\frakp:\calO \hookrightarrow \M_2(\Z_{F,\frakp})$ be a splitting.  Then the set $\Theta(\frakp)$ is in bijection with the set of left ideals of $\calO$ by $\pi \mapsto \calO\pi$.  This set of left ideals is in bijection \cite[Lemma 6.2]{KirschmerVoight} with the set $\PP^1(\F_\frakp)$: explicitly, given the splitting $\iota_\frakp$, the left ideal corresponding to $a=(x:y)\in\PP^1(\F_\frakp)$ is
\begin{equation} \label{Iaideals}
J_a = \calO\iota_\frakp^{-1}\begin{pmatrix} x & y \\ 0 & 0 \end{pmatrix} + \calO\frakp.
\end{equation}
By strong approximation \cite[Th\'eor\`eme III.4.3]{Vigneras}, each of the ideals $J_a$ is principal, so $J_a=\calO \pi_a$ with $\nrd(\pi_a)=p$ for all $a \in \PP^1(\F_\frakp)$.  Therefore, we have $\Theta(\frakp)=\{\pi_a : a \in \PP^1(\F_\frakp)\}$.  

This definition reduces to the one given in (\ref{Heckehilbertslsh}) for Hilbert modular forms with the choices $\pi_\infty=\begin{pmatrix} p & 0 \\ 0 & 1 \end{pmatrix}$ and $\pi_a=\begin{pmatrix} 1 & a \\ 0 & p \end{pmatrix}$ for $a \in \F_\frakp$.

Having treated Hilbert modular forms in the previous section, now suppose that $B \not\cong \M_2(F)$, or equivalently that $B$ is a division ring. Then a modular form is vacuously a cusp form as there are no cusps!  We then refer to quaternionic modular forms equally well as \emph{quaternionic cusp forms} and let $S_2^B(\frakN)=M_2^B(\frakN)$.  Here, a cusp form $f$ gives a holomorphic differential $r$-form $(2\pi i)^r f(z_1,\dots,z_r)\, dz_1 \cdots dz_r$ on the associated \emph{quaternionic Shimura variety} $X_0^B(\frakN) = \Gamma_0^B(\frakN) \backslash \calH^r$, a complex variety of dimension $r$.  

The important case for us will be when $r=1$.  Then $\Gamma_0^B(\frakN) \subseteq  \PGL_2^+(\R)$ acts on the upper half-plane and the quotient $\Gamma_0^B(\frakN) \backslash \calH$ can be given the structure of a Riemann surface, known as a \emph{Shimura curve}.  In this simple case, a cusp form for $B$ is simply a holomorphic map $f:\calH \to \C$ such that $f(\gamma z)=(c_1z+d_1)^2 f(z)$ for all $\gamma \in \Gamma_0^B(\frakN)$, where $\gamma_1=v_1(\gamma)=\begin{pmatrix} a_1 & b_1 \\ c_1 & d_1 \end{pmatrix}$ and $v_1$ is the unique split real place of $F$.

\medskip

Next, suppose that $B$ is definite, so that $r=0$.  Recall that $\calO=\calO_0(\frakN) \subseteq  \calO_0(1)$ is an Eichler order of level $\frakN$.  A right fractional $\calO$-ideal is \emph{invertible} if there exists a left fractional $\calO$-ideal $I^{-1}$ such that $I^{-1}I=\calO$, or equivalently if $I$ is \emph{locally principal}, i.e., for each (finite) prime ideal $\frakp$ of $\Z_F$, the ideal $I_\frakp$ is a principal right $\calO_\frakp$-ideal.  If $I$ is invertible, then necessarily $I^{-1}=\{x \in B: xI \subseteq  \calO\}$.

Let $I,J$ be invertible right fractional $\calO$-ideals.  We say that $I$ and $J$ are in the same \emph{right ideal class} (or are \emph{isomorphic}) if there exists an $x \in B^\times$ such that $I=x J$, or equivalently if $I$ and $J$ are isomorphic as right $\calO$-modules.  We write $[I]$ for the equivalence class of $I$ under this relation and denote the set of invertible right $\calO$-ideal classes by $\Cl \calO$.  The set $\Cl \calO$ is finite and $H=\#\Cl \calO$ is independent of the choice of Eichler order $\calO=\calO_0(\frakN)$ of level $\frakN$.

\begin{defn}
Let $B$ be definite.  A \emph{quaternionic modular form} for $B$ of parallel weight $2$ and level $\frakN$ is a map
\[ f:\Cl \calO_0(\frakN) \to \C. \]
\end{defn}

The space of quaternionic modular forms $M_2^B(\frakN)$ is obviously a $\C$-vector space of dimension equal to $H$.  

A modular form for $B$ which is orthogonal to the ($1$-dimensional) subspace of constant functions is called a \emph{cusp form} for $B$; the space of such forms is denoted $S_2^B(\frakN)$.

The Hecke operators are defined on $M_2^B(\frakN)$ as follows. Let $\frakp$ be a prime ideal of $\Z_F$ with $\frakp \nmid \frakN$.  For a right $\calO$-ideal $I$ with $\nrd(I)$ coprime to $\frakp$, the Hecke operator $T_\frakp$ is given by
\begin{equation} \label{heckedefdefinite}
(T_\frakp f)([I]) = \sum_{\substack{J \subseteq  I \\ \nrd(J I^{-1})=\frakp}} f([J]),
\end{equation}
the sum over all invertible right $\calO$-ideals $J \subseteq I$ such that $\nrd(J)=\frakp \nrd(I)$.  As in (\ref{Heckeindefinite}), this sum is also naturally over $\PP^1(\F_\frakp)$, indexing the ideals of norm index $\frakp$.  This definition does not depend on the choice of representative $I$ in its ideal class and extends by linearity to all of $S_2^B(\frakN)$.

\medskip

Consequent to the definitions in the previous paragraphs, we may now consider the Hecke modules of quaternionic cusp forms over $F$ for the different quaternion algebras $B$ over $F$.  These spaces are related to each other, and thus to spaces of Hilbert modular forms, according to their arithmetic invariants by the \emph{Jacquet-Langlands correspondence} as follows.

\begin{thm}[Eichler-Shimizu-Jacquet-Langlands]\label{shimizuclno1} 
Let $B$ be a quaternion algebra over $F$ of discriminant $\frakD$ and let $\frakN$ be an ideal coprime to $\frakD$.  Then
there is an injective map of Hecke modules
\[ S_{2}^{B}(\frakN) \hookrightarrow S_{2}(\frakD\frakN) \] 
whose image consists of those Hilbert cusp forms which are new at all primes $\frakp \mid \frakD$.
\end{thm}

\begin{proof} 
See Jacquet and Langlands~\cite[Chap. XVI]{jacqlang}, Gelbart and Jacquet \cite[\S 8]{geljac} and work of Hida~\cite{hida}; another useful reference is Hida \cite[Proposition 2.12]{HidaCM} who deduces Theorem \ref{shimizuclno1} from the representation theoretic results of Jacquet and Langlands.
\end{proof}

Theorem~\ref{shimizuclno1} yields an isomorphism 
\[ S_{2}^B(\frakN) \cong S_{2}(\frakN) \] 
when the quaternion algebra $B$ has discriminant $\frakD=(1)$.  Since a quaternion algebra must be ramified at an even number of places, when $n=[F:\Q]$ is even we can achieve this for the definite quaternion algebra $B$ which is ramified at exactly the real places of $F$ (and no finite place).  When $n$ is odd, the simplest choice is to instead take $B$ to be ramified at all but one real place of $F$ (and still no finite place), and hence $B$ is indefinite (and $g=1$).

\begin{rmk} \label{rmkgenspace}
Note that in general a space of newforms can be realized as a Hecke module inside many different spaces of quaternionic cusp forms.  Indeed, for any factorization $\frakM=\frakD \frakN$ with $\frakD$ squarefree and $\frakN$ coprime to $\frakD$, we consider a quaternion algebra $B$ of discriminant $\frakD$ (ramified at either all or all but one real place of $F$) and thereby realize $S_2^B(\frakN) \cong S_2(\frakM)^{\frakD\text{-new}}$.  For example, if $p,q$ are (rational) primes, then the space $S_2(pq)^{\text{new}}$ of classical newforms can, after splitting off old subspaces, be computed using an indefinite quaternion algebra of discriminant $1$ or $pq$ (corresponding to a modular curve or a Shimura curve, respectively) or a definite quaternion algebra of discriminant $p$ or $q$.
\end{rmk}

Our main conclusion from this section is that to compute spaces of Hilbert cusp forms it suffices to compute instead spaces of quaternionic cusp forms.  The explicit description of $S_2^B(\frakN)$ as a Hecke module varies according as if $B$ is definite or indefinite.

\section{Definite method}\label{sec:definite-classno1}

In this section, we discuss a method for computing Hilbert modular forms using a definite quaternion algebra $B$.  We continue with our notation and our assumption that $F$ has strict class number $1$.  We accordingly call the method in this section the \emph{definite method}: it is a generalization of the method of Brandt matrices mentioned briefly in \S 2 and was first exhibited by Eichler \cite{Eichler} and Pizer~\cite{Pizer} for $F=\Q$, but the first explicit algorithm was given by Socrates and Whitehouse \cite{SocratesWhitehouse}.

Let $I_1,\dots,I_H$ be a set of representative right ideals for $\Cl \calO$, with $H=\# \Cl \calO$. As vector spaces, we  have simply that 
\[ M_2^B(\frakN)=\Map(\Cl \calO,\C) \cong \textstyle{\bigoplus_{i=1}^{H}} \C \, I_i, \]
associating to each ideal (class) its characteristic function.  Let $\calO_i=\calO_L(I_i)$ be the left order of $I_i$  and let $e_i=\# (\calO_i^\times/\Z_F^\times)$.  

The action of the Hecke operators is defined by (\ref{heckedefdefinite}): we define the $\frakp$th-\emph{Brandt matrix} $T(\frakp)$ for $\calO$ to be the matrix whose $(i,j)$th entry is equal to  
\begin{equation}
b(\frakp)_{i,j}=\#\{ J \subseteq  I_j : \nrd(JI_j^{-1}) = \frakp \text{ and } [J]=[I_i] \} \in \Z.
\end{equation}
The Brandt matrix $T(\frakp)$ is an $H \times H$-matrix with integral entries such that the sum of the entries in each column is equal to $N\frakp+1$.  The Hecke operator $T_\frakp$ then acts by $T(\frakp)$ on $\bigoplus_i \C I_i$ (on the right), identifying an ideal class with its characteristic function.  

The Brandt matrix is just a compact way of writing down the adjacency matrix of the graph   with vertices $X=\Cl \calO$   where there is a directed edge from $I_i$ to each ideal class which represents an ideal of index $N\frakp$ in $I_i$.  Indeed, consider the graph whose vertices are right $\calO$-ideals of norm a power of $\frakp$ and draw a directed edge from $I$ to $J$ if $\nrd(JI^{-1})=\frakp$.  Then this graph is a $k$-regular tree with $N\frakp+1$ edges leaving each vertex.  The above adjacency matrix is obtained by taking the quotient of this graph by identifying two ideals if they are in the same ideal class.

Alternatively, we may give an expression for the Brandt matrices in terms of elements instead of ideals.  A containment $J \subseteq I_j$ of right $\calO$-ideals with $[I_i]=[J]$ corresponds to an element $x \in JI_i^{-1} \subset I_j I_i^{-1}$ via $J = x I_i$, and we have $\nrd(JI^{-1})=\frakp$ if and only if $\nrd(x)\Z_F=\frakp$.  

Writing $JI_i^{-1}=x\calO_i$, we see that $x$ is unique up to multiplication on the right by $\calO_i^\times$.  We have $\calO_i^\times = (\calO_i)_1^\times \Z_F^\times$ and $(\calO_i)_1^\times \cap \Z_F^\times = \{\pm 1\}$, so $2e_i=\# (\calO_i)_1^\times$.  To eliminate the contribution of the factor $\Z_F^\times$, we normalize as follows: let $p$ be a totally positive generator for $\frakp$ and similarly $q_i$ for $\nrd(I_i)$ for $i=1,\dots,H$.  Then $x \in I_j I_i^{-1}$ can be chosen so that $\nrd(x)(q_j/q_i) = p$ and is unique up to multiplication by $(\calO_i)_1^\times$.  
\begin{equation} \label{bijquad}
b(\frakp)_{i,j}=\frac{1}{2e_i} \#\left\{ x \in I_j I_i^{-1} : \nrd(x)\frac{q_j}{q_i} = p \right\}.
\end{equation}

The advantage of the expression (\ref{bijquad}) is that it can be expressed simply in terms of a quadratic form.  Since $B$ is definite, the space $B \hookrightarrow B \otimes_{\Q} \R \cong \HH^n \cong \R^{4n}$ comes equipped with the positive definite quadratic form $\Tr \nrd : B \to \R$.  If $J$ is a $\Z_F$-lattice, then $J \cong \Z^{4n}$ embeds as a Euclidean lattice $J \hookrightarrow \R^{4n}$ with respect to this quadratic form.  It follows that one can compute $b(\frakp)_{i,j}$ by computing all elements $x \in I_i I_j^{-1}$ such that $\Tr (q_j/q_i)\nrd(x) \leq \Tr p$, a finite set.

Before giving references for the technical details about how the Brandt matrices above are computed explicitly, we pause to give three examples.

\begin{exm} \label{defexmoverQQ}
Consider the quaternion algebra $B=\quat{-1,-23}{\Q}$, so that $B$ is generated by $i,j$ subject to $i^2=-1$, $j^2=-23$ and $ji=-ij$.  We have the maximal order $\calO=\calO_0(1)=\Z \oplus \Z i \oplus \Z k \oplus \Z ik$ where $k=(j+1)/2$.  We consider the prime $p=2$; we have an embedding $\calO \hookrightarrow \M_2(\Z_2)$ where $i,k \mapsto \begin{pmatrix} 0 & 1 \\ 1 & 0 \end{pmatrix}, \begin{pmatrix} 1 & 0 \\ 0 & 0 \end{pmatrix} \pmod{2}$.  

We begin by computing the ideal classes in $\calO$.  We start with $C_1=[\calO]$.  We have $3$ ideals of norm $2$, namely $I_{(0:1)}=2\calO + ik \calO$, $I_{(1:1)}=2\calO + (i+1)k\calO$, and $I_{(1:0)}=2\calO + k\calO$.  If one of these ideals is principal, then it is generated by an element of reduced norm $2$.  The reduced norm gives a quadratic form
\begin{align*}
\nrd:\calO &\to \Z \\
x+yi+zk+wik &\mapsto x^2+xz+y^2+yw+6z^2+6w^2
\end{align*}
We see immediately that $\nrd(x+yi+zk+wik)=2$ if and only if $z=w=0$ and $x=y= \pm 1$.  This shows that $I_{(1:1)}=(i+1)\calO$ is principal but $I_{(1:0)}$ and $I_{(0:1)}$ are not.  Note also that we find $2e_1=4$ solutions matching (\ref{bijquad}).  We notice, however, that $iI_{(1:0)} = I_{(0:1)}$, so we have just a second ideal class $C_2=[I_{(0:1)}]$.  

Now of the two ideals contained in $I_{(0:1)}$ of norm $4$, we have $I_{(0:1)}^{(4)}=4\calO + i(k+2)\calO$ belongs to $C_1$ whereas $I_{(2:1)}^{(4)}=4\calO + (2i+2k+ik)\calO$ gives rise to a new ideal class $C_2=[I_{(2:1)}]$.  If we continue, each of the two ideals contained in $I_{(0:1)}^{(4)}$ of norm $8$ belong to $C_1$, and it follows that $H=\# \Cl \calO=3$.  

From this computation, we have also computed the Brandt matrix $T(2)=\begin{pmatrix} 1 & 1 & 0 \\ 2 & 1 & 3 \\ 0 & 1 & 0 \end{pmatrix}$.  Indeed, the first column encodes the fact that of the three right $\calO$-ideals of reduced norm $2$, there is one which is principal and hence belongs to $C_1$ and two that belong to $C_2$.  We think of this matrix as acting on the right on row vectors.  

The characteristic polynomial of $T(2)$ factors as $(x-3)(x^2+x-1)$.  The vector $(1,1,1)$ is an eigenvector with eigenvalue $3$ which generates the space of constant functions and gives rise to the Eisenstein series having eigenvalues $a_p=p+1$ for all primes $p \neq 23$.  The space of cusp forms $S_2^B(1)$ is correspondingly of dimension $2$ and is irreducible as a Hecke module. The Hecke module $S_2^B(1)$ can be explicitly identified with $S_2(23)$ using theta correspondence.  For example, the series
\[ \theta_1(q) = \sum_{\gamma \in \calO} q^{\nrd(\gamma)} = \sum_{x,y,z,w \in \Z} q^{x^2+xz+y^2+yw+6z^2+6w^2} = 1 + 4q^2 + 4q^4 + 4q^8 + 8q^{10} + \dots \]
is the $q$-expansion of a modular form of level $23$ and weight $2$ and corresponds to (the characteristic of) $C_1$.  
For more details, we refer to Pizer~\cite[Theorem 2.29]{Pizer}, where the first computer algorithm for computing Brandt matrices over $\Q$ is also described.  
\end{exm}

Now we give an example over a quadratic field.

\begin{exm} \label{defexmqsrt5}
Let $F=\Q(\sqrt{5})$.  Then $\Z_F=\Z[w]$ where $w=(1+\sqrt{5})/2$ satisfying $w^2-w-1=0$.  The prime $61$ splits in $F$; let $\frakN=(3w+7)\Z_F$ be one of the primes above it.  We consider the (Hamilton) quaternion algebra $B=\quat{-1,-1}{F}$ over $F$ of discriminant $\frakD=(1)$.  We have the maximal order
\[\calO_0(1) = \Z_F\oplus i \Z_F \oplus k \Z_F \oplus ik \Z_F, \quad \text{ where } \quad k=\frac{(1 + w) + w i + j}{2}, \]
and the Eichler order $\calO \subseteq  \calO_0(1)$ of level $\frakN$ given by
\[\calO = \Z_F\oplus (3w+7)i \Z_F \oplus (-30i + k)\Z_F \oplus (w+20i+ik)\Z_F .\]

The class number of $\calO$ is $H=3$.  The following ideals give a set of representatives for $\Cl \calO$: we take $I_1=\calO$, 
\begin{align*}
I_2 &= 2\calO + ((w+2) - (2w+2)i + (-1+3w)ik)\calO \\
&= 2\Z_F \oplus (6w+14)i\Z_F \oplus ((w+1) + (-2w+5)i - k)\Z_F \oplus (1 -6i + wik)\Z_F 
\end{align*}
and $I_3 = 2\calO + ((w+1) + (1-w)i + (2-2w)k)\calO$.


We compute the orders $e_i=\#(\calO_i/\Z_F^\times)$ as $e_1=2$, $e_2=5$ and $e_3=3$.  For example, the element $u=(2w)i - k - wik \in \calO$ satisfies the equation $u^2+1=0$, and so yields an element of order $2$ in $\calO^\times/\Z_F^\times$.  It follows that none of these orders are isomorphic (i.e., conjugate) in $B$.   

The first few Brandt matrices are:
\begin{gather*}   
T(2)=\begin{pmatrix} 1& 5& 3\\ 2& 0& 0\\ 2& 0& 2\end{pmatrix}, \quad T(\sqrt{5})=\begin{pmatrix}4& 0&  3\\ 0& 1& 3\\ 2& 5& 0\end{pmatrix}, \quad T(3)=\begin{pmatrix}4& 5& 6\\ 2& 0& 3\\ 4& 5& 1\end{pmatrix}, \\
T(w+3)=\begin{pmatrix}4& 10& 6\\ 4& 2& 0\\ 4& 0& 6\end{pmatrix}, \quad T(w-4)=\begin{pmatrix}6& 5& 6\\ 2& 2& 3\\ 4& 5& 3\end{pmatrix}.
\end{gather*}
We note that $N(w+3)=N(w-4)=11$.

For example, the first column of the matrix $T(2)$ records the fact that of the $5=N(2)+1$ right $\calO$-ideals of norm $(2)$, there is exactly one which is principal, two are isomorphic to $I_2$ and the other two are isomorphic to $I_3$.  

The space $S_2^B(\frakN)$ of cusp forms is an irreducible 2-dimensional Hecke module, represented by a constituent form $f$ with corresponding eigenvector $(2,5w,-3w-3)$.  In particular, the ring of Hecke operators $\T_f=\Z[a_\frakp]$ restricted to $f$ is equal to $\T_f=\Z[w]$, by coincidence.  We have the following table of eigenvalues for $f$.
\begin{align*}
\begin{array}{c||ccccc}
\frakp & (2) & (w+2) & (3) & (w+3) & (w-4) \\
N\frakp& 4& 5& 9& 11 & 11 \\ \hline
\rule{0pt}{2.5ex} 
a_\frakp(f)& 2w - 2& -3w + 1& -w - 2& 4w - 2 & -w
\end{array}
\end{align*}
For further discussion of the geometric objects which arise from this computation, see the discussion in Section~\ref{sec:examples-1}.
\end{exm}

Finally, one interesting example.

\begin{exm}
Let $F=\Q(\sqrt{15})$ and let $\frakN=(5,\sqrt{15})$.  Then there exists a cusp form constituent of an irreducible space of dimension $8$ in $S_2(\frakN)$ such that no single Hecke eigenvalue generates the entire field $E$ of Hecke eigenvalues.   
\[
\xymatrix@R=2ex@C=1ex{
& & E_{\textup{gal}}=\Q(\sqrt{17},i,\sqrt{2},\sqrt{u}) \ar@{-}[d] \\
& & E=\Q(\sqrt{17},i,\sqrt{u}) \ar@{-}[dl] \ar@{-}[d] \ar@{-}[dr] \\
& \Q(i,\sqrt{17}) \ar@{-}[dl] \ar@{-}[d] \ar@{-}[dr] & 
\Q(\sqrt{17},\sqrt{u}) \ar@{-}[d] & \Q(\sqrt{17},\sqrt{-u}) \ar@{-}[dl] \\
\Q(\sqrt{-17}) \ar@{-}[dr] & \Q(i) \ar@{-}[d] & \Q(\sqrt{17}) \ar@{-}[dl] \\
& \Q
}   \]
Here, $u=(5+\sqrt{17})/2$ and $E_{\textup{gal}}$ is the Galois closure of $E$.  Each Hecke eigenvalue $a_\frakp$ for this form $f$ generates a proper subfields of $E$.  (There are also examples of this phenomenon over $\Q$, and they are related to the phenomenon of inner twists; this was analyzed over $\Q$ by Koo, Stein, and Wiese \cite{KooStein}.)
\end{exm}

With these examples in hand, we now give an overview of how these computations are performed; for more detail, see work of the first author \cite{Dembeleclassno1}.  It is clear we need several algorithms to compute Brandt matrices.  First, we need a basic suite of algorithms for working with quaternion orders and ideals; these are discussed in work of Kirschmer and the second author \cite[Section 1]{KirschmerVoight}, and build on basic tools for number rings by Cohen \cite{Cohen1}.  As part of this suite, we need a method to compute a maximal order, which is covered by work of the second author \cite{VoightMaxOrder}.  Next, we need to compute a set of representatives for $\Cl \calO$ and to test if two right $\calO$-ideals are isomorphic: this is covered by Kirschmer and the second author \cite[Sections 6--7]{KirschmerVoight}, including a runtime analysis.  To compute a set of representatives, we use direct enumeration in the tree as in Example \ref{defexmoverQQ} and a mass formula due to Eichler as an additional termination criterion.  To test for isomorphism, we use lattice methods to find short vectors with respect to the quadratic form $\Tr \nrd$.

\medskip

In this method, to compute with level $\frakN$ one must compute a set of representatives $\Cl \calO=\Cl \calO_0(\frakN)$ anew.  The first author has given an improvement on this basic algorithm, allowing us to work always with ideal classes $\Cl \calO_0(1)$ belonging to the maximal order at the small price of a more complicated description of the Hecke module.  The proof of correctness for this method is best explained in the adelic language, so we refer to Section 8 for more detail.

Let $I_1,\dots,I_h$ be representatives for $\Cl \calO_0(1)$ and suppose that $\frakN$ is relatively prime to $\nrd(I_i)$ for each $i$---this is made possible by strong approximation.  As before, let $\frakp$ be a prime of $\Z_F$ with $\frakp \nmid \frakD$.  Let $p$ be a totally positive generator for $\frakp$, and let $q_i$ be a totally positive generator for $\nrd(I_i)$.  By our notation, we have $\calO_L(I_i) = \calO_0(1)_i$.  Let $\Gamma_i = \calO_0(1)_i^\times/\Z_F^\times$.  For each $i,j$, consider the set  
\[ \Theta(\frakp)_{i,j}=\Gamma_j \big\backslash\left\{x \in I_j I_i^{-1} : \nrd(x)\frac{q_i}{q_j} = p \right\} \]
where $\Gamma_j$ acts by the identification $\calO_0(1)_i^\times/\Z_F^\times = \calO_0(1)_1^\times/\{\pm 1\}$.  Via a splitting isomorphism 
\[ \iota_\frakN : \calO_0(1) \hookrightarrow \calO_0(1) \otimes \Z_{F,\frakN} \cong \M_2(\Z_{F,\frakN}), \]
the group $\calO_0(1)^\times$ acts on $\PP^1(\Z_F/\frakN)$ and since $\calO_0(1) \otimes \Z_{F,\frakN} \cong \calO_0(1)_i \otimes \Z_{F,\frakN}$ for each $i$ (since $\nrd(I_i)$ is prime to $\frakN$), the group $\Gamma_i=\calO_0(1)_i^\times/\Z_F^\times$ similarly acts on $\PP^1(\Z_F/\frakN)$.  

We then define a Hecke module structure on $\bigoplus_{i=1}^{h} \C[\Gamma_i \backslash \PP^1(\Z_F/\frakN)]$ via the map 
\begin{align*} 
\C[\Gamma_j\backslash \PP^1(\Z_F/\frakN)] &\to \C[\Gamma_i\backslash\PP^1(\Z_F/\frakN)]\\
\Gamma_j x &\mapsto \sum_{\gamma\in \Theta(\frakp)_{i,j}} \Gamma_i (\gamma x)
\end{align*} 
on each component.  It is a nontrivial but nevertheless routine calculation that this Hecke module is isomorphic to the Hecke module $M_2^B(\frakN)$ defined by the Brandt matrices at the beginning of this section.

\begin{exm} We keep the notations of Example~\ref{defexmqsrt5}. The quaternion algebra $B$ has class number 1, thus the maximal order $\calO_0(1)$ is unique up to conjugation. We have
$\Gamma=\calO_0(1)^\times/\Z_F^\times$ has cardinality $60$. We consider the splitting map
\[\bar{\iota}_\frakN : \calO_0(1)\to\calO_0(1) \otimes_{\Z_F}(\Z_F/\frakN) \cong \M_2(\Z_F/\frakN)\]
given by
\[\bar{\iota}_\frakN(i)=\begin{pmatrix}11& 0\\ 37& 50\end{pmatrix},\,\,\bar{\iota}_\frakN(k)=\begin{pmatrix}47& 58\\ 18& 33\end{pmatrix}, \,\,\bar{\iota}_\frakN(ik)=\begin{pmatrix}29& 28\\ 16& 14\end{pmatrix}
.\] (One directly verifies that $\bar{\iota}_\frakN(ik)=\bar{\iota}_\frakN(i)\bar{\iota}_\frakN(k)$.) We let $\Gamma$ act on $\PP^1(\Z_F/\frakN)$ on the left via $\bar{\iota}_\frakN$. By the above discussion, we have
\[M_2^B(\frakN)\cong \C[\Gamma\backslash\PP^1(\Z_F/\frakN)].\]
The action of $\Gamma$ on $\PP^1(\Z_F/\frakN)$ has three orbits which are represented by $x_1=(1:0)$, $x_2=(1:1)$ and $x_3=(23:1)$ whose stabilizers have cardinality $e_1=2$, $e_2=5$ and $e_3=3$.  Thus $M_2^B(\frakN)$ is a free module generated by the orbits $\Gamma x_1$, $\Gamma x_2$ and $\Gamma x_3$. Writing down the Hecke action in that basis, we obtain the same Hecke operators as in Example~\ref{defexmqsrt5}.
\end{exm}


\begin{rmk} 
The approach presented above has some advantages over the usual definition of Brandt matrices above.  First of all, it is better suited for working with more general level structures, such as those that do not come from Eichler orders.  For example, adding a character in this context is quite transparent.  Secondly, when working over the same number field, a substantial amount of the required data can be precomputed and reused as the level varies, and consequently one gains significantly in the efficiency of the computation.
\end{rmk}

\section{Indefinite method}

In this section, we discuss a method for computing Hilbert modular forms using a indefinite quaternion algebra $B$ with $r=1$.  We accordingly call our method the \emph{indefinite method}.  The method is due to Greenberg and the second author \cite{GreenbergVoight}.  We seek to generalize the method of modular symbols by working with (co)homology.  We continue to suppose that $F$ has strict class number $1$, and we assume that $B \not\cong \M_2(\Q)$ for uniformity of presentation.

Recall that in this case we have defined a group $\Gamma=\Gamma_0^B(\frakN) \subseteq  \GL_2^+(\R)$ such that $\PGamma=\Gamma/\Z_F^\times \subseteq  \PGL_2^+(\R)$ is discrete; the quotient $X=X_0^B(\frakN)=\Gamma \backslash \calH$ is a Shimura curve and quaternionic cusp forms on $B$ correspond to holomorphic differential $1$-forms on $X$.  Integration gives a Hecke-equivariant isomorphism which is the analogue of the Eichler-Shimura theorem, namely
\[ S_2^B(\frakN) \oplus \overline{S_2^B(\frakN)} \xrightarrow{\sim} H^1( X_0^B(\frakN), \C). \]
We recover $S_2^B(\frakN)$ by taking the $+$-eigenspace for complex conjugation on both sides.  Putting this together with the Jacquet-Langlands correspondence, we have
\[ S_2(\frakD\frakN)^{\frakD-\text{new}} \cong S_2^B(\frakN) \cong H^1( X^B_0(\frakN), \C)^+. \]

We have the identifications
\[ H^1(X,\C)=H^1(X^B_0(\frakN),\C) \cong H^1(\Gamma_0^B(\frakN),\C)=\Hom(\Gamma_0^B(\frakN),\C)=\Hom(\Gamma,\C). \]
To complete this description, we must relate the action of the Hecke operators.  Let $\frakp \nmid \frakD\frakN$ be prime and let $p$ be a totally positive generator of $\frakp$.  As in (\ref{Heckeindefinite}), let 
\begin{equation} \label{heckethetapclno1}
\Theta(\frakp)=\calO_{+}^\times \backslash \left\{\pi\in\calO_{+}:\nrd(\pi)=p\right\}
\end{equation}
and choose representatives $\pi_a$ for these orbits labeled by $a \in \PP^1(\F_\frakp)$.  Then any $\gamma \in \Gamma$ by right multiplication permutes the elements of $\Theta(\frakp)$, and hence there is a unique permutation $\gamma^*$ of $\PP^1(\F_\frakp)$ such that for all $a \in \PP^1(\F_\frakp)$ we have
\[ \pi_a \gamma = \delta_a \pi_{\gamma^* a} \]
with $\delta_a \in \Gamma$.  For $f \in H^1(\Gamma, \C)=\Hom(\Gamma,\C)$, we then define
\begin{equation} \label{Heckeindefiniteincohom}
(T_\frakp f)(\gamma)=\sum_{\pi_a \in \Theta(\frakp)} f(\delta_a).
\end{equation}
In a similar way, we compute the action of complex conjugation $T_\infty$ via the relation
\[ (T_\infty f)(\gamma) = f(\delta) \]
where $\mu \gamma = \delta \mu$ and $\mu \in \calO^\times \setminus \calO_1^\times$.

We begin with two examples, to illustrate the objects and methods involved.

\begin{exm}
Let $F=\Q(\sqrt{29})$.  Then $\Z_F=\Z[w]$ is the ring of integers of $F$ where $w=(1+\sqrt{29})/2$ satisfies $w^2-w-7=0$.  Indeed $F$ has strict class number $1$ and $u=w+2$ is a fundamental unit of $F$.  

Let $B=\quat{-1,u}{F}$, so $B$ is generated over $F$ by $i,j$ subject to $i^2=-1$, $j^2=u$, and $ji=-ij$.  The algebra $B$ is ramified at the prime ideal $2\Z_F$ and the nonidentity real place of $B$ (taking $\sqrt{29} \mapsto -\sqrt{29}$) and no other place.  The identity real place gives an embedding 
\begin{align*}
\iota_\infty: B &\hookrightarrow B \otimes_F \R \cong \M_2(\R) \\
i,j &\mapsto \begin{pmatrix} 0 & 1 \\ -1 & 0 \end{pmatrix}, \begin{pmatrix} \sqrt{u} & 0 \\ 0 & -\sqrt{u} \end{pmatrix} = \begin{pmatrix} 2.27... & 0 \\ 0 & -2.27... \end{pmatrix}
\end{align*}

Let 
\[ \calO=\calO_0(1)=\Z_F \oplus \Z_F i \oplus \Z_F j \oplus \Z_F k \] 
where $k=(1+i)(w+1+j)/2$.  Then $\calO$ is a maximal order of $B$ with discriminant $\frakD=2\Z_F$.  Let $\Gamma=\iota_\infty(\calO_1^\times) \subseteq \SL_2(\R)$ be as above and let $X = X(1)_\C = \Gamma \backslash \calH$ be the associated Shimura curve.  

Although they are not an intrinsic part of our algorithm, we mention that the area of $X$ (normalized so that an ideal triangle has area $1/2$) is given by
\[ A=\frac{4}{(2\pi)^{2n}} d_F^{3/2} \zeta_F(2) \Phi(\frakD) = \frac{4}{(2\pi)^4}\sqrt{29}^3 \zeta_F(2)(4-1) = \frac{3}{2} \]
where $\Phi(\frakD)=\prod_{\frakp \mid \frakD}(\N\frakp + 1)$, and the genus of $X$ is given by the Riemann-Hurwitz formula as
\[ A = 2g-2 + \sum_q e_q\left(1-\frac{1}{q}\right) \]
where $e_q$ is the number of elliptic cycles of order $q \in \Z_{\geq 2}$ in $\Gamma$.  An explicit formula for $e_q$ given in terms of class numbers and Legendre symbols yields $e_2=3$ and $e_q=0$ for $q>2$.  Thus $2g-2=0$, so $g=1$.  For more details on these formulas and further introduction, see work of the second author \cite{VoightShimClay} and the references given there.

Next, we compute a fundamental domain for $\Gamma$, yielding a presentation for $\Gamma$; we consider this as a black box for now.  The domain, displayed in the unit disc, is as follows.  

\begin{center}
\includegraphics{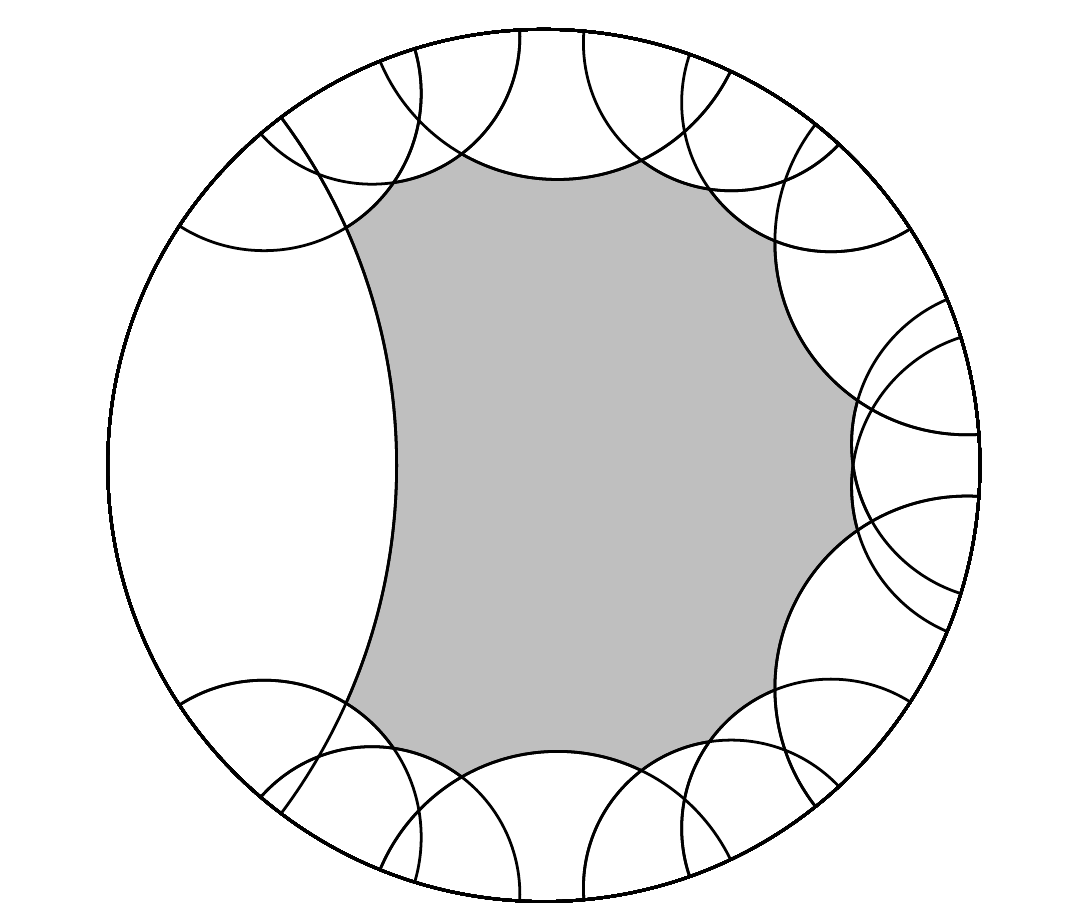}
\end{center}

We obtain the presentation
\[ \Gamma \cong \langle \gamma, \gamma', \delta_1, \delta_2, \delta_3 : \delta_1^2=\delta_2^2=\delta_3^2 = [\gamma,\gamma']\delta_1\delta_2\delta_3=1 \rangle \]
where 
\begin{align*}
\gamma &= -w-1+i-2j+k \\
\gamma' &= 2+2i+(w-1)j - (w-1) k \\
\delta_1 &= (2w+2) + wi+j+4k \\
\delta_2 &= i \\
\delta_3 &= (w+1)+(2w+3)i-j-k.
\end{align*} 

The above method gives the isomorphisms of Hecke modules 
\[ S_2(2\Z_F)^{\textup{new}} \cong S_2^B(1)=\{f : \calH \to \C : f(gz)\,d(gz) = f(z)\,dz \text{ for all $g \in \Gamma$}\} \cong H^1(\Gamma,\C)^+. \]
We compute that $H^1(\Gamma,\C)^+ \cong \Hom(\Gamma,\C)^+ = \C f$ where $f$ is the characteristic function of $\gamma$, i.e., $f(\gamma)=1$, $f(\gamma')=0$ and $f(\delta_i)=0$ for $i=1,2,3$.

We compute the Hecke operator $T_\frakp$ for $\frakp$ odd according the definition (\ref{Heckeindefiniteincohom}).  Let $\frakp=(w+1)\Z_F$.  Then $N\frakp = p = 5$.  We compute the action of $T_\frakp$ on $H^1(\Gamma,\C)^+$ given by $T_\frakp f = a_\frakp(f) f$.  The Hecke operators act as a sum over $p+1$ left ideals of reduced norm $p$, indexed by $\PP^1(\F_5)$.  Let
\begin{align*} 
\iota_\frakp : \calO &\hookrightarrow \M_2(\Z_{F,\frakp}) \cong \M_2(\Z_5) \\
i,j,ij &\mapsto \begin{pmatrix} 0 & 1 \\ -1 & 0 \end{pmatrix}, \begin{pmatrix} 1 & 0 \\ 0 & -1 \end{pmatrix}, \begin{pmatrix} 0 & -1 \\ -1 & 0 \end{pmatrix} \pmod{5}.
\end{align*}
Specifying the images modulo $\frakp$ gives them uniquely, as they lift to $\M_2(\Z_5)$ using Hensel's lemma.  (Note that $j^2=u \equiv 1 \pmod{w+1}$.)

Let $J_\infty,J_0,\dots,J_4$ be defined by
\[ J_a=J_{(x:y)} = \calO \iota_\frakp^{-1} \begin{pmatrix} x & y \\ 0 & 0 \end{pmatrix} + \frakp \calO \]
as in (\ref{Iaideals}).  Then $J_a=\calO \pi_a$ are principal left $\calO$-ideals by strong approximation.
For example, $J_0=\calO(i-ij) + (w+2)\calO = \calO \pi_0$ where $\pi_0=(-w+3)+wi+j+ij$.  

We compute the Hecke operators as in (\ref{Heckeindefiniteincohom}).  For $f:\Gamma \to \C$ and $\gamma \in \Gamma$, we compute elements $\delta_a \in \Gamma$ indexed by $a \in \PP^1(\F_\frakp)$ and $\gamma^*$ a permutation of $\PP^1(\F_\frakp)$ such that $\pi_a \gamma = \delta_a\pi_{\gamma^* a} $ for all $a \in \PP^1(\F_\frakp)$; then 
\[ (T_\frakp f)(\gamma) = \sum_{a \in \PP^1(\F_\frakp)} f(\delta_a). \]
The contribution to the sum for $f$ simply counts the number of occurrences of $\gamma$ in the product $\pi_a \gamma \pi_{\gamma^*a}=\delta_a \in \Gamma$.  Carrying out this computation for various primes, we obtain the following table.
\[
\begin{array}{c||ccccccccccc}
N\frakp & 5 & 7 & 9 & 13 & 23 & 29 & 53 & 59 & 67 & 71 & 83 \\
\hline
\rule{0pt}{2.5ex} 
a_\frakp(f) & 1 & -2 & 5 & -1 & -6 & 10 & -1 & 10 & 8 & -8 & 13 \\
N\frakp+1-a(\frakp) & 5 & 10 & 5 & 15 & 30 & 20 & 55 & 50 & 60 & 80 & 70
\end{array}  
\] 
Here we list only the norm of the prime as the eigenvalue does not depend on the choice of prime $\frakp$ of the given norm.  This suggests that $f$ corresponds to a base change of a form from $\Q$.

So we look through tables of elliptic curves over $\Q$ whose conductor is divisible only by $2$ and $29$.  We find the curve $E$ labelled \textsf{1682c1}, where $1682=2\cdot 29^2$, given by
\[ E: y^2+xy=x^3+x^2-51318x-4555676. \]
Let $E_F$ denote the base change of $E$ to $F$.  We compute that the twist $E_F'$ of $E_F$ by $-u\sqrt{29}$, given by
\[ E_F': y^2+(w+1)xy = x^3 + (-w+1)x^2+(-11w-20)x+(23w+52) \] 
has conductor $2\Z_F$.  Since the extension $F/\Q$ is abelian, by base change theorems we know that there exists a Hilbert cusp form associated to $E_F'$ over $F$ which is new of level $2\Z_F$, which therefore must be equal to $f$.  This verifies that the Jacobian $J(1)$ of $X(1)$ is isogenous to $E_F'$.  We verify that $\#E_F'(\F_\frakp)=N\frakp+1-a(\frakp)$ and (as suggested by the table) that $E$ has a $5$-torsion point, $(-1,-2w-5)$ (and consequently so too does $J(1)$).
 \end{exm}

\begin{exm} \label{indefexmqsrt5}
To illustrate the Jacquet-Langlands correspondence (Theorem \ref{shimizuclno1}) in action, we return to Example \ref{defexmqsrt5}.  Recall $F=\Q(\sqrt{5})$ and $w=(1+\sqrt{5})/2$.  We find the quaternion algebra $B=\quat{w,-(3w+7)}{F}$ which is ramified at $\frakN=(3w+7)\Z_F$, a prime of norm $61$, and one infinite place.  
The order $\calO=\calO_0(1)$ with $\Z_F$-basis $1,i,k,ik$, where $k=((w+1)+wi+j)/2$, is maximal.  As above, we compute that $A(X)=10$ and $g(X)=2=\dim S_2^B(1)=\dim S_2(\frakN)\spnew$.  Now we have the following fundamental domain:

\begin{center}
\includegraphics{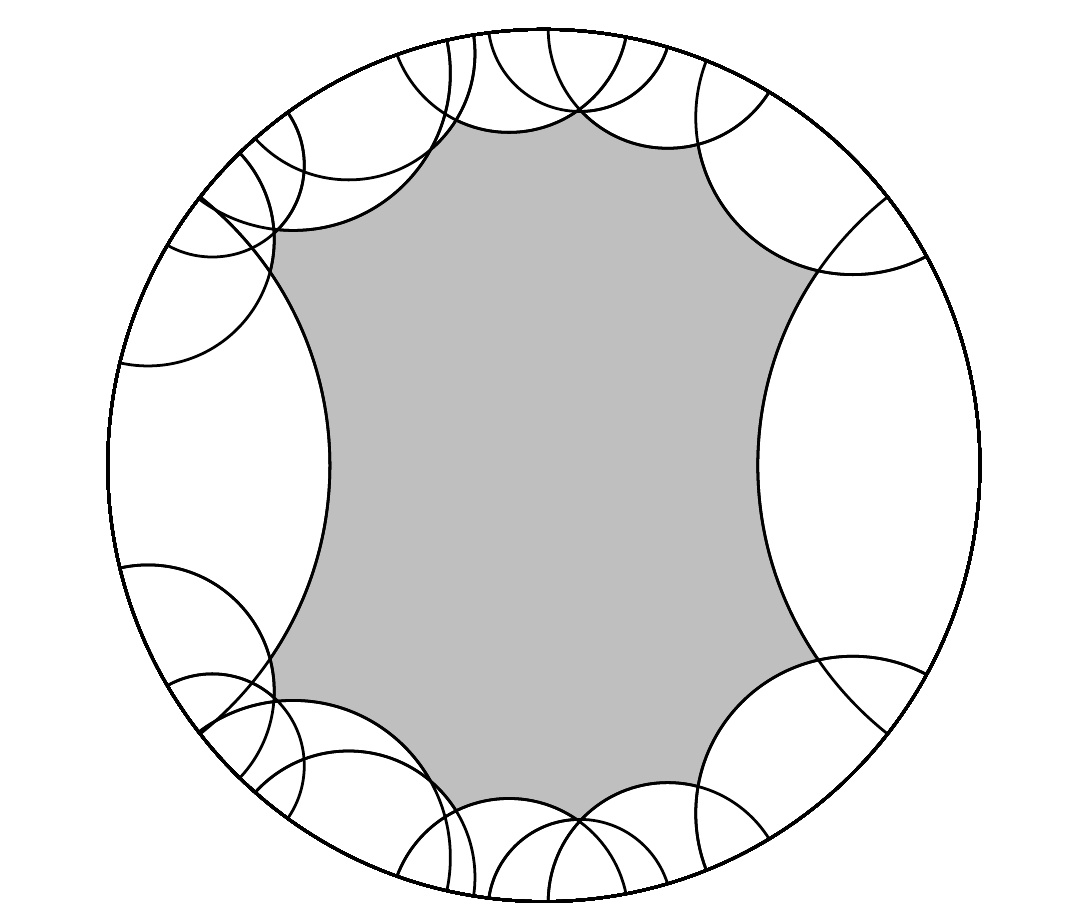}
\end{center}

We obtain the presentation 
\[ \Gamma(1) \cong \langle \gamma_1,\gamma_1',\gamma_2,\gamma_2' : [\gamma_1,\gamma_1'][\gamma_2,\gamma_2'] = 1 \rangle. \]
(In particular, $\Gamma(1)$ is a torsion-free group.)
We have $\dim H^1(X,\C)=4$, and on the basis of characteristic functions given by $\gamma_1,\gamma_1',\gamma_2,\gamma_2'$, the action of complex conjugation is given by the matrix $\begin{pmatrix} 1 & 1 & 0 & 0 \\ 0 & -1 & 0 & 0 \\ 0 & 1 & 0 & -1 \\ 0 & -1 & -1 & 0 \end{pmatrix}$: this is computed in a way 
We thus obtain a basis for $H=H^1(X,\C)^+ \cong S_2^B(1)$.  Computing Hecke operators as in the previous example, we find as in Example \ref{defexmqsrt5} that $S_2^B(1)$ is irreducible as a Hecke module, and find that
\begin{gather*} 
H \mid T_2 = \begin{pmatrix} 2 & 2 \\ -2 & -4 \end{pmatrix}, \quad 
H \mid T_{\sqrt{5}} = \begin{pmatrix} -5 & -3 \\ 3 & 4 \end{pmatrix}, \quad 
H \mid T_{3} = \begin{pmatrix} -4 & -1 \\ 1 & -1 \end{pmatrix}, \\
H \mid T_{w+3} = \begin{pmatrix} 6 & 4 \\ -4 & -6 \end{pmatrix}, \quad 
H \mid T_{w-4} = \begin{pmatrix} -2 & -1 \\ 1 & 1 \end{pmatrix}.
\end{gather*}
Happily, the characteristic polynomials of these operators agree with those computed using the definite method.
\end{exm}

We now give an overview of how these computations are performed: for more details, see the reference by Greenberg and the second author \cite{GreenbergVoight}.  To compute effectively the systems of Hecke eigenvalues in the cohomology of a Shimura curve, we need several algorithms.  First, we need to compute an explicit finite presentation of $\Gamma$ with a solution to the word problem in $\Gamma$, i.e., given $\delta \in \Gamma$, write $\delta$ as an explicit word in the generators for $\Gamma$.  Secondly, we need to compute a generator (with totally positive reduced norm) of a left ideal $I \subseteq \calO$.  The first of these problems is solved by computing a Dirichlet domain; the second is solved using lattice methods.  We discuss each of these in turn.

Let $p \in \calH$ have trivial stabilizer $\Gamma_p=\{g \in \Gamma : gp=p\}=\{1\}$.  The \emph{Dirichlet domain} centered at $p$ is the set
\[ D(p)=\{z \in \calH : d(z,p) \leq d(gz,p) \text{ for all $g \in \Gamma$}\} \]
where $d$ denotes the hyperbolic distance.  In other words, we pick in every $\Gamma$-orbit the closest points $z$ to $p$.  The set $D(p)$ is a closed, connected, hyperbolically convex fundamental domain for $\Gamma$ whose boundary consists of finitely many geodesic segments, called \emph{sides}, and comes equipped with a \emph{side pairing}, a partition of the set of sides into pairs $s,s^*$ with $s^*=g(s)$ for some $g \in \Gamma$.  (We must respect a convention when a side $s$ is fixed by an element of order $2$, considering $s$ to be the union of two sides meeting at the fixed point of $g$.)  

The second author has proven \cite{V-fd} that there exists an algorithm which computes a Dirichlet domain $D$ for $\Gamma$, a side pairing for $D$, and a finite presentation for $\Gamma$ with a minimal set of generators together with a solution to the word problem for $\Gamma$.  This algorithm computes $D$ inside the unit disc $\calD$, and we consider now $\Gamma$ acting on $\calD$ by a conformal transformation which maps $p \to 0$.  We find then that $\calD$ can be computed as an \emph{exterior domain}, namely, the intersection of the exteriors of \emph{isometric circles} $I(g)$ for elements $g \in \Gamma$, where $I(g)=\{z \in \C : |cz+d|=1\}$ where $g=\begin{pmatrix} a & b \\ c & d \end{pmatrix} \in SU(1,1)$.  We find generators for the group by enumerating short vectors with respect to a certain positive definite quadratic form depending on $p$.  

The solution to the word problem comes from a reduction algorithm that finds elements which contribute to the fundamental domain and removes redundant elements.  Given a finite subset $G \subset \Gamma \setminus \{1\}$   we say that $\gamma$ is \emph{$G$-reduced} if for all $g \in G$, we have   $d(\gamma 0,0) \leq d(g\gamma 0, 0)$.     We have a straightforward algorithm to obtain a $G$-reduced element, which we denote $\gamma \mapsto \red_G(\gamma)$:   if $d(\gamma 0,0) > d(g\gamma 0,0)$ for some $g \in G$,   set $\gamma := g\gamma$ and repeat.  When the exterior domain of $G$ is a fundamental domain, we have $\red_G(\gamma)=1$ if and only if $\gamma \in \Gamma$.  This reduction is analogous to the generalized division algorithm in a polynomial ring over a field.  

\begin{exm}
Let $F$ be the (totally real) cubic subfield of $\Q(\zeta_{13})$ with discriminant $d_F=169$.     We have $F=\Q(b)$ where $b^3+4b^2+b-1=0$.  ($F$ has strict class number $1$.)  

The quaternion algebra $B=\quat{-1,b}{F}$ has discriminant $\frakD=(1)$ and is ramified at $2$ of the $3$ real places of $F$.  We take $\calO$ to be an Eichler order of level $\frakp=(b+2)$, a prime ideal of norm $5$; explicitly, we have
\[ \calO=\Z_F \oplus (b+2) i \Z_F \oplus \frac{b^2+(b+4)i+j}{2} \Z_F  \oplus \frac{b+(b^2+4)\alpha+\alpha\beta}{2} \Z_F. \]

We compute a fundamental domain for the group $\Gamma=\Gamma_0^{(1)}(\frakp)=\iota_\infty(\calO_1^*)/\{\pm 1\}$.  We take $p=9/10 i \in \calH$.

We first enumerate elements of $\calO$ by their absolute reduced norm.  Of the first $260$ elements, we find $29$ elements of reduced norm $1$.

\begin{center}
\includegraphics[width=2.5in]{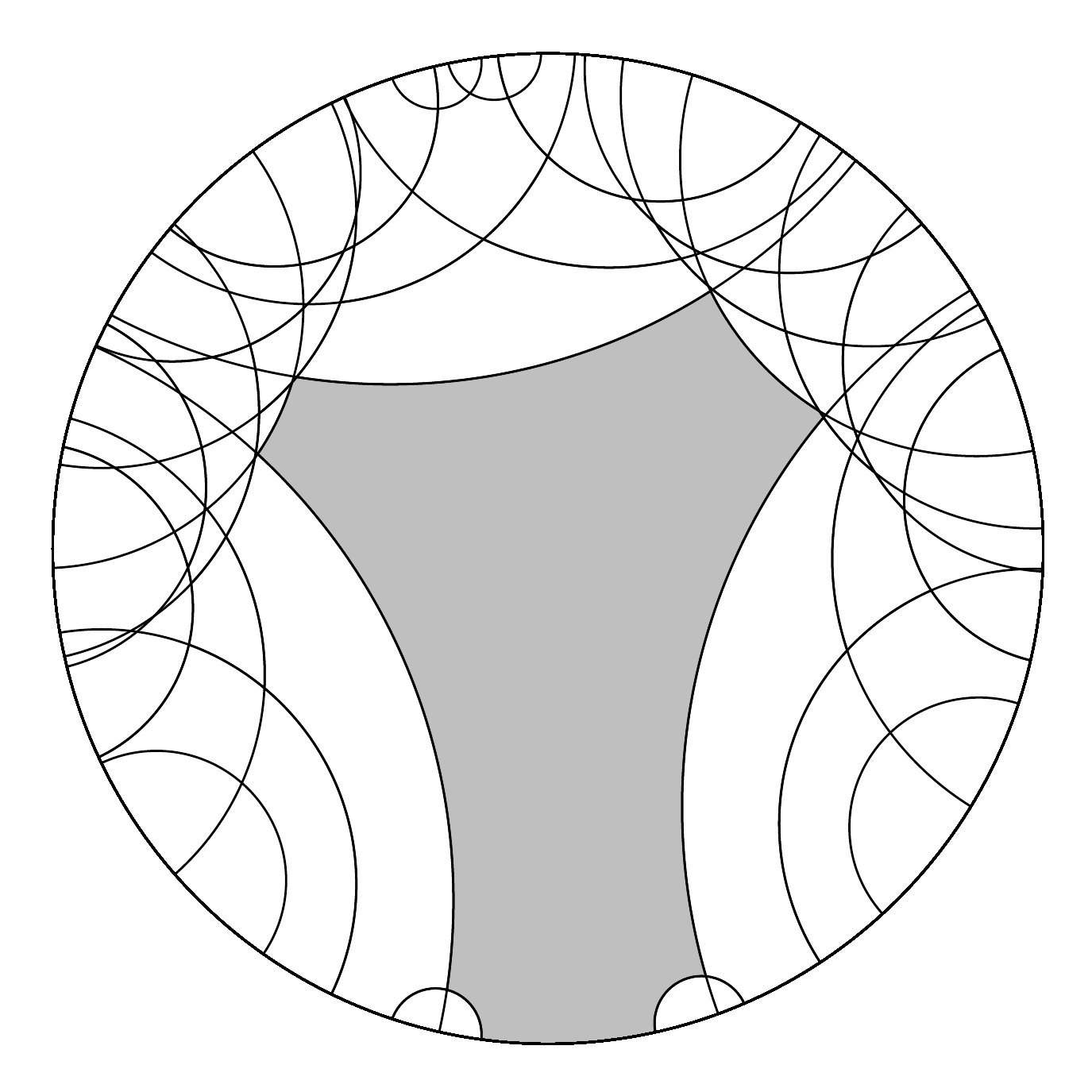}
\end{center}

Let $G$ be the set of elements which contribute to the boundary.  For each $g \not \in G$,   we compute $\red_G(g)$.  Each in fact reduces to $1$, so we are left with $8$ elements.

\vfill\newpage

\begin{center}
\includegraphics[width=2.5in]{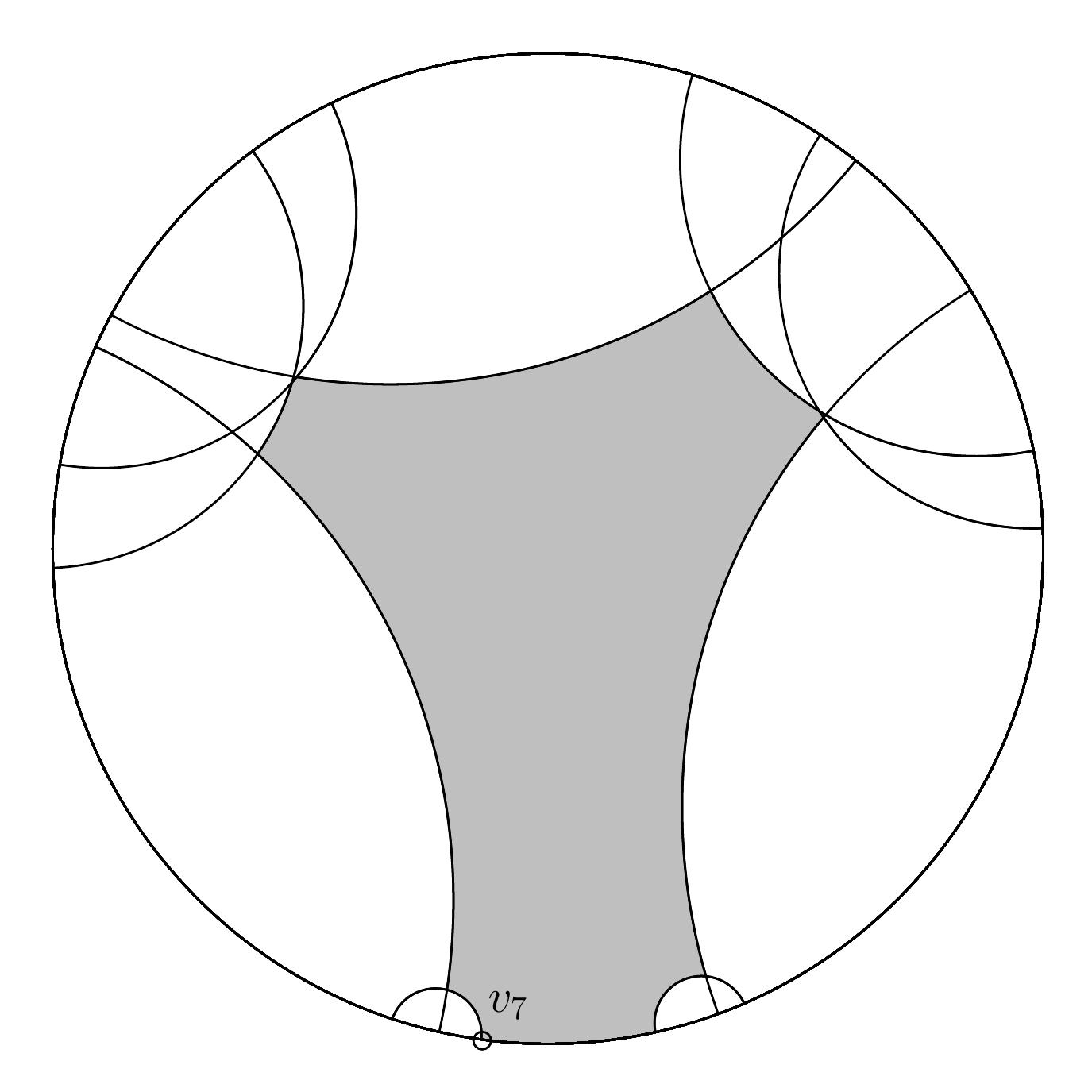}
\end{center}

We next enumerate elements in $\calO$ moving the center in the direction of the infinite vertex $v_7$.  We find an \emph{enveloper}, an element $g$ such that $v_7$ lies in the interior of $I(g)$, and reduce to obtain the following.

\begin{center}
\includegraphics[width=2.5in]{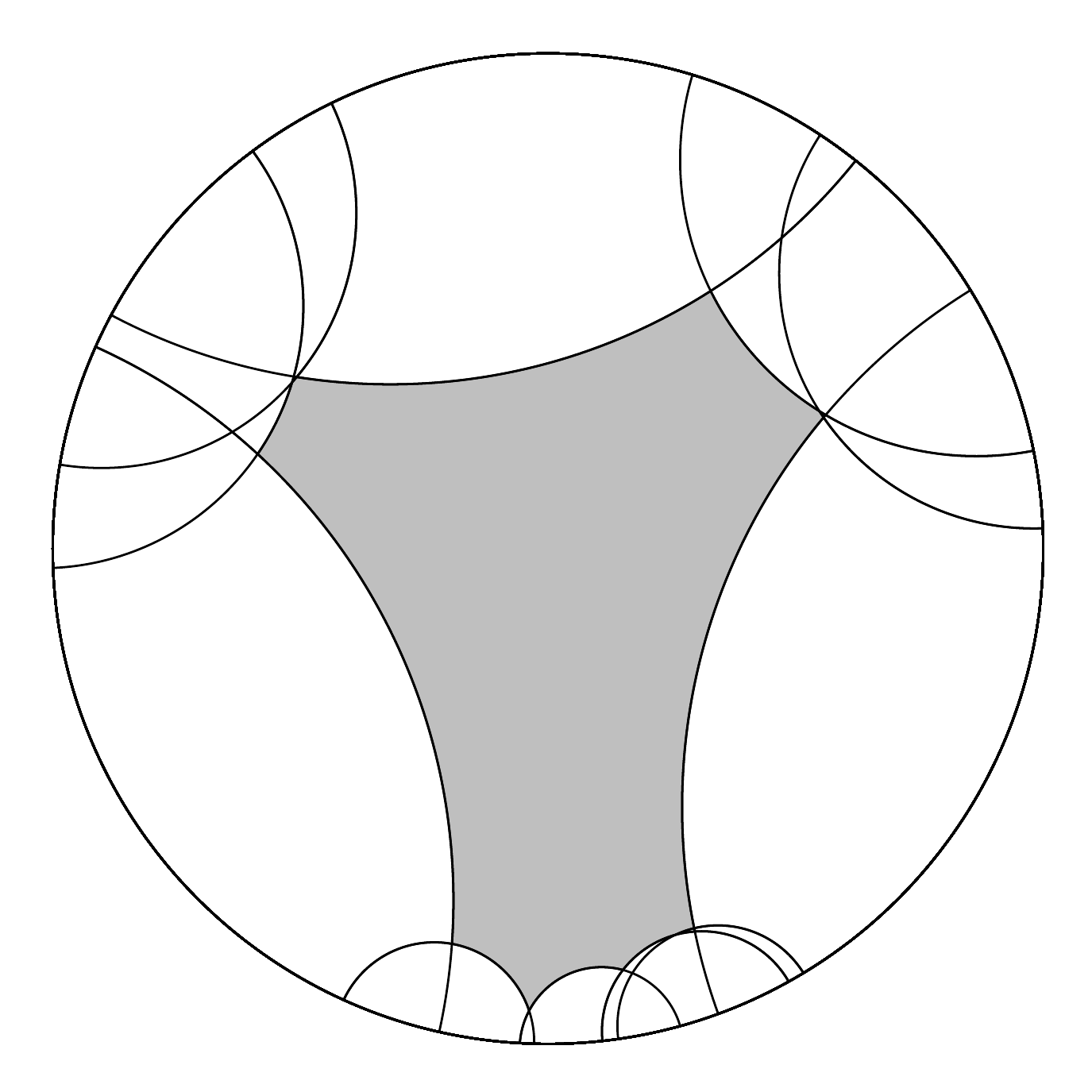}
\vspace{-0.25in}
\end{center}

The domain now has finite area.  We now attempt to pair each vertex to construct a side pairing.  

\begin{center}
\includegraphics[width=2.5in]{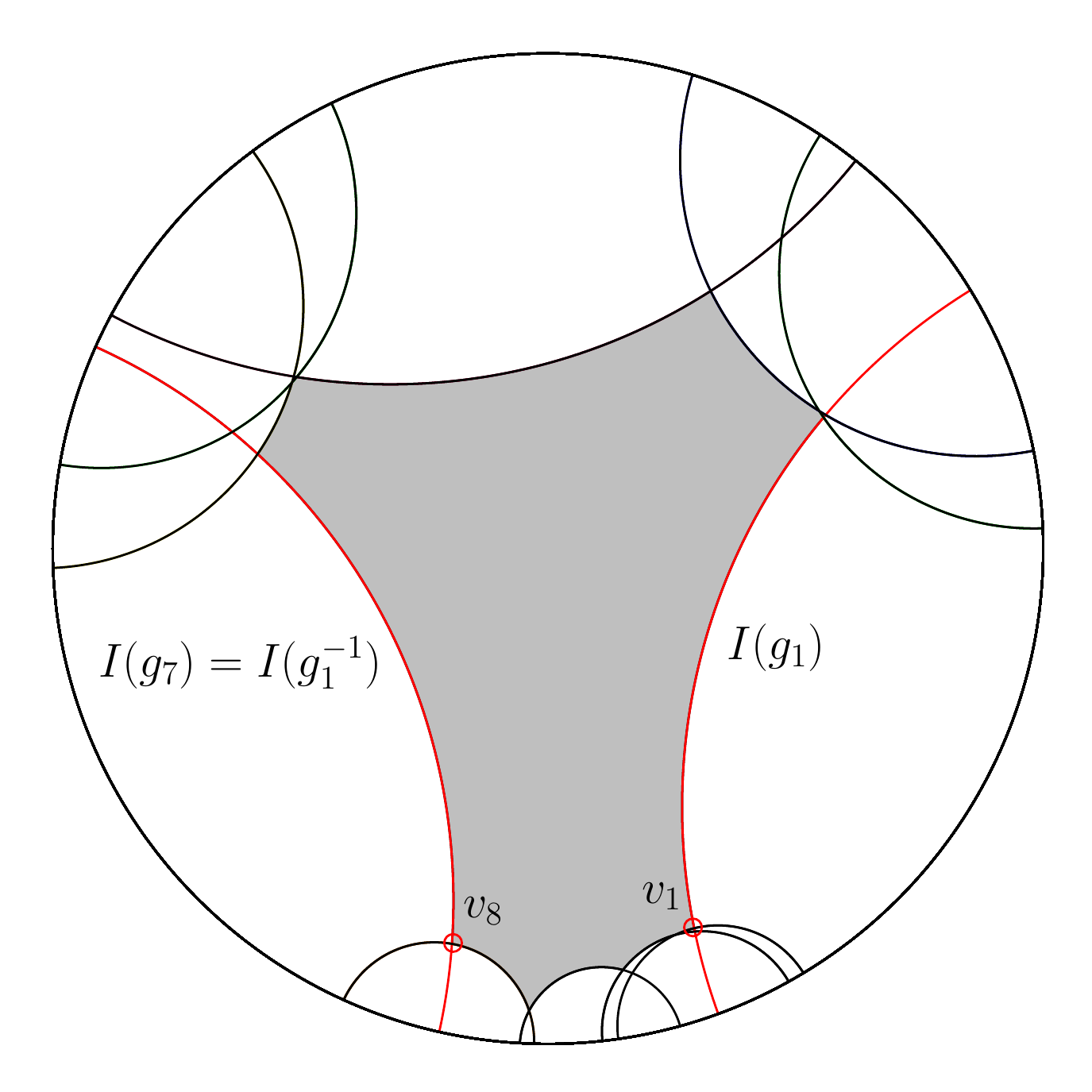}  
\vspace{-0.25in}
\end{center}

For example, the first vertex $v_1$ pairs with $v_8$, pairing the isometric circle $I(g_1)$ of $g_1$ with $I(g_7)$.  We continue, but find that $v_9$ does not pair with another vertex.

\begin{center}
\includegraphics[width=2.5in]{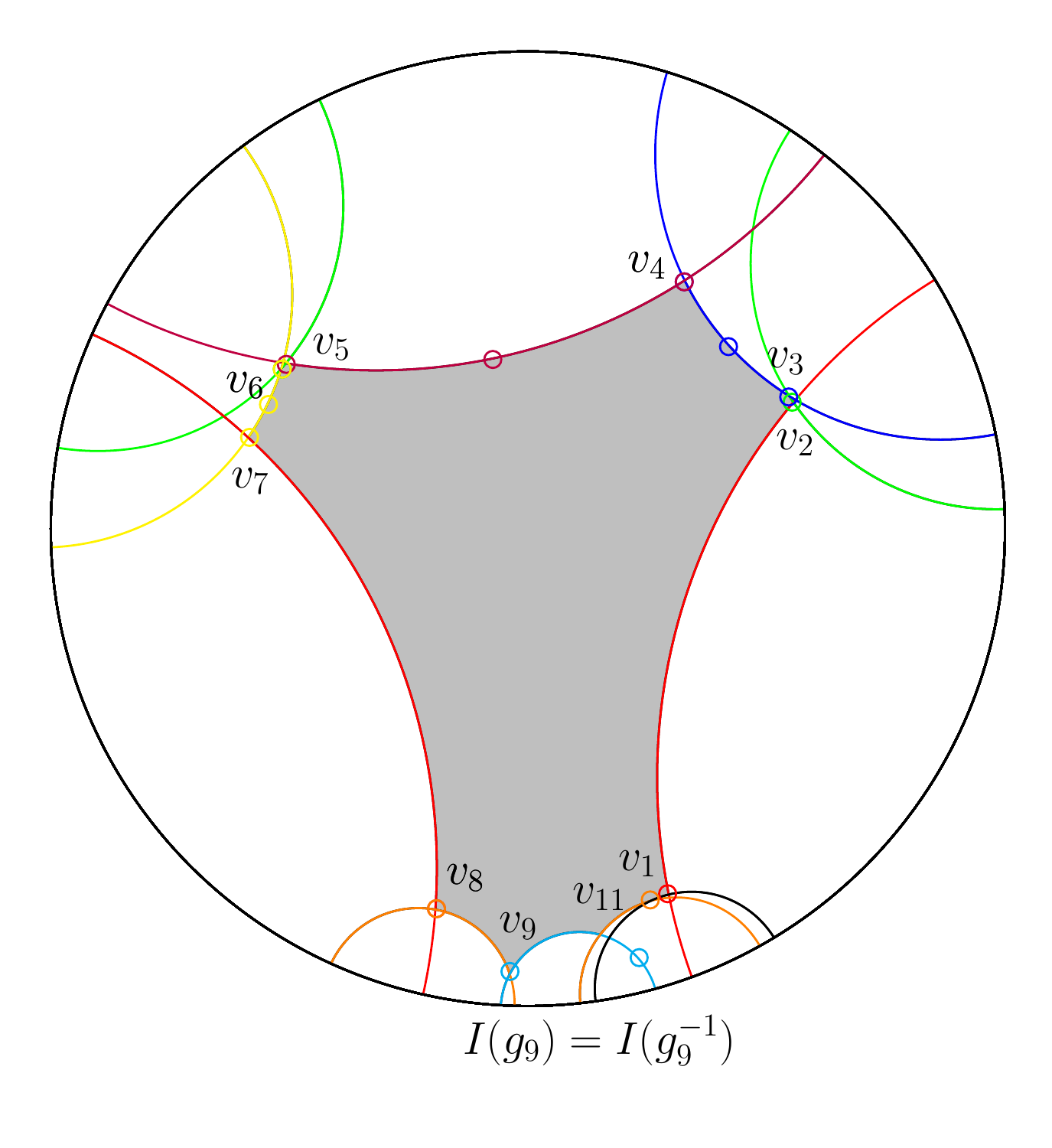}
\end{center}

Indeed, $g_9(v_9)$ does not lie in the exterior domain.  We compute the reduction $\red_G(g_9;v_9)$, analogously defined.  We then obtain a domain with the right area, so we are done!

\begin{center}
\includegraphics[width=2.5in]{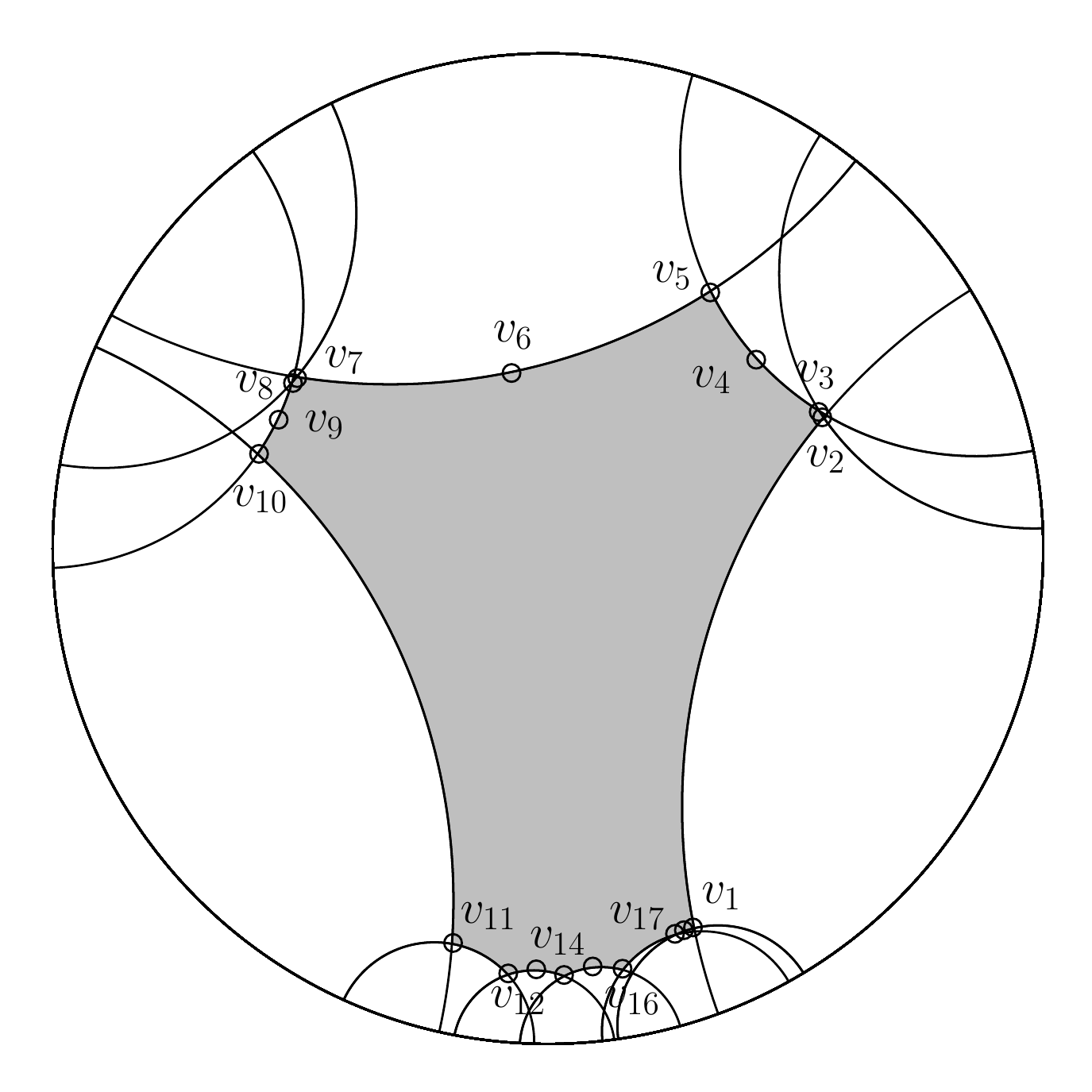}  
\end{center}
\end{exm}

(In order to get an honest side pairing, we agree to the convention that a side $I(g)$ which is fixed by an element $g$, necessarily of order $2$, is in fact the union of two sides which meet at the unique fixed point of $g$; the corresponding vertex then appears along such an side.)

As to the second problem, that of principalizing ideals, we refer to work of Kirschmer and the second author \cite[\S 4]{KirschmerVoight}, which relies again upon enumeration in a lattice.  This method can be improved using the computation of a fundamental domain, which allows one to further effectively reduce lattice elements.

\medskip

As in the previous section, when the level $\frakN$ varies, by changing the coefficient module one can work always with the group $\Gamma(1)$ associated to the maximal order $\calO_0(1)$ rather than recomputing a fundamental domain for $\Gamma=\Gamma_0(\frakN)$ each time.  This is a simple application of Shapiro's lemma: the isomorphism
\[ H^1(\Gamma,\C) \cong H^1(\Gamma(1), \Coind_{\Gamma}^{\Gamma(1)} \C), \]
is an isomorphism of Hecke modules: we give $\Coind_{\Gamma}^{\Gamma(1)} \C = \Hom(\Gamma(1)/\Gamma, \C)$ the natural structure as a $\Gamma(1)$-module, and for a cochain $f \in H^1(\Gamma, \Coind_{\Gamma}^{\Gamma(1)} \C)$, we define 
\begin{equation} 
(T_\frakp f)(\gamma)=\sum_{\pi_a \in \Theta(\frakp)} f(\delta_a)
\end{equation}
as in (\ref{Heckeindefiniteincohom}).  We require still that $\Theta(\frakp) \subset \calO_+$; after enumerating a set of representatives for $\Gamma \backslash \Gamma(1)$, we can simply multiply elements $\pi_a \in \calO_0(1)$ by the representative for its coset.

\medskip

To conclude, we compare the above method to the method of modular symbols introduced when $F=\Q$. The Shimura curves $X=X_0^B(\frakN)$ do not have cusps, and so the method of modular symbols does not generalize directly.    However, the side pairing of a Dirichlet domain for $\Gamma$ gives an explicit characterization of the gluing relations which describe $X$ as a Riemann surface,   hence one obtains a complete description for the homology group $H_1(X,\Z)$.    Paths are now written $\{v,\gamma v\}$ for $v$ a vertex on a side paired by $\gamma \in G$.     The analogue of the Manin trick in the context of Shimura curves is played by the solution to the word problem in $\Gamma$, and so in some sense this can be seen as a partial extension of the Euclidean algorithm to totally real fields.  Our point of view is to work dually with cohomology, but computationally these are equivalent \cite[\S 6]{GreenbergVoight}.

\section{Examples}\label{sec:examples-1}

In this section, we provide some detailed examples to illustrate applications of the algorithms introduced in the previous sections.

Our first motivation comes from the association between Hilbert modular forms and abelian varieties.  The Eichler-Shimura construction (see Knapp~\cite[Chap. XII]{knapp} or Shimura \cite[Chapter 7]{shimura3}) attaches to a classical newform $f(z)=\sum_n a_n q^n \in S_2(N)^{\textup{new}}$ an abelian variety $A_f$ with several properties.  First, $A_f$ has dimension equal to $[E_f:\Q]$ where $E_f=\Q(\{a_n\})$ is the field of Fourier coefficients of $f$ and $\End_{\Q}(A_f) \otimes \Q \cong E_f$.  Second, $A_f$ has good reduction for all primes $p \nmid N$.  Last, we have an equality of $L$-functions
\[ L(A_f,s) = \prod_{\sigma \in \Hom(E_f,\C)} L(f^{\sigma},s) \]
where $f^\sigma$ is obtained by letting $\sigma$ act on the Fourier coefficients of $f$, so that
\[ L(f^{\sigma},s) = \sum_{n=1}^{\infty} \frac{a_n^{\sigma}}{n^s}. \]
The abelian variety $A_f$ arises as the quotient of $J_0(N)$ by the ideal of $\End J_0(N)$ generated by $T_n - a_n$. 

The simplest example of this construction is the case of a newform $f$ with rational Fourier coefficients $\Q(\{a_n\})=\Q$.  Then $A_f$ is an elliptic curve of conductor $N$ obtained analytically via the map
\begin{align*}
X_0(N) &\to A_f \\
\tau &\mapsto z_\tau = 2\pi i\int_{i\infty}^{\tau} f(z)\,dz = \sum_{n=1}^{\infty} \frac{a_n}{n}e^{2\pi i\tau}.
\end{align*}
The equality of $L$-functions is equivalent to the statement $\# A_f(\F_p) = p+1-a_p$ for all primes $p$, and this follows in the work of Eichler and Shimura by a comparison of correspondences in characteristic $p$ with the Frobenius morphism.

The Eichler-Shimura construction extends to the setting of totally real fields in the form of the following conjecture (see e.g.\ Darmon \cite[Section 7.4]{Darmon}).

\begin{conj}[Eichler-Shimura]\label{conj1} 
Let $f \in S_2(\frakN)^{\textup{new}}$ be a Hilbert newform of parallel weight $2$ and level $\frakN$ over $F$.  Let $E_f=\Q(\{a_{\frakn}\})$ be the field of Fourier coefficients of $f$.  Then there exists an abelian variety $A_f$ of dimension $g=[E_f:\Q]$ defined over $F$ of conductor $\frakN^g$ such that $\End_{F}(A_f)\otimes\Q \cong E_f$ and 
\[ L(A_f,s)=\prod_{\sigma\in\Hom(E_f,\C)}L(f^\sigma,s). \]
\end{conj}

In particular, if $f$ has rational Fourier coefficients, Conjecture \ref{conj1} predicts that one can associate to $f$ an elliptic curve $A_f$ defined over $F$ with conductor $\frakN$ such that $L(A_f,s)=L(f,s)$.  

Conjecture \ref{conj1} is known to be true in many cases.  It is known if $f$ appears in the cohomology of a Shimura curve (as in Sections 3 and 5 above; see also Zhang~\cite{zhang}): this holds whenever $n=[F:\Q]$ is odd or if $\frakN$ is exactly divisible by a prime $\frakp$.  If $f$ appears in the cohomology of a Shimura curve $X$ and $f$ has rational Fourier coefficients, then when the conjecture holds there is a morphism $X \to A_f$ and we say that $A_f$ is \emph{uniformized} by $X$.  

However, Conjecture \ref{conj1} is not known in complete generality: by the Jacquet-Langlands correspondence, one has left exactly those forms $f$ over a field $F$ of even degree $n$ and \emph{squarefull} level (i.e.\ if $\frakp \mid \frakN$ then $\frakp^2 \mid \frakN$), for example $\frakN=(1)$.  One expects still that the conjecture is true in this case: Blasius~\cite{blasius}, for example, shows that the conjecture is true under the hypothesis of the Hodge conjecture.

The converse to the Eichler-Shimura construction (Conjecture \ref{conj1}) over $\Q$ is known as the Shimura-Taniyama conjecture.  An abelian variety $A$ over a number field $F$ is \emph{of $\GL_2$-type} if $\End_F A \otimes \Q$ is a number field of degree $\dim A$.  The Shimura-Taniyama conjecture states that given an abelian variety $A$ of $\GL_2$-type over $\Q$, there exists an integer $N\ge 1$ and a surjective morphism $J_0(N)\to A$.  This conjecture is a theorem, a consequence of the proof of Serre's conjecture by Khare and Wintenberger \cite{khare1}.  In the setting of totally real fields, the analogous conjecture is as follows.  

\begin{conj}[Shimura-Taniyama]\label{conj2} 
Let $A$ be an abelian variety of $\GL_2$-type over a totally real field $F$.  Then there exists a Hilbert newform $f$ of parallel weight $2$ such that 
$E_f\cong \End_F(A)\otimes\Q$ and 
\[ L(A,s)=\prod_{\sigma\in\Hom(E_f,\C)} L(f^\sigma,s).\]
\end{conj}

If both of these conjectures are true, then by Faltings' isogeny theorem, the abelian variety $A$ in Conjecture \ref{conj2} would be isogenous to the abelian variety $A_f$ constructed in Conjecture \ref{conj1}.  An abelian variety of $\GL_2$-type over a totally real field $F$ which satisfies the conclusion of Shimura-Taniyama conjecture (Conjecture \ref{conj2}) is called \emph{modular}. 

There has been tremendous progress in adapting the techniques initiated by Wiles~\cite{wiles}---which led to the proof of the Shimura-Taniyama conjecture in the case $F=\Q$---to totally real fields.  For example, there are modularity results which imply that wide classes of abelian varieties of $\GL_2$-type are modular: see for example work of  Skinner and Wiles~\cite{skinner-wiles1, skinner-wiles2}, Kisin~\cite{kisin1, kisin2}, Snowden~\cite{snowden} and Geraghty~\cite{geraghty}, and the results below.   However, a complete proof of Conjceture \ref{conj2} remains elusive.

\medskip

With these conjectures in mind, we turn to some computational examples.  We begin with $F=\Q(\sqrt{5})$ and let $w=(1+\sqrt{5})/2$.  For some further discussion on this case, see work of the first author \cite{Dembelesqrt5}.

\begin{exm} \label{le30}
We consider the spaces $S_2(\frakN)$ with $N(\frakN) \leq 30$.  

We compute using the definite method of Section 4 using the quaternion algebra over $F$ which is ramified at no finite place.  We find that $\dim S_2(\frakN)=0$ for all ideals $\frakN$ of $\Z_F$ with $N(\frakN) \leq 30$ and $\dim S_2(\frakN)=1$ for $\frakN$ with norm $31$.  

Now let $\frakN$ be either of the primes above 31.  By the Jacquet-Langlands correspondence (Theorem \ref{shimizuclno1}) and the accompanying discussion, the space $S_2(\frakN)$ can also be computed using the indefinite quaternion algebra $B$ ramified at $\frakN$ and one real place.  We find agreeably that the Shimura curve $X_0^B(1)$ has genus one and thus its Jacobian is an elliptic curve.  By a search, we find the elliptic curve 
\[ A:\,y^2+xy+w y=x^3-(1+w)x^2 \] 
of conductor $\frakN=(5+2w)$ of norm $31$.  

We now show that $A$ is modular.  In brief, we verify that the Galois representation $\overline{\rho}_{A,3}:\Gal(\overline{F}/F) \to \GL_2(\F_3) \subseteq \GL_2(\F_9)$ is irreducible, solvable, and ordinary since $\#A(F_9)=8$ (so $3 \nmid a_3(A)=9+1-8=2$), so $A$ is modular by Skinner and Wiles \cite{skinner-wiles1}.  Since $A$ has conductor $\frakp$ and the space $S_2(\frakp)$ has dimension $1$, it follows by Falting's isogeny theorem that $A$ is isogeneous to the Jacobian of the Shimura curve $X_0^B(1)$.  

In particular, the elliptic curves over $F=\Q(\sqrt{5})$ with smallest conductor have prime conductor  dividing $31$ and are isogenous to $A$ or its Galois conjugate.
\end{exm}

\begin{exm}\label{thm:ex-root-5} 
Let $\frakN \subseteq \Z_F$ be such that $N(\frakN) \le 100$.

If $\frakN \nmid 61$, then $\dim S_2(\frakN)^{\textup{new}} \leq 2$: each newform $f$ of conductor $\frakN$ has rational Fourier coefficients and is either the base change of a classical modular form over $\Q$ or associated to an elliptic curve which is uniformized by a Shimura curve.  

We compute each space in turn using the definite method.  Those forms $f$ which have the property that $a_{\frakp}(f)=a_{\overline{\frakp}}(f)$, where $\overline{\phantom{x}}$ denotes the Galois involution, are candidates to arise from base change over $\Q$.  For each of these forms $f$, we find a candidate classical modular form $g$ using the tables of Cremona \cite{cremonatables}: such a form will have conductor supported at $5$ and the primes dividing $N(\frakN)$.  Each curve comes with a Weierstrass equation, and we find a quadratic twist of the base change of this curve to $F$ which has conductor $\frakN$.  Since every elliptic curve over $\Q$ is modular, the base change and its quadratic twist are also modular, and since we exhaust the space this way we are done.

For example, when $\frakN=(8)$ we have $\dim S_2(8)^{\textup{new}}=1$ and find the form $f$ with Hecke eigenvalues $a_{(w+2)}=-2$, $a_{(3)}=2$, $a_{(w+3)}=a_{(w+7)}=-4$, $a_{(w-5)}=a_{(w+4)}=4$, and so on.  In the tables of Cremona, we locate the form 
\[ g(q)=q-2q^3-2q^7+q^9-4q^{11}-4q^{13}-4q^{19}+\dots \in S_2(200) \]
associated to an elliptic curve $A_g$ with label \textsf{200b} and equation $y^2=x^3+x^2-3x-2$.  We find that the quadratic twist of (the base change of) $A_g$ by $-(w+2)=-\sqrt{5}$ given by the equation $y^2 = x^3 + (w - 1)x^2 - wx$ has conductor $(8)$, so we conclude that $f$ is the base change of $g$.  Note that this example is not covered by the known cases of the Eichler-Shimura construction, since the level $(8)=(2)^3$ is squarefull.
\end{exm}

\begin{exm}
Now consider the prime level $\frakN=(3w+7) \mid 61$, then $S_2(\frakN)^{\textup{new}}$ is a $2$-dimensional irreducible Hecke module arising from the Jacobian of the Shimura curve $X_0^B(1)$ where $B$ is the quaternion algebra ramified at one of the infinite places and $\frakN$.  This follows from Examples \ref{defexmqsrt5} and \ref{indefexmqsrt5}.

One is naturally led to ask, can one identify the genus 2 Shimura curve $X=X_0^B(1)$ with $B$ of discriminant $\frakN$?

Consider the hyperelliptic curve $C:y^2+ Q(x)y=P(x)$ over $F$ with 
\begin{align*}
P(x) &= -w x^4 + (w - 1)x^3 + (5w + 4)x^2 + (6w + 4)x + 2w + 1\\
Q(x) &= x^3 + (w - 1)x^2 + w x + 1.
\end{align*}
This curve was obtained via specialization of the Brumer-Hashimoto family \cite{brumer2,Hashimoto} of curves (see also Wilson \cite{wilson}) whose Jacobian has real multiplication by $\Q(\sqrt{5})$, so in particular so does the Jacobian $J=\Jac(C)$ of $C$.  The discriminant of $C$ is $\disc(C)=\frakN^2$.  One can show that $J$ is modular using the theorem of Skinner and Wiles \cite[Theorem A]{skinner-wiles2} and the fact that $J$ has a torsion point of order 31.  Since $C$ is hyperelliptic of genus $2$, we compute that the conductor of (the abelian surface) $J$ is $\frakN^2$ and that its level is $\frakN$.  Since $J$ is modular of level $\frakN$, it corresponds to the unique Hecke constituent of that level. It follows then by Faltings' isogeny theorem that $J$ and the Jacobian of the Shimura curve $X$ are isogenous.
\end{exm} 

These examples exhaust all modular abelian varieties of $\GL_2$-type over $\Q(\sqrt{5})$ with level $\frakN$ of norm $N(\frakN)\le 100$ (up to isogeny).

We conclude with some (further) cases where Conjectures \ref{conj1} and \ref{conj2} are not known.

\begin{exm}
The first case of squarefull level with a form that is not a base change is level $\frakN=(w+36)=(w+3)^2$ so $N(\frakN)=121$.  We have $\dim S_2(w+36)=1$ and the corresponding newform $f$ has the following eigenvalues (where $t^2=3$):
\[ 
\begin{array}{c||cccccccccc}
\frakp & (2) & (w+2) & (3) & (w+3) & (w-4) & (w-5) & (w+4) & (w+5) & (w-6)\\ 
N\frakp& 4& 5 &9& 11 & 11 & 19 & 19 & 29 & 29\\\hline 
\rule{0pt}{2.5ex} 
a_\frakp(f)& t & -t & -1 & 0 & -2t & 4t & -2 & -5t & -3\\ 
\end{array} \]
In this case, the Eichler-Shimura conjecture predicts the existence of an abelian surface defined over $F=\Q(\sqrt{5})$ with good reduction away from $\frakp=(w+3)$ and with real multiplication by $\Q(\sqrt{3})$.  The form $f$ is not a base change from $\Q$ since $a_\frakp(f) \neq a_{\overline{\frakp}}(f)$ in general.  Can one find an explicit genus $2$ curve over $F$, analogous to the previous example, with the $L$-function of its Jacobian given by the above Frobenius data?
\end{exm}

\begin{exm}
We next consider a situation when we are successful in establishing the correspondence in a case which is not covered by known results.  Consider $\frakN=(17w-8)=(w+4)^2$ so $N(\frakN)=361$.  There exist forms $f,g \in S_2(\frakN)$ with Hecke eigenvalues as follows:
\[ 
\begin{array}{c||ccccccccc}
\frakp & (2) & (w+2) & (3) & (w+3) & (w-4) & (w-5) & (w+4) & (w+5) & (w-6)\\\hline
N\frakp& 4& 5 &9& 11 & 11 & 19 & 19 & 29 & 29\\ 
\rule{0pt}{2.5ex} 
a_\frakp(f)& 2 & -3 & 1 & 3 & 3 & -1 & 0 & 3 & -6\\
a_\frakp(g)& -2 & -3 & -1 & -3 & -3 & -1 & 0 & 3 & 6\\ 
\end{array} \]
The form $g$ is a quadratic twist of $f$ by $w(w+4)$, and the forms $f,g$ are not base changes from $\Q$.  

We use the method of Cremona and Lingham \cite{cremling}---without an attempt to be exhaustive---to find an elliptic curve $E$ over $F$ with good reduction away from $\frakp=(w+4)$.  We find the curve
\[ E:y^2 + (w + 1)y = x^3 + wx^2 - x - w \]
(which could also be found by a naive search) with conductor $\frakN$.  We verify that $\#E(\F_\frakp)=N\frakp+1-a_\frakp(f)$ for all primes $\frakp$ up to the limit of the computation.  We prove that $E$ is modular using the fact that $E$ has a $3$-isogeny, with kernel defined over $F'=F(\sqrt{w(w+4)})$ generated by the points with $x=w-1$, so we can apply the theorem of Skinner and Wiles.  Therefore indeed $E$ corresponds to $f$.
\end{exm}

\medskip

We now turn to the existence of elliptic curves with everywhere good reduction over totally real fields.  The Shimura-Taniyama conjecture predicts that such curves arise from cusp forms of level $(1)$.  

\begin{exm}
Let $F=\Q(\lambda)=\Q(\zeta_{13})^+$ be the totally real subfield of the cyclotomic field $\Q(\zeta_{13})$ with $\lambda=\zeta_{13}+\zeta_{13}^{-1}$.

The space $S_2(1)$ over $F$ has dimension 1, represented by the cusp form $f$.  Conjecture \ref{conj1} predicts the existence of an elliptic curve $A_f$ over $F$ with everywhere good reduction. We note that $A_f$ must be a $\Q$-curve since the twist $f^{\sigma}$ of $f$ by an element $\sigma \in \Gal(F/\Q)$ is again a form of level $1$ and so (since $\dim S_2(1)=1$) $f^{\sigma}=f$ and consequently $A_f$ must be isogenous its Galois conjugates. 

Indeed, let 
\[ A:y^2+a_1xy+a_3y=x^3+a_2x^2+a_4x + a_6 \] 
be the elliptic curve with the following coefficients:
\begin{align*}
a_1&= \lambda^2 + \lambda + 1,\\
a_2&= -\lambda^3 + \lambda^2 + \lambda,\\
a_3&= \lambda^5 + \lambda^3 + \lambda,\\
a_4&= \lambda^5 - \lambda^4 - 11\lambda^3 - \lambda^2 + 3\lambda,\\
a_6&= 13\lambda^5 - 31\lambda^4 - 48\lambda^3 + 52\lambda^2 + 33\lambda - 10.
\end{align*}
We verify that $A$ has everywhere good reduction and that $\#A(F)_{\textup{tors}}=19$ with $A(F)_{\textup{tors}}$ generated by the point
\[ (\lambda^5 - \lambda^4 - 4\lambda^3 + 3\lambda^2 + 2\lambda - 1 ,
    -\lambda^5 + \lambda^4 + \lambda^2). \]
Therefore, by Skinner and Wiles~\cite[Theorem 5]{skinner-wiles1} applied to the $3$-adic representation attached to $A$ or~\cite[Theorem A]{skinner-wiles2} to the $19$-adic one, we conclude that $A$ is modular, and thus $A$ is our abelian variety $A_f$.

The elliptic curve $A$ was obtained from a curve over $\Q(\sqrt{13}) \subseteq F$ as follows.  There is a $2$-dimensional Hecke constituent $g$ in the space of classical modular forms $S_2(169)^{\rm new}$ over $\Q$ given by
\[ g(q)=q+\sqrt{3}q^2+2q^3+q^4-\sqrt{3}q^5+\ldots; \] 
let $A_g$ be the associated abelian surface, defined over $\Q$.  Roberts and Washington~\cite{rob-wash} showed that $A_g$ twisted by $\sqrt{13}$ is isomorphic over $F$ to $E \times E^\sigma$, where
\[ E: y^2=x^3+\frac{3483\sqrt{13}-12285}{2}+(74844\sqrt{13}-270270), \] 
a $\Q$-curve of conductor $(13)$ and $\sigma$ is the nontrivial automorphism of $\Q(\sqrt{13})$.  

The abelian variety $A_g$ is related to the elliptic curve $A$ as follows.  Roberts and Washington show that $J_1(13)$ and $A_g$ are isogenous over $\overline{\Q}$.  The abelian surface $J_1(13)$ is isogenous over $F$ to $A \times A$, since $J_1(13)$ obtains good reduction over $F$ (see Mazur and Wiles \cite[Chapter 3.2, Proposition 2]{mazur-wiles} and Schoof \cite{schoof}).  Indeed, Serre showed that the abelian surface $J_1(13)$ has a rational point of order 19 and Mazur and Tate showed that $J_1(13)$ twisted by $\Q(\sqrt{13})$ is a product of two elliptic curves.
We thereby determine an equation for $A$ by showing that it is a quadratic twist of $E$ by some unit in $F$.  
We would like to thank Elkies and Watkins for pointing us to the rich history of this curve. 
\end{exm}

The next example gives the list of all modular elliptic curves (up to isogeny) with everywhere good reduction over a real quadratic field $F$ of discriminant $\le 1000$ and having strict class number one.

\begin{exm} 
Let $0 \leq D\leq 1000$ be a fundamental discriminant such that $F=\Q(\sqrt{D})$ has strict class number one.  We have $\Z_F=\Z[w]$ with $w=\sqrt{D},(1+\sqrt{D})/2$, according as $D \equiv 0,1 \pmod{4}$.  Let $A$ be a modular elliptic curve over $F$ with everywhere good reduction.  Then $A$ corresponds to a cusp form $f$ of level $1$ with rational Fourier coefficients.  Computing the spaces $S_2(1)$, we find that either $A$ is a $\Q$-curve (so $A$ is isogenous to its Galois conjugate), or $D\in\{509, 853, 929, 997\}$ and the isogeny class of $A$ is represented by the curve with coefficients as below, up to Galois conjugacy:
\[ 
\begin{array}{c||ccccc}
D&a_1&a_2&a_3&a_4&a_6\\\hline
\rule{0pt}{2.5ex} 
509&-1&2+2w&-w&162+3w&71+34w\\
853&w&-1-w&0&-5921+429w&145066-9437w\\
929&w& 1 - w& 1 + 20w& -1738 - 82w& -11808 - 201w\\
997&0& w& 1& - 289-24w&- 2334 -144w\\
\end{array} \]
The $\Q$-curves over $F$ are listed in Cremona~\cite{cremona}: for example, the curve $y^2+xy+uy=x^3$ with $u=(5+\sqrt{29})/2$ is the smallest such and was discovered by Tate~\cite{serre4}.

The curve for discriminant $D=509$ was discovered by Pinch \cite[p.\ 415]{schoof}: it was the first example of an elliptic curve having everywhere good reduction which is not a $\Q$-curve, and it was proven to be modular by Socrates and Whitehouse \cite{SocratesWhitehouse} using the method of Faltings and Serre.   The curve for $D=853$ was discovered by Cremona and Watkins independently; the one for $D=929$ by Elkies; and the one for $D=997$ by Cremona. Their modularity can be established by applying the theorem of Skinner and Wiles~\cite[Theorem 5]{skinner-wiles1} to the associated $3$-adic representation.  
\end{exm}

\medskip

As a third and final application, we consider the existence of number fields with small ramification---these are linked to Hilbert modular forms via their associated Galois representations.

In the late 1990s, Gross proposed the following conjecture.

\begin{conj}\label{conj:unramified} 
For every prime $p$, there exists a nonsolvable Galois number field $K$ ramified only at $p$.
\end{conj}

Conjecture~\ref{conj:unramified} is true for $p\ge 11$ by an argument due to Serre and Swinnerton-Dyer \cite{serre3}.  For $k=12,16,18,20,22$ or $26$, let $\Delta_k \in S_k(1)$ be the unique newform of level 1 and weight $k$.  For each such $k$ and every prime $p$, there is an associated Galois representation $\overline{\rho}_{k,p}:\Gal(\overline{\Q}/\Q) \to \GL_2(\F_p)$ ramified only at $p$.  In particular, the field $K$ fixed by the kernel of $\rho_{k,p}$ is a number field ramified only at $p$ with Galois group $\Gal(K/\Q)=\img \overline{\rho}_{k,p}$.

We say that the prime $p$ is \emph{exceptional} for $k$ if $\rho_k$ is not surjective.  The ``open image'' theorem of Serre \cite{serre5} implies that the set of primes where $\overline{\rho}_k$ is exceptional is finite.  Serre and Swinnerton-Dyer produced the following table of exceptional primes:
\[ \begin{array}{c|c}
k&\mbox{\rm Exceptional $p$}\\\hline
12&2,3,5,7,23,691\\
16&2,3,5,7,11,31,59,3617\\
18&2,3,5,7,11,13,43867\\
20&2,3,5,7,11,13,283,617\\
22&2,3,5,7,13,17,131,593\\
26&2,3,5,7,11,17,19,657931
\end{array}
\]
(The large primes that occur are divisors of denominators of Bernoulli numbers.)  It follows that for every $p\ge 11$ that Conjecture \ref{conj:unramified} is true.  

Conversely, it is a consequence of the proof of Serre's conjecture by Khare and Wintenberger \cite{khare1} together with standard level lowering arguments that if $p \leq 7$ then any odd representation of the absolute Galois group $\Gal(\overline{\Q}/\Q)$ to $\GL_2(\overline{\F}_p)$ which is ramified only at $p$ is necessarily reducible (and thus solvable).  

To address Gross' conjecture, we instead look at the residual representations associated to Hilbert cusp forms of parallel weight $2$ and level $1$ over fields $F$ ramified only at a prime $p \leq 7$.  This idea was first pursued by the first author \cite{dembeleram2} following a suggestion of Gross, who found such a form with $F=\Q(\zeta_{32})^+$, the totally real subfield of the cyclotomic field $\Q(\zeta_{32})$, settling the case $p=2$.  This line of inquiry was followed further by the authors with Matthew Greenberg \cite{dgv}, settling the cases $p=3,5$.  (The case $p=7$ was recently settled by Dieulefait \cite{dieulefait} by considering the mod $7$ Galois representation attached to a genus $2$ Siegel cusp form of level $1$ and weight $28$.  Therefore Gross' conjecture is now a theorem.)  

We sketch below the resolution of the case $p=5$.  We take the base field $F=\Q(b)$ to be the subfield of $\Q(\zeta_{25})$ of degree $5$, where $b^5+5b^4-25b^2-25b-5=0$.  The field $F$ has strict class number $1$.  Let $E$ be the elliptic curve over $F$ with $j$-invariant given by
\begin{align*}
(5\cdot 7^{11}) j=& -16863524372777476b^4 - 88540369937983588b^3 + 11247914660553215b^2 \\
& \quad +   464399360515483572b + 353505866738383680
\end{align*}
and minimal conductor $\frakN$.  Then $\frakN=\gp_5\gp_7$, where $\gp_5=((-2b^3 - 12b^2 + 31b + 25)/7)$ is the unique prime above $5$ and $\gp_7=((-2b^4 - 9b^3 + 8b^2 + 53b + 6)/7)$ is one of the five primes above $7$.
Roberts \cite{roberts} showed that the mod $5$ Galois representation 
\[ \overline{\rho}_{E,5}:\Gal(\overline{F}/F) \to \End(E[5]) \cong \GL_2(\F_{\frakp_5})=\GL_2(\F_5) \] 
has projective image $\PGL_2(\F_5)$ and is ramified only at the prime $\gp_5$ (and not $\gp_7$).  (We will verify below that $E$ has indeed the right conductor.)  The representation $\overline{\rho}_{E,5}$ and its Galois conjugates gives an extension $K$ of $\Q$ with Galois group
\[ \Gal(K/\Q)\cong \PSL_2(\F_5)^5\rtimes \Z/10\Z. \]

By work of Skinner and Wiles (as in Example \ref{le30}), we prove that the $3$-adic representation $\rho_{E,3}$ associated to $E$ is modular, and hence $E$ itself is modular.  Since $E$ is modular and $F$ has odd degree, $E$ is uniformized by a Shimura curve.  Namely, let $B$ be the quaternion algebra over $F$ ramified at four of the five real places. Let $X_0^B(\gp_5\gp_7)$ be the Shimura curve associated to an Eichler order of level $\gp_5\gp_7$ contained in a maximal order $\CO$, and let $J_0^B(\gp_5\gp_7)$ be the Jacobian of $X_0^B(\gp_5\gp_7)$. We compute that $J_0^B(\gp_5\gp_7)^{\rm new}$ has dimension $203$. We then find the unique Hilbert newform $f_E$ of parallel weight $2$ and level $\gp_5\gp_7$ with integer Fourier coefficients which corresponds to $E$.  

The elliptic curve $E$ found by Roberts \cite{roberts} was obtained from our computations of Hilbert modular forms at level $\gp_5$ over $F$.  The space $S_2(\gp_5)^{\textup{new}}$ has 2 Hecke constituents of dimension 10 and 20, respectively. Let $f$ be a newform in the 20-dimensional constituent.  Let $\T_f$ be the restriction of the Hecke algebra $\T=\Z[T_\frakp]_{\frakp}$ to this constituent: this is the constituent which yields the Galois representation obtained in the 5-torsion of Roberts' curve.  Let $E_f=\T_f\otimes\Q=\Q(a_\gp(f))$ be the field of Fourier coefficients of $f$ and let $\Delta=\mathrm{Aut}(E_f)$. By direct calculations, we see that $K_f=E_f^{\Delta}$ is the (totally real) quartic field generated by a root of the polynomial $x^4 + 2x^3 - 75x^2 - 112x + 816$.  

The Galois group $\Gal(F/\Q)$ acts on $\T$ via its action on ideals of $\Z_F$, namely $\sigma(T_\frakp) = T_{\sigma(\frakp)}$.  This action preserves the decomposition of $\T$ into Hecke-irreducible components.  By work of Shimura, the action of $\Delta$ on Fourier expansions preserves the Hecke constituents. In particular, both these actions preserve $\T_f$ and hence $K_f$ and must be compatible.  Therefore, for each $\sigma\in\Gal(F/\Q)$, there is a unique $\tau=\tau(\sigma)\in\Delta$ such that, for all prime ideals $\gp\subset\Z_F$, we have
\[ a_{\sigma(\gp)}(f)=\tau(a_{\gp}(f)). \]
The map $\tau$ thus yields a homomorphism 
\begin{equation} \label{tausigma}
\begin{aligned}
\Gal(F/\Q) &\to \Delta \\
\sigma &\mapsto \tau(\sigma)
\end{aligned}
\end{equation}
By direct calculation, we show that this map is an isomorphism.  Since $\Gal(F/\Q)$ is abelian, the field $E_f$ must be a ray class field over $K_f$; in fact, we compute that it has conductor equal to a prime above 71 which splits completely in $E_f$.  The ideal $5\Z_{K_f}$ factors as $5\Z_{K_f}=\mathfrak{P}^2\mathfrak{P}'$. The prime $\frakP$ splits completely in $E_f$, and the primes above it in $E_f$ are permuted by $\Gal(E_f/K_f)=\Aut(E_f)$; the prime $\mathfrak{P}'$ is inert.  

We have a Galois representation
\[ \overline{\rho}_f:\Gal(\overline{F}/F) \to \GL_2(\T_f \otimes \F_5) \]
with $\Tr(\overline{\rho}_f(\Frob_\gp)) \equiv a_\gp(f) \pmod{5}$ and $\det(\overline{\rho}_f(\Frob_\gp)) \equiv N\gp \pmod{5}$ for all $\gp\nmid 5$.  Let $\mathfrak{m}_f^{(i)}$, $i=1,\ldots,5$ be the five maximal ideals above (the image of) $\gP$.
They give rise to the 5 residual representations 
\[ \overline{\rho}_f^{(i)}:\Gal(\overline{F}/F) \to \GL_2(\F_5). \]
Let $L_i$ be the fixed field of $\ker (\overline{\rho}_f^{(i)})$, and $L$ the compositum of the $L_i$.
Then $L$ is a Galois extension of $F$ and is ramified only at $5$.  Since $F$ is Galois and ramified only at $5$, and by the above $\Gal(F/\Q)$ permutes the fields $L_i$ so preserves $L$, we see that $L$ is a Galois extension of $\Q$ ramified only at $5$.  By a result of Shepherd-Barron and Taylor \cite{SBT}, each extension $L_i$ can be realized in the $5$-torsion of an elliptic curve $E_i/F$.  

Recall the projective representation $P\overline{\rho}_{E,5}$ from Roberts' elliptic curve $E$ is surjective, ramified at $\frakp_5$ but unramified at $\frakp_7$.  Therefore, the level $\mathfrak{p}_5\mathfrak{p}_7$ is a nonoptimal level for $\overline{\rho}_{E,5}$; thus, by Mazur's Principle~\cite{Jarvis}, we have $\overline{\rho}_{E,5}\cong\overline{\rho}_f^{(i)}$ for some $1\le i\le 5$.  In other words, the extension constructed by Roberts is isomorphic to our field $L$.

Roberts \cite{roberts} has given an explicit equation for the number field $L$ (obtained from the $5$-division polynomial of the elliptic curve $E$): the field $L$ is the the splitting field of the polynomial
\begin{align*}
&x^{25} - 25x^{22} + 25x^{21} + 110x^{20} - 625x^{19} + 1250x^{18} - 3625x^{17}\\
&\quad + 21750x^{16} - 57200x^{15} + 112500x^{14} - 240625x^{13} + 448125x^{12}\\
&\quad - 1126250x^{11} + 1744825x^{10} - 1006875x^9 - 705000x^8 + 4269125x^7\\
&\quad - 3551000x^6 + 949625x^5 - 792500x^4 + 1303750x^3 - 899750x^2 \\
&\quad + 291625x - 36535.
\end{align*}

\medskip

To conclude, we consider a question that touches on each of the above three subjects.  We reconsider Gross' conjecture (Conjecture \ref{conj:unramified}) in the case $p=2$ \cite{dembeleram2}.  The nonsolvable field which is ramified only at $2$ arises from the Galois representation associated to a constituent eigenform $f \in S_2(1)$ in a $16$-dimensional subspace of the space of Hilbert cusp forms of parallel weight $2$ and level $1$ over $F=\Q(\zeta_{32})^+$.  Let $E_f=\Q(a_\frakp)$ be the field of Fourier coefficients of $f$ and let $\Delta=\mathrm{Aut}(E_f)$. Let $K_f$ be the fixed field of $\Delta$ so that $\Gal(E_f/K_f)=\Delta$.  The map (\ref{tausigma}) in this context is again an isomorphism, and so $E_f$ is abelian over $K_f$ and $8=[F:\Q]$ divides $[E_f:\Q]=16$.  By direct calculations, we show that $K_f=\Q(\sqrt{5})$.

The Eichler-Shimura construction (Conjecture \ref{conj2}) predicts the existence a $16$-dimensional abelian variety $A_f$ defined over $F$ with everywhere good reduction and real multiplication by $E_f$ associated to $f$.  More should be true, as communicated to us by Gross (private communication).  In fact, $A_f$ should descend to an abelian variety of dimension $16$ over $\Q$, and we should have $L(A_f,s)=L(f,s)L(f^{\sigma},s)$, where $\sigma \in \Gal(E_f/\Q)$ is any element that restricts to the nontrivial element of $K_f=\Q(\sqrt{5})$.  The endomorphisms of $A_f$ over $\Q$ should be the ring of integers of $K_f$, and over $\Q$, the variety $A_f$ would have bad reduction only at the prime $2$; the nonsolvable extension would then arise as its $2$-division field.  The conductor $N=2^{124}=d^4$ of $A_f$ over $\Q$ can be computed from the functional equation of $L(f,s)$, where we note $d=2^{31}$ is the discriminant of $F$.  We note that $A_f$ is not of $\GL_2$-type over $\Q$ itself as it would be modular by the proof of Serre's conjecture.

Although one knows that the associated Galois representation exists by work of Taylor \cite{taylor1}, as Gross says, ``With such nice properties, it's a shame that we can't even prove that the abelian variety $A_f$ exists!  That's an advantage you have when $F$ has odd degree.''  

\section{Adelic quaternionic modular forms}\label{sec:adelic-quat-hmfs}

In this section, we begin again, and we revisit the definition of Hilbert and quaternionic modular forms allowing $F$ to have arbitrary class number; we refer to Hida \cite{Hidabook} as a reference for this section.

We renew our notation.  Let $F$ be a totally real field of degree $n=[F:\Q]$ with ring of integers $\Z_F$.  Let $B$ be a quaternion algebra over $F$ of discriminant $\frakD$.  Let $v_1,\dots,v_n$ be the real places of $F$ (abbreviating as before $x_i=v_i(x)$ for $x \in F$), and suppose that $B$ is split at $v_1,\dots,v_r$ and ramified at $v_{r+1},\dots,v_n$, i.e.\ 
\begin{equation} \label{BR}
B \hookrightarrow B_\infty = B \otimes_{\Q} \R \xrightarrow{\sim} \M_2(\R)^r \times \HH^{n-r}.
\end{equation}
Let $\iota_i$ denote the $i$th projection $B \to \M_2(\R)$ and $\iota=(\iota_1,\dots,\iota_r)$.  Let 
\[ F_+^\times=\{x \in F: x_i>0 \text{ for all $i$}\} \] 
be the group of totally positive elements of $F$ and let $\Z_{F,+}^\times = \Z_F^\times \cap F_+^\times$.  

Let $\calO_0(1) \subseteq B$ be a maximal order.  Let $\frakN$ be an ideal of $\Z_F$ coprime to $\frakD$ and let $\calO=\calO_0(\frakN) \subseteq \calO_0(1)$ be an Eichler order of level $\frakN$.  

With a view towards generalizations, rather than viewing modular forms as functions on (a Cartesian power of) the upper half-plane which transform in a certain way, we now view them instead more canonically as functions on $B_\infty^\times$.  Let $\calH^{\pm} = \C \setminus \R$ be the union of the upper and lower half-planes.  Via the embeddings $v_1,\dots,v_r$, corresponding to the first $r$ factors in (\ref{BR}), the group $B_\infty^\times$ acts on $(\calH^{\pm})^r$ on the right transitively with the stabilizer of $(\sqrt{-1},\dots,\sqrt{-1}) \in \calH^r$ being 
\[ K_\infty = (\R^\times \SO_2(\R))^r \times (\HH^\times)^{n-r}. \]
Therefore we can identify
\begin{equation} \label{identifyKoo}
\begin{aligned}
B_\infty^\times / K_\infty &\to (\calH^{\pm})^r\\
g &\mapsto z=g(\sqrt{-1},\dots,\sqrt{-1}).
\end{aligned}
\end{equation}

From this perspective, it is natural to consider the other (nonarchimedean) places of $F$ at the same time.  Let 
\[ \Zhat = \varprojlim_{n} \Z/n\Z=\prodprime_p \Z_p \] 
(where ${}'$ denotes the restricted direct product) and let $\widehat{\phantom{x}}$ denote tensor with $\Zhat$ over $\Z$. 
We will define modular forms on $B$ as analytic functions on $B_\infty^\times \times \Bhat^\times$ which are invariant on the left by $B^\times$ and transform by $K_\infty \times \calOhat^\times$ in a specified way.  

We must first define the codomain of these functions to obtain forms of arbitrary weight.  Let $k=(k_1,\dots,k_n) \in (\Z_{\geq 2})^n$ and suppose that the $k_i$ have the same parity; we call $k$ a \emph{weight}.  Let
\[ k_0=\max_i k_i, \quad m_i = (k_0-k_i)/2, \quad\text{and}\quad w_i=k_i-2. \]

For an integer $w\geq 0$, let $P_w=P_w(\C)$ be the subspace of $\C[x,y]$ consisting of homogeneous polynomials of degree $w$.  For $\gamma\in\GL_2(\C)$, let $\overline{\gamma}$ be the adjoint of $\gamma$, so that $\gamma\overline{\gamma}=\overline{\gamma}\gamma=\det\gamma$.  Define a right action of $\GL_2(\C)$ on $P_w(\C)$ by
\[
(q \cdot \gamma)(x,y) = 
q((x\,\,y)\bar{\gamma})=q(dx-cy,-bx+ay)
\]
for $\gamma=\left(\begin{matrix}a&b\\c&d\end{matrix}\right) \in \GL_2(\C)$ and $q \in P_w(\C)$.  For $m\in\Z$, $\GL_2(\C)$ also acts on $P_w(\C)$ via the character $\gamma\mapsto(\det\gamma)^m$. By twisting the above action by this character, we get a right $\GL_2(\C)$-module denoted by $P_w(m)(\C)$. Define the right $\GL_2(\C)^{n-r}$-module
\begin{equation} \label{WCC}
W_k(\C)=W(\C)= P_{w_{r+1}}(m_{r+1})(\C)\otimes \cdots\otimes P_{w_n}(m_n)(\C).
\end{equation} 
(By convention, if $r=n$ then we set $W_k(\C)=\C$.)  For the ramified real places $v_{r+1},\dots,v_n$ of $F$, we choose splittings
\[ \iota_i : B \hookrightarrow B \otimes_F \C \cong \M_2(\C). \]
We abbreviate as above $\gamma_i = \iota_i(\gamma)$ for $\gamma \in B$.  Then $W_k(\C)$ becomes a right $B^\times$-module via $\gamma \mapsto (\gamma_{r+1},\dots,\gamma_n) \in \GL_2(\C)^{n-r}$.  From now on, $W_k(\C)$ will be endowed with this action, which we denote by $x \mapsto x^\gamma$ for $x \in W_k(\C)$ and $\gamma \in B^\times$.  One may identify $W_k(\C)$ with the subspace of the algebra $\C[x_{r+1},y_{r+1},\ldots,x_n,y_n]$ consisting of those polynomials $q$ which are homogeneous in $(x_i,y_i)$ of degree $w_i$ but with a twisted action.

We consider the space of functions $\phi:B_\infty^\times\times \Bhat^\times \to W_k(\C)$, with a right action of $K_\infty^\times\times \Bhat^\times$ defined by
\begin{equation} \label{fslshadel}
(\phi \slsh{k} (\kappa,\betahat))(g, \alphahat)= \left(\prod_{i=1}^{r} \frac{j(\kappa_i, \sqrt{-1})^{k_i}}{(\det \kappa_i)^{m_i+k_i-1}}\right) \phi(g \kappa^{-1}, \alphahat \betahat^{-1})^{\kappa}.
\end{equation}
where recall $j(\kappa_i,\sqrt{-1})=c_i \sqrt{-1} + d_i \in \C$ if $\kappa_i = \begin{pmatrix} a_i & b_i \\ c_i & d_i \end{pmatrix}$.   (The presence of the inverses is forced as we want a right action on functions via multiplication of the argument on the right.  This almost extends to a right action of $B_\infty^\times\times \Bhat^\times$, except that $j(gh,z)=j(g,hz)j(h,z) \neq j(g,z) j(h,z)$ unless $h$ fixes $z$.)

\begin{defn}\label{def:quat-mod-forms}
A \emph{(quaternionic) modular form} of weight $k$ and level $\frakN$ for $B$ is an analytic function 
\[ \phi: B_\infty^\times \times \Bhat^\times \to W_k(\C) \] 
such that for all $(g,\alphahat) \in B_\infty^\times \times \Bhat^\times$ we have:
\begin{enumroman}
\item $(\phi\slsh{k}(\kappa,\uhat))(g, \alphahat)=\phi(g, \alphahat)$ for all $\kappa \in K_\infty$ and $\uhat \in \calOhat^\times$; and
\item $\phi(\gamma g,\gamma \alphahat) = \phi(g,\alphahat)$ for all $\gamma \in B^\times$.
\end{enumroman}
\end{defn} 

In other words, Definition~\ref{def:quat-mod-forms} says that a quaternionic modular form of weight $k$ and level $\frakN$ is an analytic function which is $B^\times$-invariant on the left and $(K_\infty\times \calOhat^\times)$-equivariant on the right under the action (\ref{fslshadel}).  In particular, we have 
\begin{equation} \label{indepofuhat}
\phi(g,\alphahat \uhat) = (\phi \slsh{k} (1,\uhat))(g,\alphahat \uhat) = \phi(g, \alphahat)
\end{equation}
for all $\uhat \in \calOhat^\times$ and
\begin{equation} \label{indepofkappa}
\phi(g\kappa,\alphahat) = (\phi \slsh{k} (\kappa,1))(g\kappa, \alphahat) = \left(\prod_{i=1}^{r} \frac{j(\kappa_i, \sqrt{-1})^{k_i}}{(\det \kappa_i)^{m_i+k_i-1}}\right) \phi(g, \alphahat)^{\kappa}.
\end{equation}

We denote by $M_k^B(\frakN)$ the space of quaternionic modular forms of weight $k$ and level $\frakN$ for $B$.

Let $\phi$ be a quaternionic modular form.  For $(z,\alphahat\calOhat^\times) \in (\calH^{\pm})^r\times \Bhat^\times/\calOhat^\times$, choose $g \in B_\infty^\times$ such that $g(\sqrt{-1},\dots,\sqrt{-1})=z$ and define
\begin{equation} \label{onupperH}
\begin{aligned}
f:(\calH^{\pm})^r \times \Bhat^\times/\calOhat^\times &\to W_k(\C) \\
f(z,\alphahat\calOhat^\times) &= \left(\prod_{i=1}^{r} \frac{(\det g_i)^{m_i+k_i-1}}{j(g_i, \sqrt{-1})^{k_i}}\right)\phi(g, \alphahat)^{g^{-1}}.
\end{aligned}
\end{equation}
The map $f$ in (\ref{onupperH}) is well-defined by (\ref{indepofuhat}) and (\ref{indepofkappa}): we have
\[ 
\left(\prod_{i=1}^{r} \frac{(\det g_i \kappa_i)^{m_i+k_i-1}}{j(g_i \kappa_i, \sqrt{-1})^{k_i}}\right)\phi(g \kappa, \alphahat \uhat)^{(g\kappa)^{-1}} =
\left(\prod_{i=1}^{r} \frac{(\det g_i)^{m_i+k_i-1}}{j(g_i, \sqrt{-1})^{k_i}}\right)\phi(g,\alphahat)^{g^{-1}} \]
using the fact that $j(g_i\kappa_i,\sqrt{-1})=j(g_i,\sqrt{-1})j(\kappa_i,\sqrt{-1})$.

The identity 7.5(ii) is translated as follows.  Let $\gamma \in B^\times$.  Then $g(\sqrt{-1}, \dots, \sqrt{-1}) = z$ is equivalent to $(\gamma g)(\sqrt{-1}, \dots, \sqrt{-1})=\gamma z$, so
\begin{equation} \label{fgammarelat}
\begin{aligned}
f(\gamma z, \gamma \alphahat \calOhat^\times) &= 
\left(\prod_{i=1}^{r} \frac{j(\gamma_i g_i, \sqrt{-1})^{k_i}}{(\det \gamma_i g_i)^{m_i+k_i-1}}\right)\phi(\gamma g, \gamma \alphahat)^{(\gamma g)^{-1}} 
\\
&= \left(\prod_{i=1}^{r} \frac{j(\gamma_i, z_i)^{k_i}}{(\det \gamma_i)^{m_i+k_i-1}}\right)
\left(\prod_{i=1}^{r} \frac{j(g_i, \sqrt{-1})^{k_i}}{(\det g_i)^{m_i+k_i-1}}\right)
\phi(g, \alphahat)^{g^{-1} \gamma^{-1}} \\
&= 
\left(\prod_{i=1}^{r} \frac{j(\gamma_i,z_i)^{k_i}}{(\det \gamma_i)^{m_i+k_i-1}}\right) f(z,\alphahat \calOhat^\times)^{\gamma^{-1}}
\end{aligned}
\end{equation}
where now $j(\gamma_i,z_i)=c_i z_i + d_i \in \C$ if $\gamma_i = \begin{pmatrix} a_i & b_i \\ c_i & d_i \end{pmatrix}$ for all $\gamma \in B^\times$ and we have the relation $j(\gamma \delta,z)=j(\gamma,\delta z)j(\delta,z)$ for all $z \in \calH$ and $g,h \in B^\times$.  Accordingly, we define a right action of $B^\times$ on the space of functions in~(\ref{onupperH}) by 
\begin{equation} \label{actonH}
(f\slsh{k}\gamma)(z, \alphahat \calOhat^\times)= \left(\prod_{i=1}^{r} \frac{(\det \gamma_i)^{m_i+k_i-1}}{j( \gamma_i, z_i)^{k_i}}\right) f(\gamma z, \gamma \alphahat \calOhat^\times)^{\gamma}.
\end{equation}
Then
\[ (f \slsh{k} \gamma)(z, \alphahat\calOhat^\times) = f(z,\alphahat \calOhat^\times). \]
Note that the central subgroup $F^\times \subseteq B^\times$ acts by $(f \slsh{k} a)(z, \alphahat \calOhat^\times) = N_{F/\Q}(a)^{k_0-2} f(z, a\alphahat \calOhat^\times)$ for $a\in F^\times$. 

The $\C$-vector space of modular forms of weight $k$ and level $\frakN$ on $B$ is finite-dimensional and is denoted $M_k^B(\frakN)$.  

\begin{lem} \label{modformourdef}
There is a bijection between $M_k^B(\frakN)$ and the space of functions $f:(\calH^{\pm})^r \times \Bhat^\times/\calOhat^\times \to W_k(\C)$ that are holomorphic in the first variable and locally constant in the second one and such that
\[f\slsh{k}\gamma=f \text{ for all $\gamma\in B^\times$}. \]
\end{lem}

From now on, we will only work with modular forms $f$ as presented in Lemma \ref{modformourdef}.  

We define the \emph{quaternionic Shimura variety of level $\frakN$} associated to $B$ as the double coset
\[X_0^B(\frakN)(\C)=B^\times \backslash (B_\infty^\times/K_\infty \times \Bhat^\times/\calOhat^\times) =
B^\times \backslash ((\calH^{\pm})^r \times \Bhat^\times/\calOhat^\times); \]
the set $X_0^B(\frakN)(\C)$ can be equipped the structure of a complex (possibly disconnected) Riemannian manifold of dimension $r$.

\begin{exm}
We recover first the definition of classical modular forms when $F=\Q$ and $B=\M_2(\Q)$.  For simplicity, we take $N=1$.  In this case, $r=1$ so $W_k(\C)=\C$ and $m=0$. 

The action (\ref{actonH}) is simply
\[ (f \slsh{k} \gamma)(z,\alphahat \calOhat^\times) = \frac{(\det \gamma)^{k-1}}{j(\gamma,z)^k} f(\gamma z,\gamma \alphahat \calOhat^\times). \]
We take the definition (\ref{onupperH}) as our starting point.  The element $\begin{pmatrix} -1 & 0 \\ 0 & 1 \end{pmatrix} \in B^\times = \GL_2(\Q)$ identifies the upper and lower half-planes, so a modular form $f:\calH^\pm \times \GL_2(\Qhat)/\GL_2(\Zhat) \to \C$ is determined by its restriction $f:\calH \times \GL_2(\Qhat)/\widehat{\Gamma}_0(N) \to \C$ and the subgroup of $\GL_2(\Q)$ which preserves $\calH$ is exactly $\GL_2^+(\Q)$.  

We wish to recover the classical action using this new action of $\GL_2^+(\Q)$, so we are led to consider the double coset $\GL_2^+(\Q)\backslash\!\GL_2(\Qhat)/\GL_2(\Zhat)$.  An element of this double coset is specified by an element $\alphahat\in \GL_2(\Qhat)$ up to right-multiplication by $\GL_2(\Zhat)$, i.e.\ a $\Zhat$-lattice in $\Qhat^2$, i.e.\ $\widehat{\Lambda} \in \Lat(\widehat{\Q}^2)$, specified by the rows of $\alphahat$.  But the map $\Lat(\Q^2) \to \Lat(\widehat{\Q}^2)$ by $\Lambda \mapsto \widehat{\Lambda}$ is a bijection, with inverse $\widehat{\Lambda} \mapsto \widehat{\Lambda} \cap \Q^2$.  And since $\GL_2^+(\Q)$ acts transitively on the left on the set of lattices $\Lat(\Q^2)$, we conclude that $\# \GL_2^+(\Q) \backslash\!\GL_2(\Qhat)/\GL_2(\Zhat) = 1$.  

It follows that $f$ is uniquely specified by the function $f(z,\calOhat^\times)$ for $z \in \calH$, which by abuse of notation we write simply $f:\calH \to \C$.  The stabilizer of $\GL_2^+(\Q)$ acting on $\GL_2(\Zhat)$ by multiplication on the left is $\GL_2^+(\Z)$, so we recover the condition
\[ (f \slsh{k} \gamma)(z) = \frac{(\det \gamma)^{k-1}}{j(\gamma,z)^k} f(\gamma z) \]
for all $\gamma \in \GL_2^+(\Z)$, which is exactly the definition given in Section 1.

The interested reader can modify this argument for $N>1$; alternatively, we give a general derivation below.
\end{exm}

We now define cusp forms.  If $B \cong \M_2(F)$, then we are in the situation of Hilbert modular forms (but over a field with arbitrary class number): so we define a \emph{cusp form} to be a form $f$ such that $f(z) \to 0$ whenever $z$ tends to a cusp $\PP^1(F) \hookrightarrow \PP^1(\R)^n$.  Otherwise, the Shimura variety $X_0^B(\frakN)(\C)$ is compact.  If $B$ is indefinite, so $0<r \leq n$, then there are no cusps, and we define the space of cusp forms to be $S_k^B(\frakN)=M_k^B(\frakN)$.  Finally, suppose $B$ is definite; then $r=0$.  If $k=(2,\dots,2)$, then we define $E_k^B(\frakN)$ to be the space of those $f \in M_k^B(\frakN)$ such that $f$ factors through $\nrd:\Bhat^\times \to \Fhat^\times$; otherwise, we set $E_k^B(\frakN)=0$.  Then there is an orthogonal decomposition $M_k^B(\frakN)=S_k^B(\frakN) \oplus E_k^B(\frakN)$ and we call $S_k^B(\frakN)$ the space of \emph{cusp forms} for $B$.

The spaces $M_k^B(\frakN)$ and $S_k^B(\frakN)$ come equipped with an action of pairwise commuting diagonalizable \emph{Hecke operators} $T_\frakn$ indexed by the nonzero ideals $\frakn$ of $\Z_F$, defined as follows.  
Given $f\in S_k^B(\frakN)$ and $\pihat \in \Bhat^\times$, we define a Hecke operator associated to $\pihat$ as follows: we write
\begin{equation} \label{doublecosetadel}
\calOhat^\times \pihat \calOhat^\times= \bigsqcup_i \calOhat^\times \pihat_i 
\end{equation}
and let 
\begin{equation} \label{Heckeindefiniteadel1}
(T_\pi f)(z, \alphahat \calOhat^\times)= \sum_{i} f(z, \alphahat \pihat_i^{-1} \calOhat^\times).
\end{equation}
(Again, although it may seem unnatural, the choice of inverse here is to make the definitions agree with the classical case.)

For a prime $\frakp$ of $\Z_F$ with $\frakp \nmid \frakD\frakN$, we denote by $T_\frakp$ the Hecke operator $T_{\pihat}$ where $\pihat \in \Bhat^\times$ is such that $\pihat_v=1$ for $v \neq \frakp$ and $\pihat_\frakp=\begin{pmatrix} p & 0 \\ 0 & 1 \end{pmatrix} \in \calO \otimes_{\Z_F} \Z_{F,\frakp} \cong \M_2(\Z_{F,\frakp})$ where $p \in \Z_{F,\frakp}$ is a uniformizer at $\frakp$.  

Equivalently, for a prime $\frakp$ and $\phat \in \ZFhat$ such that $\phat \ZFhat \cap \Z_F = \frakp$, we define 
\[ \Theta(\frakp)= \calOhat^\times  \backslash \{\pihat \in\calOhat : \nrd(\pihat)=\phat\}. \]
a set of cardinality $N\frakp+1$, and define
\begin{equation} \label{Heckeindefiniteadel2}
(T_\frakp f)(z,\alphahat \calOhat^\times) = \sum_{\pihat \in \Theta(\frakp)} f(z,\alphahat \pihat^{-1} \calOhat^\times)
\end{equation} 
where we have implicitly chosen representatives $\pihat \in \calOhat$ for the orbits in $\Theta(\frakp)$.  For an ideal $\frakn$ of $\Z_F$, the operator $T_\frakn$ is defined analogously. 

We say that a cusp form $f$ is a \emph{newform} if it is an eigenvector of the Hecke operators which does not belong to $M_k(\frakM)$ for $\frakM \mid \frakN$.

\medskip

To unpack this definition further, and to relate this definition with the definitions given previously, we investigate the structure of the Shimura variety 
\[ X_0^B(\frakN)(\C)=B^\times \backslash (B_\infty^\times/K_\infty \times \Bhat^\times/\calOhat^\times) = B^\times \backslash ((\calH^{\pm})^r \times \Bhat^\times / \calOhat^\times). \]

By Eichler's theorem of norms \cite[Theor\`eme III.4.1]{Vigneras}, we have 
\[ \nrd(B^\times)=F_{(+)}^\times = \{ a \in F^\times : v_i(a)>0\text{ for $i=r+1,\dots,n$}\}, \]
i.e.\ the norms from $B^\times$ consists of the subgroup of elements of $F$ which are positive at all real places which are ramified in $B$.  In particular, $B^\times/B_+^\times \cong (\Z/2\Z)^r$, where 
\[ B_+^\times=\{ \gamma \in B : \nrd(\gamma) \in F_+^\times\} \] 
is the subgroup of $B^\times$ whose elements have totally positive reduced norm.  

The group $B_+^\times$ acts on $\calH^r$, therefore we may identify
\[ X_0^B(\frakN)(\C) = B_+^\times \backslash (\calH^r \times \Bhat^\times / \calOhat^\times) \]
and a modular form on $(\calH^{\pm})^r \times \Bhat^\times/\calOhat^\times$ can be uniquely recovered from its restriction to $\calH^r \times \Bhat^\times/\calOhat^\times$.  Now we have a natural (continuous) projection map
\[ X_0^B(\frakN)(\C) \to B_+^\times \backslash \Bhat^\times / \calOhat^\times. \]
The reduced norm gives a surjective map
\begin{equation} \label{strongapprox}
\nrd:B_+^\times \backslash \Bhat^\times / \calOhat^\times \to F_+^\times \backslash \Fhat^\times / \ZFhat^\times \cong \Cl^+ \Z_F.
\end{equation}
where $\Cl^+ \Z_F$ denotes the strict class group of $\Z_F$, i.e.\ the ray class group of $\Z_F$ with modulus equal to the product of all real (infinite) places of $F$.  The theorem of strong approximation \cite[Th\'eor\`eme III.4.3]{Vigneras} implies that (\ref{strongapprox}) is a bijection if $B$ is indefinite.  So our description will accordingly depend on whether $B$ is indefinite or definite.  

\medskip

First, suppose that $B$ is indefinite.  Then space $X_0^B(\frakN)(\C)$ is the disjoint union of connected Riemannian manifolds indexed by $\Cl^+ \Z_F$, which we identify explicitly as follows.  Let the ideals $\fraka \subseteq \Z_F$ form a set of representatives for $\Cl^+ \Z_F$, and let $\ahat \in \ZFhat$ be such that $\ahat\,\ZFhat \cap \Z_F = \fraka$. (For the trivial class $\fraka=\Z_F$, we choose $\ahat=\widehat{1}$). By (\ref{strongapprox}), there exists $\alphahat \in \Bhat^\times$ such that $\nrd(\alphahat)=\ahat$. We let $\calO_{\fraka}=\alphahat \calOhat \alphahat^{-1} \cap B$ so that $\calO_{(1)}=\calO$, and we put $\Gamma_\fraka=\calO_{\fraka,+}^\times = \calOhat_{\fraka}^\times \cap B_+^\times$. Then we have 
\begin{equation} \label{breakup}
 X_0^B(\frakN)(\C) = \bigsqcup_{[\fraka] \in \Cl^+(\Z_F)} B_+^\times (\calH^r \times \alphahat \calOhat^\times) \xrightarrow{\sim} \bigsqcup_{[\fraka] \in \Cl^+(\Z_F)} \Gamma_\fraka \backslash \calH^r,
\end{equation}
where the last identification is obtained via the bijection
\begin{equation} \label{BplusHr}
\begin{aligned} 
B_+^\times \setminus (\calH^r \times \alphahat \calOhat^\times) &\xrightarrow{\sim} \Gamma_\fraka \backslash \calH^r \\
B_+^\times (z,\alphahat \calOhat^\times) &\mapsto z
\end{aligned}
\end{equation}

Now let $f \in M_k^B(\frakN)$, so that $f: (\calH^\pm)^r \times \Bhat^\times/\calOhat^\times \to W_k(\C)$ satisfies $f \slsh{k} \gamma = f$ for all $\gamma \in B^\times$.  Let $M_k^B(\frakN,\fraka)$ be the space of functions $f_\fraka:\calH^r \to W_k(\C)$ such that $f_\fraka \slsh{k} \gamma = f_\fraka$ for all $\gamma \in \Gamma_\fraka$, where we define
\[ (f_\fraka \slsh{} \gamma)(z) = \left( \prod_{i=1}^r \frac{(\det \gamma_i)^{m_i+k_i-1}}{j(\gamma_i, z)^{k_i}} \right)f_\fraka(\gamma z)^{\gamma} \]
for $\gamma \in B^\times$.  Then by (\ref{BplusHr}), the map
\begin{equation}\label{adelic-to-tuples}
\begin{aligned}
M_k^B(\frakN)&\to \bigoplus_{[\fraka] \in \Cl^+(\Z_F)} M_k^B(\frakN, \fraka)\\
f&\mapsto (f_\fraka)
\end{aligned}
\end{equation} 
where
\begin{align*}
f_\fraka : \calH^r &\to W_k(\C) \\
f_\fraka(z) & =f(z,\alphahat\calOhat^\times)
\end{align*}
is an isomorphism.

We now explain how the Hecke module structure on the left-hand side of (\ref{adelic-to-tuples}), defined in (\ref{Heckeindefiniteadel1})--(\ref{Heckeindefiniteadel2}), is carried over to the right-hand side. We follow Shimura~\cite[Section 2]{shimura1}.  We consider the action on the summand corresponding to $[\fraka] \in \Cl^+(\Z_F)$.  Extending the notation above, among the representatives chosen, let $\frakb$ be such that $[\frakb]=[\fraka\frakp^{-1}]$, let $\bhat\Zhat_F \cap \Z_F = \frakb$, and let $\betahat \in \Bhat^\times$ be such that $\nrd(\betahat)=\bhat$.  

By definition, 
\[ (T_\frakp f)_\fraka(z) = (T_\frakp f)(z,\alphahat \calOhat^\times) = \sum_{\pihat \in \Theta(\frakp)} f(z, \alphahat \pihat^{-1} \calOhat^\times ). \]
Let $\pihat \in \Theta(\frakp)$.  Then by strong approximation, we have 
\begin{equation} \label{apOb}
\alphahat \pihat^{-1} \calOhat \betahat^{-1} \cap B = \varpi^{-1} \calO_\frakb
\end{equation}
with $\varpi \in B^\times$, since this lattice has reduced norm $[\fraka \frakp^{-1} (\fraka \frakp^{-1})^{-1} ] = [(1)]$.  Therefore, there exists $\uhat \in \calOhat$ such that $\alphahat \pihat^{-1} \uhat \betahat^{-1} = \varpi^{-1}$ whence
\begin{equation} \label{heckejv}
(T_\frakp f)_\fraka(z) \sum_{\pihat \in \Theta(\frakp)} f(z, \alphahat \pihat^{-1} \calOhat^\times) = \sum_{\varpi} f(z, \varpi^{-1} \betahat\calOhat^\times);
\end{equation}
the second sum runs over a choice of $\varpi$ as in equation (\ref{apOb}) corresponding to each $\pihat \in \Theta(\frakp)$.  This latter sum can be identified with a sum over values of $f_\frakb$ as follows.  We have
\begin{align*} 
f(z, \varpi^{-1} \betahat \calOhat^{\times}) &= (f\slsh{k}\varpi)(z, \varpi^{-1} \betahat \calOhat^\times) \\
&= \left(\prod_{i=1}^{r} \frac{(\det \varpi_i)^{m_i+k_i-1}}{j( \varpi_i, z_i)^{k_i}}\right) f(\varpi z, \betahat \calOhat^\times)^{\varpi} = (f_\frakb \slsh{k} \varpi)(z).
\end{align*}
The first equality follows from the $B^\times$-invariance of $f$ and the others by definition of the slash operators.  Putting these together, we have
\begin{equation} \label{finalhecke!}
(T_\frakp f)_\fraka(z)  = \sum_{\varpi} (f_\frakb \slsh{k} \varpi)(z).
\end{equation}
(The naturality of this definition explains the choice of inverses above.)

This adelic calculation can be made global as follows.  Let $I_\fraka = \alphahat \calOhat \cap B$ and $I_\frakb = \betahat \calOhat \cap B$.  For $\pihat \in \Theta(\frakp)$, we have
\[ \alphahat \pihat^{-1} \calOhat \betahat^{-1} \cap B = \varpi^{-1} \calO_\frakb \]
hence
\begin{equation} \label{IaIbJa}
\calO_\frakb \varpi =  \betahat \calOhat \pihat \alphahat^{-1} = (\betahat \calOhat) \alphahat^{-1} (\alphahat \calOhat \alphahat^{-1}) \alphahat \pihat \alphahat^{-1} \cap B = I_\frakb I_\fraka^{-1} J .
\end{equation}
The elements $\varpi$ thus obtained are characterized by their norms (in the right lattice), as with the Hecke operators defined previously (\ref{heckethetapclno1}): we analogously define
\begin{equation} \label{thetapjv}
\begin{aligned} 
\Theta(\frakp)_{\fraka,\frakb} &= \Gamma_\frakb \setminus \{ \varpi \in I_\frakb I_\fraka^{-1} \cap B^\times_+ : 
\nrd(I_\frakb I_\fraka^{-1})\frakp= (\nrd(\varpi)) \}  \\
&= \Gamma_\frakb \setminus \{ \varpi \in I_\frakb I_\fraka^{-1} \cap B^\times_+ : \nrd(\varpi)\frakb=\fraka\frakp\}.
\end{aligned}
\end{equation}
Then for $f_\frakb \in M_k^B(\frakN,\frakb)$, we have $T_\frakp f_\frakb \in M_k^B(\frakN,\fraka)$ and
\[ T_\frakp f_\frakb = \sum_{\varpi \in \Theta(\frakp)_{\fraka,\frakb}} f_\frakb \slsh{k} \varpi \]
where $[\frakb]=[\fraka\frakp^{-1}]$.

\begin{exm}
If $F$ has strict class number $1$, then $I_\fraka=I_\frakb=\calO=\calO_a$ so $\calO_\frakb \varpi_a = \calO \pi_a$ as in Section 5.
\end{exm}

We note that the isomorphism~(\ref{adelic-to-tuples}) preserves the subspace of cusp forms in a way that is compatible with the Hecke action, so we have a decomposition 
\[ S_k^B(\frakN) \xrightarrow{\sim} \bigoplus_{[\fraka] \in \Cl^+(\Z_F)} S_k^B(\frakN, \fraka). \]

\begin{exm}
Let $B=\M_2(F)$, and let $\calO=\calO_0(\frakN) \subset \calO_0(1)=\M_2(\Z_F)$.  Then we may take $\alphahat=\begin{pmatrix} \ahat & 0 \\ 0 & 1 \end{pmatrix} \in \GL_2(\Fhat)$, and so we find simply that
\[ \calO_{\fraka} = \alphahat \M_2(\Z_F) \alphahat^{-1} \cap B = \begin{pmatrix} \Z_F & \fraka \\ \frakN\fraka^{-1} & \Z_F \end{pmatrix} = \calO_0(\frakN,\fraka) \]
Let
\[ \Gamma_0(\frakN,\fraka)=\calO_0(\frakN,\fraka)_{+}^\times = 
\left\{ \gamma = \begin{pmatrix} a & b \\ c & d \end{pmatrix} \in \calO_0(\frakN, \fraka): \det\gamma\in \Z_{F, +}^\times\right\}. \]
Then 
\[ X_0^B(\frakN)(\C) = \bigsqcup_{[\fraka] \in \Cl^+(\Z_F)} \Gamma_0(\frakN,\fraka) \backslash \calH^n \]
is a disjoint union.  

A \emph{Hilbert modular form of weight $k$ and level $\frakN$} is a tuple $(f_{\fraka})$ of holomorphic functions $f_{\fraka}: \calH^n \to \C$, indexed by $\Cl^+ \Z_F$, such that for all $\fraka$ we have 
\[ (f_\fraka \slsh{k} \gamma)(z)=f_{\fraka}(z)\text{ for all $\gamma \in \Gamma_0(\frakN,\fraka)$} \] 
(with the extra assumption that $f$ is holomorphic at the cusps if $F=\Q$).  Or, put another way, let $M_k(\frakN,\fraka)$ be the set of holomorphic functions $\calH^n \to \C$ such that $(f \slsh{k} \gamma)(z)=f(z)$ for all $\gamma \in \Gamma_0(\frakN,\fraka)$; then
\[ M_k(\frakN)= \bigoplus_{[\fraka]} M_k(\frakN,\fraka). \]
In particular, we recover the definitions in Section 2 when $F$ has strict class number $1$.

A modular form $f \in M_k(\frakN,\fraka)$ admits a Fourier expansion
\[ f(z)=a_0+\sum_{\mu\in (\ga\frakd^{-1})_+}a_\mu e^{2\pi i\mathrm{Tr}(\mu z)} \]
analogous to (\ref{qExpHilbert}).  We say that $f \in M_k(\frakN,\fraka)$ is a \emph{cusp form} if $f(z) \to 0$ as $z$ tends to any cusp.  Letting $S_k(\frakN,\fraka)$ be the space of such cusp forms, we have 
\[ S_k(\frakN)=\bigoplus_{[\fraka]} S_k(\frakN,\fraka). \]

Let $f=(f_\fraka) \in S_k(\frakN)$ be a Hilbert cusp form and let $\frakn \subseteq \Z_F$ be an ideal.  Suppose that $[\frakn]=[\fraka\frakd^{-1}]$ amongst the representatives chosen for $\Cl^+(\Z_F)$, and let $\nu \in \Z_F$ be such that $\frakn=\nu \fraka\frakd^{-1}$.  We define $a_\frakn=\nu^m a_\nu(f_{\fraka})$; the transformation rule implies that $a_\frakn$ only depends on $\frakn$ and we call $a_\frakn$ the \emph{Fourier coefficient} of $f$ at $\frakn$.
\end{exm}

\medskip

Now suppose that $B$ is definite.  Then the Shimura variety is simply 
\[ X_0^B(\frakN)(\C) = B^\times \backslash \Bhat^\times / \calOhat^\times = \Cl \calO \]
and so is canonically identified with the set of right ideal classes of $\calO$.  Note that the reduced norm map (\ref{strongapprox}) here is the map $\nrd:\Cl \calO \to \Cl^+ \Z_F$ which is surjective but not a bijection, in general.  
A modular form $f \in M_k^B(\frakN)$ is then just a map $f: \Bhat^\times/\calOhat^\times \to W_k(\C)$ such that $f\slsh{k}\gamma=f$ for all $\gamma\in B^\times$.  Such a function is completely determined by its values on a set of representatives of $\Cl \calO$; moreover, given any right ideal $I=\alphahat\calOhat \cap B$, the stabilizer of $B^\times$ acting on $\alphahat \calOhat$ by left multiplication is $\calO_L(I)^\times= \alphahat\calOhat^\times\alphahat^{-1}\cap B^\times$.
Therefore, there is an isomorphism of complex vector spaces given by
\begin{equation} \label{defarbclnoII}
\begin{aligned}
M_k^B(\frakN)&\to& \bigoplus_{\substack{[I] \in \Cl(\calO) \\ I=\alphahat\calOhat \cap B}} W_k(\C)^{\Gamma(I)}\\
f&\mapsto& (f(\alphahat)),
\end{aligned}
\end{equation}
where $\Gamma(I)=\alphahat\calOhat^\times\alphahat^{-1}\cap B^\times=\calO_L(I)^\times$ and $W_k(\C)^{\Gamma(I)}$ is the $\Gamma(I)$-invariant subspace of $W_k(\C)$.  

\medskip

Having now discussed both the definite and indefinite cases in turn, we return to a general quaternion algebra $B$.  Let $f \in S_k^B(\frakN)^{\textup{new}}$ be a newform.  A theorem of Shimura states that the coefficients $a_{\frakn}$ are algebraic integers and $E_f=\Q(\{a_\frakn\})$ is a number field.  The Hecke eigenvalues $a_\frakn$ determine the $L$-series
\[ L(f,s)=\sum_{\frakn \subseteq \Z_F} \frac{a_\frakn}{N\frakn^s} = \prod_{\frakp \nmid \frakN} \left( 1-\frac{a_\frakp}{N\frakp^s} + \frac{1}{N\frakp^{2s+1-k_0}}\right)^{-1} \prod_{\frakp \mid \frakN} \left(1-\frac{a_\frakp}{N\frakp^s}\right)^{-1} \]
associated to $f$ (defined for $\repart s>1$).  Moreover, associated to $f$ is a Galois representation: for $\frakl$ a prime of $\Z_{E_f}$ and $E_{f,\frakl}$ the completion of $E_f$ at $\frakl$, there is an absolutely irreducible, totally odd Galois representation
\[ \rho_{f,\frakl}:\Gal(\overline{F}/F)\to\GL_2(E_{f,\frakl}) \]
such that, for any prime $\frakp \nmid \frakl\frakN$, we have
\[ \Tr(\rho_{f,\frakl}(\Frob_{\frakp}))=a_\frakp(f) \quad \text{ and } \quad \det(\rho_{f,\frakl}(\Frob_\frakp))=N\frakp^{k_0-1}. \]
The existence of this representation is due to work of Blasius-Rogawski \cite{BlasiusRogawski}, Carayol \cite{Carayol}, Deligne \cite{Deligne2}, Saito \cite{Saito}, Taylor \cite{taylor1}, and Wiles \cite{wiles2}.

The statement of the Jacquet-Langlands correspondence (\ref{shimizuclno1}) reads the same in this more general context.

\begin{thm}[Jacquet-Langlands]\label{shimizuadel} 
There is an injective map of Hecke modules
\[ S_k^B(\frakN) \hookrightarrow S_k(\frakD\frakN) \]
whose image consists of those forms which are new at all primes dividing $\frakD$. 
\end{thm}

We are now ready to state the main general result of this article, generalizing the result of Theorem \ref{mainthmclno1} to arbitrary class number and arbitrary weight.

\begin{thm}[Demb\'el\'e-Donnelly \cite{DembeleDonnelly}, Voight \cite{Voightclno}]   
There exists an algorithm which,   given a totally real field $F$, a nonzero ideal $\frakN \subseteq  \Z_F$, and a weight $k \in (\Z_{\geq 2})^{[F:\Q]}$, computes the space $S_k(\frakN)$ of Hilbert cusp forms of level $\frakN$ over $F$ as a Hecke module.
\end{thm}

The proof of this theorem is discussed in the next two sections.  It falls again naturally into two methods, definite and indefinite, which overlap just as in Remark \ref{rmkgenspace}.

\section{Definite method, arbitrary class number}

In this section, we return to the totally definite case but allow arbitrary class number. As explained above, the space $M_k^B(\frakN)$ of modular forms of level $\frakN$ and weight $k$ on $B$ is the space of functions $f: \Bhat^\times/\calOhat^\times \to W_k(\C)$ such that $f\slsh{k}\gamma=f$ for all $\gamma\in B^\times$.  

We can use the identification (\ref{defarbclnoII}) to compute the space $S_k^B(\frakN)$ as in the direct approach of Section 4, with the appropriate modifications.  Let $I_1,\dots,I_H$ be a set of representatives for $\Cl \calO$ such that $\nrd(I_i)$ is coprime to $\frakD\frakN$ for all $i$.  Let $\alphahat_i \in \calOhat$ be such that $\alphahat_i \calOhat \cap \calO = I_i$, and let $\Gamma_i=\calO_L(I_i)^\times$. Then dualizing the isomorphism (\ref{defarbclnoII}), we have
\[ M_k^B(\frakN) \cong \bigoplus_{i=1}^{H} W_k(\C)^{\Gamma_i}. \]
The Hecke module structure on this space is defined similarly as in Section 4, as the following example illustrates.

\begin{exm} \label{defclnowoshapiro}
Consider the totally real quartic field $F=\Q(w)$ where $w^4 - 5w^2 - 2w + 1=0$.  Then $F$ has discriminant $5744=2^4 359$ and Galois group $S_4$.  We have $\Cl^+ \Z_F=2$ (but $\Cl \Z_F=1$).  

The quaternion algebra $B=\quat{-1,-1}{F}$ is unramified at all finite places (and ramified at all real places).  We compute a maximal order $\calO$ and find that $\# \Cl \calO=4$.  We compute the action of the Hecke operators as in (\ref{heckedefdefinite}): we identify the isomorphism classes of the $N\frakp+1$ right ideals of norm $\frakp$ inside each right ideal $I$ in a set of representatives for $\Cl \calO$.  
We compute, for example, that
\[ T_{(w^3 - 4w - 1)} = 
\begin{pmatrix}
0 & 0 & 1 & 1 \\
0 & 0 & 4 & 4 \\
2 & 2 & 0 & 0 \\
3 & 3 & 0 & 0 \\
\end{pmatrix} \]
where $N(w^3-4w-1)=4$; note this matrix has a block form, corresponding to the fact that $(w^3-4w-1)$ represents the nontrivial class in $\Cl^+ \Z_F$. Correspondingly,
\[ T_{(w^2-w-4)} = 
\begin{pmatrix}
6 & 2 & 0 & 0 \\
8 & 12 & 0 & 0 \\
0 & 0 & 8 & 4 \\
0 & 0 & 6 & 10 
\end{pmatrix} \]
with $N(w^2-w-4)=13$ is a block scalar matrix, as $(w^2-w-4)$ is trivial in $\Cl^+ \Z_F$.  In this case, the space $E_2(1)$ of functions that factor through the reduced norm has dimension $\dim E_2(1)=2$, so $\dim S_2(1)=2$, and we find that this space is irreducible as a Hecke module and so has a unique constituent $f$.

We obtain the following table of Hecke eigenvalues:
\[ 
\begin{array}{c||cccccccc}
\frakp & (w^3-4w-1) & (w-1) & (w^2-w-2) & (w^2-3) & (w^2-w-4) & (w^2-2) \\
N\frakp& 4 & 5 & 7 & 13 & 13 & 17 \\\hline
\rule{0pt}{2.5ex} 
a_\frakp(f)& 0 & t & -2t & -t & 4 & 3t  
\end{array} \]
Here $t$ satisfies the polynomial $t^2-6=0$.  As in Section 6, we predict the existence of an abelian variety over $F$ with real multiplication by $\Q(\sqrt{6})$ and everywhere good reduction.
\end{exm}

As in Section 4, the disadvantage of the approach used in Example \ref{defclnowoshapiro} is that for each level $\frakN$, one must compute the set of ideal classes $\Cl \calO=\Cl \calO_0(\frakN)$.  By working with a more complicated coefficient module we can work with ideal classes only with the maximal order $\calO_0(1)$, as follows.

Changing notation, now let $I_1,\dots,I_h$ be representatives for $\Cl \calO_0(1)$, with $h=\#\Cl \calO_0(1)$, and let $I_i=\alphahat_i\calO_0(1) \cap B$.  By strong approximation, we may assume that each $\nrd(I_i)$ is coprime to $\frakD\frakN$: indeed, we may assume each $\nrd(I_i)$ is supported in any set $S$ of primes that generate $\Cl^+ \Z_F$.  
Let $\calO_0(1)_i=\calO_L(I_i)=\alphahat_i \calOhat_0(1) \alphahat_i^{-1} \cap B$ be the left order of $I_i$.  Then $\calO_0(1)_i\otimes_{\Z_F}\Z_{F,\frakN} \cong \calO_0(1)\otimes_{\Z_F}\Z_{F,\frakN}$.

Let $\betahat_a$ for $a \in \PP^1(\Z_F/\frakN)$ represent the $\calO_0(1),\calO$-connecting ideals of norm $\frakN$: that is, $\betahat_a \in \calOhat_0(1)$ and if $J_a=\calOhat_0(1)\betahat_a \cap B$ then $\calO_R(J_a)=\calO$.  Then the set $\{I_i J_a\}_{i,a}$, where $I_i J_a = \alphahat_i \betahat_a \calOhat \cap B$, covers all isomorphism classes of right $\calO$-ideals, but not necessarily uniquely: two such ideals $I_i J_a$ and $I_j J_b$ are isomorphic if and only if $i=j$ and there exists $\gamma \in \calO_0(1)_i^\times$ such that $\gamma J_a = J_b$, comparing the elements $\alphahat_i \betahat_a, \alphahat_j \betahat_b \in \calOhat_0(1)$.  The action of $\calO_0(1)_i^\times$ can be equivalently given on the set of indices $a \in \PP^1(\Z_F/\frakN)$: via the (reduction of a) splitting map 
\begin{equation} \label{identactionO1N}
\iota_\frakN: \calO_0(1) \hookrightarrow \calO_0(1)\otimes_{\Z_F}\Z_{F,\frakN}\cong\M_2(\Z_{F,\frakN}),
\end{equation} 
each $\calO_0(1)_i^\times$ acts on the left on $\PP^1(\Z_F/\frakN)$, and we have $\calOhat_0(1)_i^\times/\calOhat_i^\times \xrightarrow{\sim} \PP^1(\Z_F/\frakN)$.  

We conclude that $M_k^B(\frakN) \cong \calM_k^B(\frakN)=\bigoplus_{i=1}^h \calM_k^B(\frakN)_i$, where
\[ \calM_k^B(\frakN)_i = \left\{f:\,\PP^1(\Z_F/\frakN) \to W_k(\C):\, f\slsh{k}\gamma=f\text{ for all }\gamma\in \calO_0(1)_i^\times\right\}. \]
In this presentation, the Hecke operators act as follows.  For a prime $\frakp$, let
\[ \Theta(\frakp)_{i,j}=\calO_0(1)_i^\times\backslash\left\{x\in I_iI_j^{-1}: \nrd(x I_i I_j^{-1})=\frakp\right\}. \]
We then define the linear map $T_\frakp:\calM_k^B(\frakN) \to \calM_k^B(\frakN)$ on each component by the rule
\begin{equation} \label{defheckeadel}
\begin{aligned}
(T_\frakp)_{i,j} : \calM_k^B(\frakN)_i &\to \calM_k^B(\frakN)_j \\
f&\mapsto \sum_{\gamma \in \Theta(\frakp)_{i,j}} f \slsh{k} \gamma.
\end{aligned}
\end{equation}
This is indeed an isomorphism of Hecke modules.  For further details, see work of the first author \cite[Theorem 2]{Dembeleclassno1} which traces these maps under the assumption that $F$ has narrow class number one, but this assumption can be easily removed. 

Put another way, by the decomposition
\[ \Bhat^\times=\bigsqcup_{i=1}^h B^\times \alphahat_i \calOhat_0(1)^\times, \] 
we decompose the set $X_0^B(\frakN)$ as 
\[X_0^B(\frakN)=B^\times\backslash\Bhat^\times/\calOhat^\times=\bigsqcup_{i=1}^h B^\times\backslash\left(B^\times \alphahat_i \calOhat_0(1)^\times\right)/\calOhat^\times \xrightarrow{\sim}  \bigsqcup_{i=1}^h\calO_0(1)_i^\times\backslash\calOhat_0(1)_i^\times/\calOhat_i^\times,\]
where the last identification is obtained by sending $\gamma \alphahat_i u$ to $\alphahat_i u \alphahat_i^{-1}$. Thus, analogously to~(\ref{breakup}), we get a decomposition
\begin{equation} \label{decompident}
X_0^B(\frakN) =\bigsqcup_{i=1}^h X_0^B(\frakN)_i=\bigsqcup_{i=1}^h\Gamma_i\backslash\PP^1(\Z_F/\frakN).
\end{equation}
In particular, this gives $X_0^B(1)= \Cl\calO_0(1)$. From this, we get an identification
\begin{equation}\label{def-adelic-tuples}
\begin{aligned}
M_k^B(\frakN) \to& \bigoplus_{i=1}^h \calM_k^B(\frakN)_i\\
f\mapsto& (f_i)_i,
\end{aligned}
\end{equation}
where we set $f_i(x)=f(\alphahat \alphahat_i)$ after choosing $\alphahat \in \calOhat_0(1)_i^\times$ such that $x=\alphahat\cdot\infty_i$. Again, the decomposition~(\ref{def-adelic-tuples}) is analogous to~(\ref{adelic-to-tuples}), and one shows that it is a Hecke module isomorphism by arguing similarly. 

Now $f \in M_k^B(\frakN)$ is by definition a map $f: \Bhat^\times/\calOhat^\times \to W_k(\C)$ such that $f \slsh{k} \gamma = f$ for all $\gamma \in B^\times$.  Associated to such a map, via the identification (\ref{decompident}), such a map is uniquely defined by a tuple of maps $(f_i)_i$ with $f_i: \calOhat_0(1)_i^\times / \calOhat_i^\times \to W_k(\C)$ such that $f \slsh{k} \gamma = f$ for all $\gamma \in \Gamma_i=\calO_0(1)_i^\times$.  In other words, 
\[ M_k^B(\frakN) \xrightarrow{\sim} 
\bigoplus_{i=1}^h H^0( \Gamma_i, \Hom(\calOhat_0(1)_i^\times / \calOhat_i^\times, W_k(\C))) \cong 
\bigoplus_{i=1}^h H^0\bigl( \Gamma_i, \Coind_{\calOhat_i^\times}^{\calOhat_0(1)_i^\times}W_k(\C)\bigr). \]
But then as in (\ref{identactionO1N}), we have
\[ H^0\bigl( \Gamma_i, \Coind_{\calOhat_i^\times}^{\calOhat_0(1)_i^\times}W_k(\C)\bigr) \cong \calM_k^B(\frakN)_i. \]

\begin{exm} 
The real quadratic field $F=\Q(\sqrt{106})$ has strict class number $2$ and class number $2$. We compute that the space $S_2(1)$ of Hilbert cusp forms of level 1 and parallel weight $2$ has dimension $50$. It decomposes into four Hecke constituents of dimension $1$,  six of dimension $2$, two of dimension $4$  and one of dimension $26$. The table below contains the first few Hecke eigenvalues of the one-dimensional constituents. 

\begin{eqnarray*}
\begin{array}{c|ccccccc}
\frakp& (2, w)&(3, w + 1)& (3, w-1)& (5, w + 1)& (5, w-1)& (3w+31) & (3w-31) \\
N\frakp& 2& 3& 3& 5& 5& 7& 7\\\hline
\rule{0pt}{2.5ex} 
a_\frakp(f_1)&-1& -2& 3& 3& -2& 4& -1\\
a_\frakp(f_2)& -1& 3& -2& -2& 3& -1& 4\\
a_\frakp(f_3)& 1& 2& -3& -3& 2& 4& -1\\
a_\frakp(f_4)& 1& -3& 2& 2& -3& -1& 4 
\end{array}
\end{eqnarray*}

The forms $f_1$ and $f_2$ (resp.\ $f_3$ and $f_4$) are interchanged by the action of $\Gal(F/\Q)$ (on the ideals $\frakp$).  The forms $f_1$ and $f_3$ (resp.\ $f_2$ and $f_4$) are interchanged by the action of $\Cl^+ \Z_F$, so these forms are twists via the strict class character of $\Gal(F^+/F)$, where $F^+$ denotes the strict class field of $F$. 

Elkies has found a curve $E$ which gives rise to the above data:
\begin{align*} 
E: y^2 - wxy - 2w y &= x^3 + (-2-2w)x^2 + (-10809936+1049944w) \\
&\qquad + (-19477819120 + 1891853024w). 
\end{align*}
The curve $E$ has $j$-invariant $j(E)=264235 + 25777w$ and has everywhere good reduction.  We conclude that $E$ is modular using Kisin~\cite[Theorem 2.2.18]{kisin1} (see also Kisin~\cite[Theorem 3.5.5]{kisin2}): we need to verify that $3$ is split in $F$, that $E$ has no CM, and that the representation $\rho_3: \Gal(\overline{F}/F) \to \GL_2(\Z_3)$ has surjective reduction $\overline{\rho}_3:\Gal(\overline{F}/F) \to \GL_2(\F_3)$ which is solvable hence modular.

We find that $E$ matches the form $f_1$; so its conjugate by $\Gal(F/\Q)$ corresponds to $f_2$ and the quadratic twist of $E$ by the fundamental unit $4005-389w$ (of norm $-1$) corresponds to $f_3$ (and its conjugate to $f_4$).
\end{exm}

The input of our algorithm is a totally real number field $F$ of degree $n$, a totally definite quaternion algebra $B$ with discriminant $\frakD$, an integral ideal $\frakN\subseteq \Z_F$ which is coprime with $\frakD$, a weight $k\in \Z^n$ such that $k_i\ge 2$ and $k_i\equiv k_j\pmod{2}$, and a prime $\frakp\nmid\frakD$. The output is then a matrix giving the action of $T_\frakp$ in a basis of $M_k^B(\frakN)=\bigoplus_{i=1}^h\calM_k^B(\frakN)_i$ which is independent of $\frakp$. By computing enough $T_\frakp$ and simultaneously diagonalising, one obtains all Hecke constituents corresponding to Hilbert newforms of level $\frakN$ and weight $k$.

The algorithm starts by finding a maximal order $\calO_0(1)$, then computes a set of representatives $\Cl \calO_0(1)$ of the right ideal classes of $\calO_0(1)$ whose norms generate $\Cl^+(\Z_F)$ and are supported outside $\frakD\frakN$. This part of the algorithm uses work of the second author and Kirschmer~\cite{KirschmerVoight}, it is the most time consuming part but can be seen as a precomputation. Next, the algorithm finds a fundamental domain for the action of each $\Gamma_i$ on $\PP^1(\Z_F/\frakN)$, and computes $M_k^B(\frakN)$ as the direct sum of the
\[\calM_k^B(\frakN)_i=\bigoplus_{[x]\in X_0^B(\frakN)_i} W_k(\C)^{\Gamma_x}, \]
where $\Gamma_x$ is the stabilizer of $x$ in $\Gamma_i$. From this, one obtains a basis of $M_k^B(\frakN)$. Finally, the algorithm computes the sets $\Theta(\frakp)_{i,j}$, and then the block matrices  which give the action of $T_\frakp$ in this basis. We refer to~\cite{Dembeleclassno1} and~\cite{DembeleDonnelly} for further details on the implementation.


\section{Indefinite method, arbitrary class number}

In this section, we generalize the indefinite method to arbitrary class number.  We carry over the notation from Section 7, and now take the quaternion algebra $B$ to be ramified at all but one real place.

In this case, from (\ref{breakup})--(\ref{BplusHr}), the space $X(\C)=X_0^B(\frakN)(\C) = B_+^\times \backslash (\calH \times \Bhat^\times / \calOhat^\times)$ is the disjoint union of Riemann surfaces indexed by $\Cl^+\Z_F$.  Let $\{\fraka\}$ be a set of representatives for $\Cl^+ \Z_F$ and let $\ahat \in \ZFhat$ be such that $\ahat\,\ZFhat \cap \Z_F = \fraka$ for each $\fraka$.  Then 
\begin{equation} \label{breakupshimcurve}
 X(\C)=\bigsqcup_{[\fraka] \in \Cl^+(\Z_F)} \Gamma_\fraka \backslash \calH = \bigsqcup_{[\fraka] \in \Cl^+(\Z_F)} X_\fraka(\C) \\
\end{equation}
where $\calO_{\fraka}=\alphahat \calOhat \alphahat^{-1} \cap B$ and $\Gamma_\fraka=\calO_{\fraka,+}^\times$.

Therefore, a modular form of weight $k$ and level $\frakN$ is a tuple $(f_\fraka)$ of functions $f_\fraka:\calH \to W_k(\C)$, indexed by $[\fraka] \in \Cl^+ \Z_F$, such that for all $\fraka$, we have 
\[ (f_\fraka \slsh{k} \gamma)(z) = f_\fraka(z) \] 
for all $\gamma \in \Gamma_\fraka$ and all $z \in \calH$.  In particular, if $k=(2,\dots,2)$ is parallel weight $2$, then $(f_\fraka)$ corresponds to a tuple of holomorphic $1$-forms $((2\pi i) f_\fraka(z)\,dz)_\fraka$, one for each curve $X_\fraka(\C)$.

We compute with this space of functions by relating them to cohomology, and for that we must modify the coefficient module.  Define the right $\GL_2(\C)^n=\GL_2(\C)\times \GL_2(\C)^{n-1}$-module
\[ V_k(\C)=\bigotimes_{i=1}^{n} P_{w_i}(m_i)(\C) =P_{w_1}(m_1)(\C)\otimes W_k(\C). \]
The group $B^\times$ acts on $V_k(\C)$ via the composite splitting $B^\times \hookrightarrow \GL_2(\C)^n$ given by $\gamma \mapsto (\gamma_i)_i$.  The Eichler-Shimura theorem, combined with the isomorphism (\ref{adelic-to-tuples}), applied to each component $X_\fraka(\C)$ of $X(\C)$ in (\ref{breakupshimcurve}), gives the isomorphism of Hecke modules
\begin{equation} \label{woahSk0}
S_k^B(\frakN) \xrightarrow{\sim} \bigoplus_{[\fraka]} H^1(\Gamma_{\fraka}, V_k(\C))^+,
\end{equation}
where ${}^+$ denotes the $+1$-eigenspace for complex conjugation.  

In the description (\ref{woahSk0}), the Hecke operators act on $\bigoplus H^1(\Gamma_\frakb, V_k(\C))$ in the following way; we follow their definition in (\ref{thetapjv}).  Let $\frakp$ be a prime ideal of $\Z_F$ with $\frakp \nmid \frakD\frakN$.  We consider the $[\frakb]$-summand, and given $f \in H^1(\Gamma_{\frakb},V_k(\C))$ we will define $T_\frakp f \in H^1(\Gamma_{\fraka}, V_k(\C))$, where $[\frakb]=[\frakp^{-1} \fraka]$.
Let $I_\fraka=\alphahat\calOhat\cap B$ and $I_{\frakb}=\betahat\calOhat\cap B$ so that $\nrd(I_\frakb)=\frakb$ and $\nrd(I_{\fraka})=\fraka$, and let
\begin{align*}
\Theta(\frakp)_{\fraka,\frakb} &= \Gamma_\frakb \backslash\left\{\varpi\in B_{+}^\times\cap I_{\frakb} I_{\fraka}^{-1}:
   \nrd(I_\frakb I_\fraka^{-1})\frakp= (\nrd(\varpi)) \right\}\\
&= \Gamma_\frakb \backslash\left\{\varpi \in B_{+}^\times\cap I_{\frakb}I_{\fraka}^{-1}:
   \,\, \nrd(\varpi)\fraka=\mathfrak{p}\frakb\right\},
\end{align*}
where $\Gamma_\frakb=\calO_{\frakb, +}^\times$ acts by multiplication on the left.  Let $\gamma\in\Gamma_{\fraka}$, so that $\gamma I_{\fraka}=I_{\fraka}$.  Then the map $\varpi \mapsto \varpi \gamma$ on $B^\times$ induces a bijection (of the equivalence classes) of $\Theta(\frakp)_{\fraka,\frakb}$.  Therefore, for every $\varpi \in\Theta(\frakp)_{\fraka,\frakb}$, there exists $\delta_\varpi \in\Gamma_{\fraka}$ and $\varpi_\gamma \in \Theta(\frakp)_{\fraka,\frakb}$ such that $\varpi \gamma=\delta_\varpi \varpi_\gamma$.  From (\ref{finalhecke!}) and the Eichler-Shimura theorem, we have
\begin{equation} \label{heckeindef9}
(T_\frakp f)(\gamma) = \sum_{\varpi \in\Theta(\frakp)_{\fraka,\frakb}}f(\delta_\varpi)^{\varpi}.
\end{equation}
One can similarly define the Atkin-Lehner involutions.

Admittedly, this description is complicated, but it can be summarized simply: a Hecke operator $T_\frakp$ permutes the summands (\ref{woahSk0}) in accordance with translation by $[\frakp]$ in $\Cl^+ \Z_F$, and adjusting for this factor one can principalize as before (when the strict class number was 1). The resulting Hecke matrices are consequently block matrices.  

We illustrate this with an example; we then give a few more details on the algorithm.

\begin{exm}
Let $F=\Q(w)$ be the (totally real) cubic field of prime discriminant $257$, with $w^3-w^2-4w+3=0$.  Then $F$ has Galois group $S_3$ and $\Z_F=\Z[w]$.  The field $F$ has class number $1$ but strict class number $2$: the unit $(w-1)(w-2)$ generates the group $\Z_{F,+}^\times/\Z_F^{\times 2}$ of totally positive units modulo squares.  

Let $B=\quat{-1,w-1}{F}$ be the quaternion algebra with $i^2=-1$, $j^2=w-1$, and $ji=-ij$.  Then $B$ has discriminant $\frakD=(1)$ and is ramified at two of the three real places and unramified at the place with $w \mapsto 2.19869\dots$, corresponding to $\iota_\infty:B \hookrightarrow \M_2(\R)$.  The order
\[ \calO=\Z_F \oplus (w^2+w-3)i \Z_F \oplus ((w^2+w)-8i+j)/2 \Z_F \oplus ((w^2+w-2)i+ij)/2 \Z_F \]
is an Eichler order of level $\frakN=(w)^2$ where $N(w)=3$.  

A fundamental domain for the action of $\Gamma=\iota_\infty(\calO_+^\times)$ on $\calH$ is as follows.

\begin{center}
\includegraphics{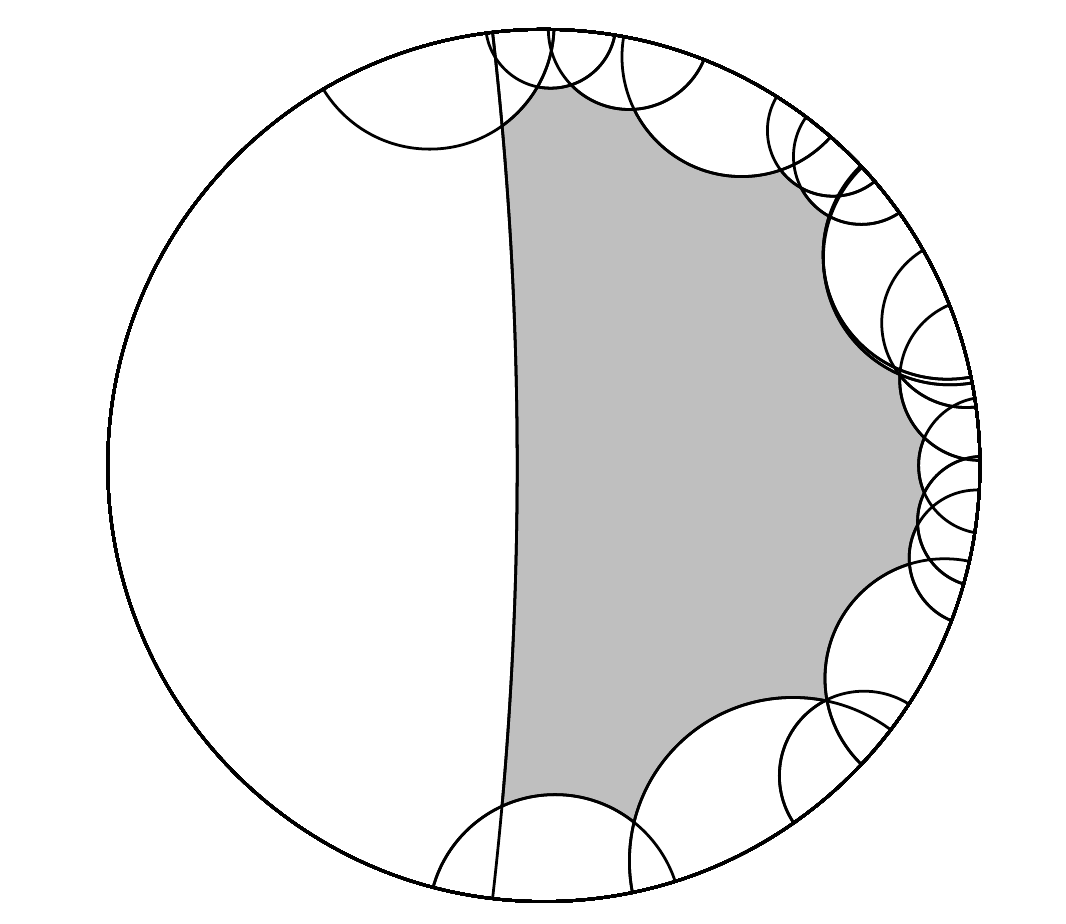}
\end{center}

The ideals $(1)$ and $\fraka=(w+1)\Z_F$ represent the classes in the strict class group $\Cl^+ \Z_F$.  The ideal $J_\fraka=2\calO + \left((w^2 + w + 2)/2 - 4i + (1/2)j\right)\calO$ has $\nrd(J_\fraka)=\fraka$.  The left order of $J_\fraka$ is $\calO_L(J_\fraka)=\calO_\fraka$ where
\begin{align*}
\calO_\fraka = \Z_F &\oplus (w^2-2w-3)i \Z_F \oplus \left((w^2 + w)/2 - 4i + (1/2)j\right)\Z_F \\
&\oplus (1/10)\left((174w^2 - 343w - 348)i + (w^2 - 2w - 2)j + (-w^2 + 2w + 2)ij\right)\Z_F.
\end{align*}
A fundamental domain for the action of $\Gamma_\fraka=\iota_\infty(\calO_{\fraka,+}^\times)$ on $\calH$ is as follows.

\begin{center}
\includegraphics{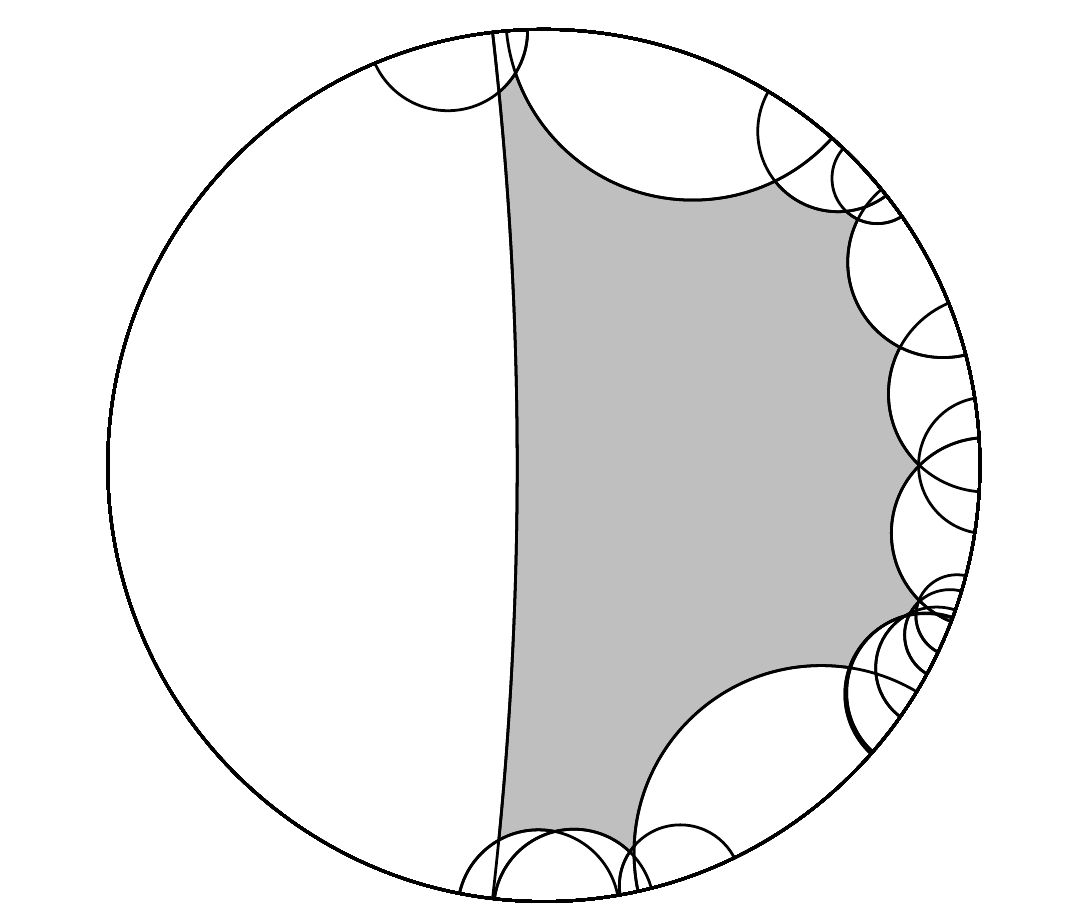}
\end{center}

The orders $\calO$ and $\calO_\fraka$ are not isomorphic since the connecting ideal $I_\fraka$ (with left order $\calO_\fraka$ and right order $\calO$) is not principal.  This implies that the groups $\Gamma$ and $\Gamma_\fraka$ are not conjugate as subgroups of $\PSL_2(\R)$ but nevertheless are isomorphic as abstract groups: they both have signature $(1;2,2,2,2)$, so that
\[ \Gamma \cong \Gamma_\fraka \cong \langle \gamma,\gamma',\delta_1,\dots,\delta_4 : \delta_1^2=\dots=\delta_4^2=[\gamma,\gamma']\delta_1\cdots \delta_4 = 1 \rangle. \]
In particular, both $X_{(1)}(\C)$ and $X_\fraka(\C)$ have genus $1$, so 
\[ \dim H^1(X(\C),\C)=\dim H^1(X_{(1)}(\C),\C) + \dim H^1(X_\fraka(\C),\C) = 4=2\dim S_2(\frakN). \]
We choose a basis of characteristic functions on $\gamma, \gamma'$ as a basis for $H^1(X_{(1)}(\C),\C)$ and similarly for $H^1(X_\fraka(\C),\C)$.

We now compute Hecke operators following the above.  Let $H=H^1(X(\C),\C)$.  We compute that complex conjugation acts on $H$ by the matrix 
\[ H \slsh{} W_\infty = \begin{pmatrix} -1 & -1 & 0 & 0 \\ 0 & 1 & 0 & 0 \\ 0 & 0 & -1 & 0 \\ 0 & 0 & 1 & 1 \end{pmatrix}. \]
Note that $W_\infty$ in this case preserves each factor.  Now consider the Hecke operator $T_\frakp$ where $\frakp=(2w-1)$ and $N(\frakp)=7$.  Then $\frakp$ represents the nontrivial class in $\Cl^+ \Z_F$.  We compute that
\[ H \slsh{} T_{\frakp} = \begin{pmatrix} 0 & 0 & -3 & 2 \\ 0 & 0 & -2 & -4 \\ -4 & -2 & 0 & 0 \\ 2 & -3 & 0 & 0 \end{pmatrix} \]
and restricting we get
\[ H^+ \slsh{} T_{\frakp} = \begin{pmatrix} 0 & -2 \\ -8 & 0 \end{pmatrix}. \]
Therefore there are two eigenspaces for $T_{\frakp}$ with eigenvalues $4,-4$.  By contrast, the Hecke operator $T_{(2)}$ acts by the scalar matrix $3$ on $H$, preserving each component.

Continuing in this way, we find the following table of eigenvalues:
\[
\begin{array}{c||ccccccccccccccc}
N\frakp & 3 & 7    & 8 & 9       & 19      & 25       & 37   & 41       & 43     & 47   & 49       & 53        &
          61 & 61 & 61 \\
\hline
\rule{0pt}{2.5ex} 
a_\frakp(f) & -1 & 4 & 3 & -4 & -4 & -8 & 4 & -6 & -8 & 0 & 4 & 12 & -8 & 2 & 4 \\
a_\frakp(g) & -1 & -4 & 3 & 4 & -4 & 8 & -4 & -6 & 8 & 0 & -4 & -12 & 8 & 2 & -4 \\
\end{array}  
\] 
Clearly, the form $g$ is the quadratic twist of the form $f$ by the nontrivial character of the strict class group $\Gal(F^+/F)$, where $F^+$ is the strict class field of $F$.  Note also that these forms do not arise from base change from $\Q$, since $a_\frakp$ has three different values for the primes $\frakp$ of norm $61$.

We are then led to search for elliptic curves of conductor $\frakN=(w)^2$, and we find two:
\begin{align*}
E_f: y^2 + (w^2 + 1)xy &= x^3 - x^2 + (-36w^2 + 51w - 18)x + (-158w^2 + 557w - 317) \\
E_g: y^2 + (w^2 + w + 1)xy + y &= x^3 + (w^2 - w - 1)x^2 + (4w^2 + 11w - 11)x + (4w^2 + w - 3)
\end{align*}
Each of these curves have nontrivial $\Z/2\Z$-torsion over $F$, so as above they are modular and we match Hecke eigenvalues to find that $E_f$ corresponds to $f$ and $E_g$ corresponds to $g$.  

In this situation, although by the theory of canonical models due to Deligne we know that the variety $X(\C)=X_{(1)}(\C) \sqcup X_{\fraka}(\C)$ has a model $X_F$ over $F$, the curves themselves are not defined over $F$---they are interchanged by the action of $\Gal(F^+/F)$.  Nevertheless, the Jacobian of $X_F$ is an abelian variety of dimension $2$ defined over $F$ which is isogenous to $E_f \times E_g$---we characterize in this way isogeny classes, not isomorphism classes.
\end{exm}

As in the case of class number $1$, the application of Shapiro's lemma allows us always to work with the group associated to a maximal order, as follows.  Let $\calO_0(1) \supseteq \calO$ be a maximal order containing $\calO$, and for each ideal $\fraka$, let $\calO_0(1)_{\fraka}=\alphahat \calOhat_0(1) \alphahat^{-1} \cap B$ be the maximal order containing $\calO_\fraka$, and let $\Gamma(1)_\fraka=\iota_\infty(\calO_{+}^\times)$.  Further, define
\[ V_k(\C)_\fraka=\Coind_{\Gamma_{\fraka}}^{\Gamma(1)_{\fraka}} V_k(\C) \]  
for each $\fraka$.  Then Shapiro's lemma implies that
\[ H^1(\Gamma_\fraka, V_k(\C)) \cong H^1(\Gamma(1)_\fraka, V_k(\C)_\fraka) \]
and so
\[ S_k^B(\frakN) \xrightarrow{\sim} \bigoplus_{[\fraka]} H^1(\Gamma(1)_{\fraka}, V_k(\C)_\fraka)^+. \]

Our algorithm takes as input a totally real field $F$ of degree $[F:\Q]=n$, a quaternion algebra $B$ over $F$ split at a unique real place, an ideal $\frakN \subset \Z_F$ coprime to the discriminant $\frakD$ of $B$, a vector $k \in (2\Z_{>0})^n$, and a prime $\frakp \nmid \frakD\frakN$, and outputs the matrix of the Hecke operator $T_\frakp$ acting on the space $H= \bigoplus_{\frakb} H^1\bigl(\Gamma(1)_{\frakb}, V_k(\C)_{\frakb}\bigr)^+$ with respect to some fixed basis which does not depend on $\frakp$.  From these matrices, we decompose the space $H$ into Hecke-irreducible subspaces by linear algebra.  We give a short overview of this algorithm.

First, some precomputation.  We precompute a set of representatives $[\fraka]$ for the strict class group $\Cl^+ \Z_F$ with each $\fraka$ coprime to $\frakp\frakD\frakN$.  For each representative ideal $\fraka$, precompute a right $\calO_0(1)$-ideal $I_{\fraka}$ such that $\nrd(I_{\fraka}) = \fraka$ and let $\calO_0(1)_\fraka=\calO_L(I_\fraka)$ be the left order of $I_{\fraka}$.  Next, we compute for each $\fraka$ a finite presentation for $\Gamma(1)_\fraka$ consisting of a (minimal) set of generators $G_\fraka$ and relations $R_\fraka$ together with a solution to the word problem for the computed presentation \cite{V-fd}.  Then using standard linear algebra techniques, we compute a basis for the space $\bigoplus_{[\fraka]} H^1(\Gamma(1)_{\fraka}, V_k(\C)_\fraka)$.  

The main issue then is to make the description (\ref{heckeindef9}) amenable to explicit computation.  First, compute a splitting $\iota_\frakp: \calO_0(1) \hookrightarrow \M_2(\Z_{F,\frakp})$.  Then for each ideal $\fraka$, perform the following steps. 

First, compute the ideal $\frakb$ with ideal class $[\frakb]=[\frakp^{-1}\fraka]$.  Compute the left ideals 
\[ J_a =\calO_a\iota_\frakp^{-1}\begin{pmatrix} x & y \\ 0 & 0 \end{pmatrix} + \calO_a\frakp \]
indexed by the elements $a=(x:y) \in \PP^1(\F_\frakp)$.  Then compute the left $\calO_{\frakb}$-ideals $I_{\frakb} I_\fraka^{-1} J_a$ and compute totally positive generators $\varpi_a \in \calO_{\fraka} \cap B_+^\times$ corresponding to $\calO_{\frakb} \varpi_a = I_{\frakb} I_\fraka^{-1} J_a$ \cite{KirschmerVoight}.

Now, for each $\gamma$ in a set of generators $G_\fraka$ for $\Gamma_\fraka$, compute the permutation $\gamma^*$ of $\PP^1(\F_\frakp)$ \cite[Algorithm 5.8]{GreenbergVoight} and then the elements $\delta_a=\varpi_a\gamma\varpi_{\gamma^*a}^{-1}$ for $a \in \PP^1(\F_\frakp)$; write each such element $\delta_a$ as a word in $G_\frakb$ and then apply the formula
\[ (T_\frakp f_\fraka)(\gamma) = \sum_{a\in\PP^1(\F_\frakp)}f_\frakb(\delta_a)^{\varpi_a}. \]

The algorithm in its full detail is rather complicated to describe; we refer the reader to work of the second author \cite{Voightclno} for the details.  

\end{document}